\documentclass[11pt]{amsart}

\usepackage[OT1]{fontenc}

\usepackage{amsmath,mathrsfs, amsfonts, amssymb, geometry, enumerate,
  caption, subcaption}

\usepackage[numbers]{natbib}

\usepackage{nameref}
\usepackage{zref-xr,zref-user}
\zxrsetup{toltxlabel}

\usepackage{tikz}
\usetikzlibrary{arrows, calc}
\usetikzlibrary{patterns, shapes, decorations.markings}

\usepackage[colorlinks,citecolor=blue,urlcolor=blue]{hyperref}
\usepackage{pgfplots}
\pgfplotsset{/pgf/number format/use comma,compat=newest}
\usetikzlibrary{external}
\numberwithin{equation}{section}

\numberwithin{equation}{section}
\theoremstyle{definition}
\newtheorem{definition}[equation]{Definition}
\theoremstyle{definition}
\newtheorem{remark}[equation]{Remark}
\theoremstyle{definition}

\theoremstyle{definition}

\theoremstyle{lemma}
\newtheorem{assumption}[equation]{Assumption}
\theoremstyle{lemma}
\newtheorem{lemma}[equation]{Lemma}
\theoremstyle{theorem}
\newtheorem{theorem}[equation]{Theorem}
\theoremstyle{proposition}
\newtheorem{proposition}[equation]{Proposition}
\theoremstyle{corollary}

\theoremstyle{corollary}
\theoremstyle{definition}
\newtheorem{example}[equation]{Example}
\theoremstyle{example}

\theoremstyle{proposition}

\theoremstyle{definition}

\newcommand{\one}{\mathbf{1}}

\newcommand{\R}{\mathbf{R}}
\newcommand{\PP}{\mathbf{P}}

\newcommand{\E}{\mathbf{E}}
\renewcommand{\P}{\mathbf{P}}

\newcommand{\T}{\mathbf{T}}
\newcommand{\N}{\mathbf{N}}
\newcommand{\Z}{\mathbf{Z}}
\newcommand{\ZZ}{\mathscr{Z}}

\newcommand{\bk}{\mathbf{k}}

\newcommand{\bj}{\mathbf{j}}

\newcommand{\bl}{\mathbf{l}}
\newcommand{\bm}{\mathbf{m}}
\newcommand{\bn}{\mathbf{n}}

\title{Phase space contraction of degenerately damped random splittings}
\author[D.P.~Herzog and J.C. Mattingly]{David P.~Herzog$^1$ and Jonathan C. Mattingly$^2$}

\address{$^1$ Department of Mathematics, Iowa State University, Ames, Iowa,
  USA}
\email{dherzog@iastate.edu}
\address{$^2$ Department of Mathematics, Duke University, Durham, North Carolina, USA}

\email{jonathan.mattingly@duke.edu}

\begin{document}
\maketitle

\begin{abstract}
When studying out-of-equilibrium systems, one often excites the dynamics
in some degrees of freedom while removing the excitation in others through damping. In order for the system to converge
 to a statistical steady state, the dynamics must transfer the
energy from the excited modes to the dissipative directions. The
 precise mechanisms underlying this transfer are of particular interest and are the
 topic of this paper. We explore a class of randomly switched models
 introduced in \cite{AMM_22,AMM_23} and provide some of the first
 results showing that minimal damping is sufficient to stabilize the system in a fluids model.

\end{abstract}

\section{Introduction and Motivation}

Many dynamical systems of scientific importance are expected to
be ergodic, where time averages of observables with respect to the system converge in the long-run to spatial averages against an
invariant measure.  We are interested in
settings where such statistical steady states are not in local equilibrium,
but rather have a nontrivial flux through the system. Particular examples include
chains of coupled oscillators with ends
connected to heat baths at different temperatures \cite{Rey-Bellet,Eckmann, HairerMattingly} and
Euler equations with forcing and  dissipation acting
on limited degrees of freedom \cite{BL_24, williamson19}. See also~\cite{CEHRB_18} for graphs of coupled oscillators with only some oscillators containing damping and random forcing. In such examples, energy is injected
in some directions in space and then transferred by the
conservative part of the dynamics to the dissapative degrees of freedom. This process typically produces an equilibrium steady state
with a flux from the forced degrees of freedom to the dissipated
ones. We will refer to such forced-dissipative systems
where the dissipation only acts on a subset of the degrees of freedom as
\textit{partially  dissipative}.

In partially dissipative systems, establishing the existence of a statistical steady state can be highly nontrivial. One of the main goals of this paper is to prove the
existence of such steady states under a randomly modified
dynamics, in settings such as finite-dimensional projections of the partially damped-driven Euler equations. 

To make the discussion more concrete, fix an ordinary differential equation (ODE)
\begin{align}\label{eq:V}
  \dot x = V(x)
\end{align}
on a Hilbert space $\mathbf{X}$ which is nonexplosive for all initial conditions $x_0 \in \mathbf{X}$ and  conserves a quantity $H\colon
\mathbf{X} \rightarrow [0, \infty)$. Note that conservation simply means that
$H(x_t) = H(x_0)$ for all $t \geq 0$. Often, the vector field $V$ is
the Hamiltonian flow associated with Hamiltonian $H$. To the dynamics~\eqref{eq:V}, one can then add dissipation
and forcing to arrive at the forced-damped system
\begin{align*}
  \dot x = V(x) -\Lambda x + F\,
\end{align*}
where $\Lambda:\mathbf{X}\rightarrow \mathbf{X}$ is symmetric, non-negative definite and $F:[0,\infty)\rightarrow \mathbf{X}$ is time-dependent.  
 When $\Lambda$ is furthermore assumed strictly positive-definite, in many scenarios it is 
not hard to show that that there exist constants $\alpha \in (0,1)$
and  $f\geq 0$ so that for some fixed time $t$
\begin{align}\label{eq:strongLyap}
  H( x_t ) \leq \alpha H(x_0) + f\,.
\end{align}
If the sublevel sets $\{ x : H (x) \leq c\}$, $c\geq 0$, are furthermore compact in $\mathbf{X}$, then
\eqref{eq:strongLyap} quickly implies the existence of a stationary
measure. In random systems with sufficient noise to ensure
that the stationary measure is unique, the bound
\eqref{eq:strongLyap}  in expectation is critical in proving that the system
converges to the equilibrium at a prescribed rate.  

In scenarios where $\Lambda$ has a nontrivial kernel, proving estimates like
\eqref{eq:strongLyap} can be difficult because it requires
detailed knowledge of the conservative dynamics~\eqref{eq:V}.  Even in this context, however, there are examples where the forcing allows the conservative part~\eqref{eq:V} to transport energy to the dissipative directions, keeping it from building up in the undamped degrees of freedom.   When this phenomenon 
happens at a sufficiently fast rate, one can ensure the existence of a
stationary state and its uniqueness if additional conditions are
satisfied.
We refer to~\cite{BBCL_24, BL_24, Cam_22, williamson19} for concrete examples in fluid-like models on Euclidean space and to~\cite{Rey-Bellet, CEHRB_18, Eckmann, HairerMattingly} for Langevin-like systems.  More closely related to the current paper is the very interesting work~\cite{BBCL_24} which appeared after this work was finished. It considers hypoelliptic SDEs with ``generic" Euler-type nonlinearities in the degenerately damped setting.  There, under the assumption that the dimension of the damped modes is at least one third of the dimension of the phase space, for a set of Euler type nonlinearities of full Lebesgue measure the resulting SDE has a unique stationary measure.  Note that this result does not identify which nonlinearities within this class have this property, but it is tantalizing because almost all of them have to have it. It is related to \cite{ williamson19} which also considered ``generic'' nonlinearities. In \cite{ williamson19} and most other fluid models~\cite{BL_24, Cam_22, williamson19}, a significant number of directions need damping. In contrast, one needs far less dissipative directions in the situation of oscillators provided the interaction potentials are strong enough relative to the pinning potentials~\cite{Rey-Bellet, CEHRB_18, Eckmann}.  However, when the interactions are weak, even in the case of a short chain of oscillators, the propagation of energy to dissipative modes can be highly nontrivial~\cite{HairerMattingly}.

In contexts similar to~\cite{BL_24, Cam_22, williamson19}, this paper will present some of the first results establishing the existence
of steady states where the dissipative directions are restricted to a small number
of degrees of freedom, even when the system's dimension is large. This will be done for \emph{randomly split} versions of equations such as Galerkin projections of two-dimensional Euler and the Lorenz '96 model.  Random splittings are simplified models which afford stronger results at the cost of disrupting the original dynamics.  However, they are insightful as they help build intuition about the full dynamics and the finer mechanisms involving the transport of energy to these isolated dissipative directions.  For example, our main results
establish the existence of statistical steady states for the randomly
split Galerkin 2D Euler equations and the randomly split Lorenz '96
models under the hypothesis that the dissipative operator is restricted to a finite
number of well-placed modes. Furthermore, the number of dissipative modes does not
grow with the dimension of the system, which is new in these
settings. 

% In \cite{williamson19}, conditions for generic Euler-like
%systems to posses an invariant measure with partial dissipation and
%forcing. In \cite{Liss}, the Galerkin approximations of 2D Euler were
%considered with a few degrees of freedom left without
%dissipation. The conditions in both examples required that a large
%fraction of the Fournier modes be dissipated. Neither were able to prove
%stability when only a finite, dimension independent number of Fournier
%modes have dissipation. 

Random splittings can be thought of as a ``randomly teased" version of the original dynamics at every time step, but in way that is natural to the dynamics. One can view one cycle of the splitting as ``generic in some sense'' evolution that is close to the true evolution. On each time step, we chose a different flow near by to the true flow.  On average, the resulting random dynamics is very close to the deterministic dynamics and all of the fluctuations are built from the base vector fields that define the
dynamics, but in a different mixture than the deterministic system. To build the random splittings, we choose
elements of the dynamics that are thought to be fundamental to the structure of the flow of the base dynamics. For example, for the 2D Euler equation, we use
the three-mode resonant interactions as our dynamical building
blocks which maintain many of the invariant
structures of the original system. This is in contrast to more traditional
stochastic agitation methods which use additive Brownian forcing.

Although some unlikely stochastic realizations are leveraged in the proof of the
existence of a stationary measure, we will see that the arguments leading to energy dissipation align well with the scenarios people often
invoke in heuristic explanations of why such systems with partial
dissipation should be stable. In particular, we will see how energy
transports via the building blocks, e.g. the three-mode resonant interactions in Euler, to the dissipative directions. Additionally, the
probabilities of the identified dissipative events scale
sufficiently well with the energy that the implied convergence rates would
be geometric for the Lorenz '96 system and stretched exponential for
the Euler system. The  stretched exponential convergence is not
completely surprising given that the fixed points of Euler and their
nearby slow dynamics persist in the random splitting.
It is also worth noting that we make no use of any additional
stochastic forcing. Consequently, we
do not deduce any convergence-to-equilibrium theorems.  Rather, we only
give existence of stationary measures.  In existing, related
works \cite{AMM_22,AMM_23,williamson19}, the uniqueness of the invariant measure for a
stochastically switched Euler system was shown within the class of
measures absolutely continuous with respect to Lebesgue measure on the constraint
surface. While these results could almost certainly be adapted to
this setting, they do not preclude the system converging to one of the
fixed points preserved by the random splitting. Random forcing is
important to show that the system leaves the neighborhood of the fixed
points in an efficient way. In future work, it would be interesting to combine random
switching and random forcing, though we do not explore that here.

\vspace{1ex}
\noindent \textbf{Paper Organization: }
In Section~\ref{sec:GRPD}, we
introduce some general results for proving the tightness needed
to guarantee the existence of an invariant measure for the system. We
also make some comments on the return times to compact sets that give
an indication what the rates of convergence to equilibrium might
be. In Section~\ref{sec:notation}, we introduce randomly split systems and the first of our two
examples; namely, the Lorenz '96 model. We demonstrate most of the ideas
introduced in this section by proving the needed assumptions for the
Lorenz '96 model. In Section~\ref{sec:lor}, we establish the required uniform
lower bound
on the probability of entering into a dissipative region in the
Lorenz '96 model. In Section~\ref{sec:NSE},  we turn to the Euler
equations and prove all of the needed estimates in that setting by
adapting the calculations used for the Lorenz '96. In many ways, however, the
Euler setting is significantly more complicated.

\vspace{1ex}
\noindent \textbf{Acknowledgments:} We graciously acknowledge support
from National Science Foundation grants DMS-2246491 and DMS-2038056.
D.P.H. and J.C.M acknowledges support from the Simons Foundation
through grants MP-TSM-00002755 and MPS-SFM-00006282, respectively.
The authors also acknowledge fruitful conversations with Andrea
Agazzi, Kyle Liss, Theo Drivas, and Omar Melikechi.  The authors are
grateful for useful feedback from the participants at the workshop
``Singularity and Prediction in Fluids'' during the summer of 2022
where an early version of these results was first presented.

\section{General Results for Partially Dissipative Systems}
\label{sec:GRPD}
In this section, $(\mathbf{X}, \rho)$ denotes a complete, separable metric space and $\mathcal{B}$ its associated Borel $\sigma$-algebra of subsets.  We let $(X_n)$ be a discrete-time Markov chain on the state space $\mathbf{X}$ with transition semigroup $(\mathcal{P}_n )$ defined over the probability space $(\Omega, \mathcal{F}, \PP, \E)$.  The operator $\mathcal{P}_n$ acts on bounded measurable functions $\phi: \mathbf{X}\rightarrow \mathbf{R}$ and finite $\mathcal{B}$ measures $\mu$ via 
\begin{align}
\label{def:semigroup}
(\mathcal{P}_n \phi)(x) := \E_x  \phi(X_n) \qquad \text{ and } \qquad (\mu \mathcal{P}_n)(B) := \int_\mathbf{X} \PP_x( X_n \in B) \mu(dx)  
\end{align}  
for $x\in \mathbf{X}$ and $B\in \mathcal{B}$.  In~\eqref{def:semigroup}, the notation $\E_x$ and $\PP_x$ means $\E$ and $\P$, but indicates that the  Markov chain has initial state $X_0=x$.  For any $n\geq 0$, $B\in \mathcal{B}$ and $x\in \mathbf{X}$, we let 
\begin{align}
\mathcal{P}_n(x, B)= (\mathcal{P}_n \one_B)(x) =\PP_x( X_n \in B)
\end{align}
denote the Markov transition probabilities.  Throughout, $\mathcal{P}_1$ will be denoted by $\mathcal{P}$.

The following elementary result is the beginning of our analysis.

\begin{proposition}\label{prop:Lyop}
  Let  $H\colon \mathbf{X} \rightarrow [1,\infty)$ and $p\colon
  [1,\infty) \rightarrow (0,1]$ be measurable, and let $\alpha \in  (0,1)$ and $f_1,f_2\geq 0$ be constants.  Suppose for any $x \in
  \mathbf{X}$ there exists an event $A=A(x) \in \mathcal{F}$  such that 
  $\PP_x(A) \geq p(H(x))$ and such that the following bounds hold
  \begin{align}
   \label{eqn:PDest} \mathbf{E}_x\big[ H(X_1) \,\one_{A}\big]
    &\leq  \alpha H(x)\mathbf{P}_x(A)+ f_1,\\
    \label{eqn:subconest} \mathbf{E}_x \big[ H(X_1) \, \one_{A^c}\big]  &
                                                                       \leq
                                                                       H(x) \mathbf{P}_x(A^c)
                                                                       + f_2                                                          
  \end{align}
  for every $x\in \mathbf{X}$.  
  Then the global estimate 
  \begin{align}
  \label{eqn:contract}
     \mathcal{P} H -H & \leq - (1-\alpha) p(H)\, H\,
                                                + f_1+f_2
  \end{align}
  holds on $\mathbf{X}$.
\end{proposition}

\begin{proof}  Observe for any $x\in \mathbf{X}$ we have  
  \begin{align*}
     \mathcal{P} H(x) &=  \mathbf{E}_x
    \big[ H\big(X_1\big) \,\one_{A}\big] 
    +   \mathbf{E}_x
    \big[ H\big(X_1\big) \,\one_{A^c}\big]    \\
    & \leq  [ \alpha  \P_x(A) + (1-   \P_x(A) ) ]
      H(x) + f_1  +f_2\\&\leq  [ 1- \P_x(A) (1-\alpha) ]
      H(x) + f_1+f_2 \\
    & \leq [ 1 - (1-\alpha) p(H(x))]   H(x) + f_1+f_2 .
  \end{align*}
  This concludes the proof. 
\end{proof}

The next lemma shows that the conclusion
of Proposition~\ref{prop:Lyop} guarantees the existence of (at least one)
stationary distribution for the Markov chain provided $x\mapsto p(H(x))H(x) $  has compact sublevel sets and the semigroup $(\mathcal{P}_n)$ is \emph{Feller}; that is, $x\mapsto \mathcal{P}\phi(x)$ is continuous whenever $\phi:\mathbf{X}\rightarrow \R$ is bounded, continuous.

\begin{lemma}
\label{lem:KBmeas} Assume $(\mathcal{P}_n)$ is Feller.  Suppose there exist measurable $H\colon \mathbf{X}\rightarrow [1,\infty)$ and $G\colon
 [1,\infty)\rightarrow (0,\infty)$ such that $\{ x\in \mathbf{X}\, : \, G(H(x)) \leq \alpha \}$ is compact for every $\alpha > 0$, and for some constant $f\geq 0$  we have the following estimate on $\mathbf{X}$:
 \begin{align}\label{eq:LyapGen}
   \mathcal{P} H- H \leq -  G(H) + f.
 \end{align}
 Then the Markov chain $(X_n)$ has at least one stationary distribution.
\end{lemma}
\begin{remark}
The result above remains true if we replace~\eqref{eq:LyapGen} with the perhaps more familiar  bound
\begin{align}
\label{eqn:KBbound}
\mathcal{P}H - H \leq -\widetilde{H}+f
\end{align}
where $\widetilde{H}: \mathbf{X}\rightarrow [1, \infty)$ is any measurable function such that $\{ x\in \mathbf{X} \, : \, \widetilde{H}(x) \leq R \}$ is compact for every $R\geq R_*>0$, for some $R_*>0$.  In other words, the righthand side of~\eqref{eq:LyapGen} need not be a function of $H$, but we will use this specific form several times below.  Weaker bounds than~\eqref{eqn:KBbound} are sometimes used to show the existence of a stationary distribution, but they usually follow the Hasminskii cycle construction which is different than the Krylov-Bogolyubov existence method used below.  See, for example,~\cite{RB_06}.    
\end{remark}
\begin{proof}[Proof of Lemma~\ref{lem:KBmeas}]
 Because this result is essentially the well-known Krylov-Bogolyubov
  theorem for Feller Markov semigroups, we sketch the proof and
  refer the reader to \cite{DaP_06, Khas_12, MT_12} for further details.
Fix $x\in \mathbf{X}$ and define the sequence of $\mathcal{B}$ probability measures $(\mu_n)$ by
  \begin{align*}
    \mu_n(B) = \frac1n \sum_{\ell=0}^{n-1} \PP_x( X_{\ell } \in B) , \,\, n\geq 1, \, B\in \mathcal{B}.
  \end{align*}
  Using~\eqref{eq:LyapGen} and the Markov property we find that
  \begin{align*}
    \sum_{\ell=0}^{n-1} \mathbf{E}_x G(H(X_{\ell })) \leq H(x) - \mathcal{P}_{n} H(x)
    + n f \leq  H(x)   + n f \,.
  \end{align*}
  Hence for any $R>0$,
  \begin{align*}
    \mu_n ( x : G(H(x)) > R) =  \frac1n \sum_{\ell=0}^{n-1}  \mathbf{P}_x(
    G(H(X_{\ell })) >R) \leq  \frac1n \sum_{\ell=0}^{n-1}  \frac{\mathbf{E}_x
    G(H(X_{\ell }))}{R} \leq \frac{H(x) + f}{R},
  \end{align*}
  which shows that the family of measures $( \mu_n)$ is tight since $\{ x\, : \, G(H(x)) \leq R\}$ is compact by hypothesis. Let
  $\mu$ be a weak limit point of this family and, when they exist, let $\mu_n(\phi), \mu_n \mathcal{P}(\phi)$ and $\mu(\phi)$ denote the expectations of $\mathcal{B}$-measurable $\phi:\mathbf{X}\rightarrow \R$ with respect to $\mu_n$, $\mu_n \mathcal{P}$ and $\mu$.  To see that
  $\mu$ is stationary, observe that for any bounded, $\mathcal{B}$-measurable $\phi: \mathbf{X}\rightarrow \R$ we have by definition of $\mu_n$
  \begin{align*}
  | \mu_n \mathcal{P}(\phi) - \mu_n(\phi) |= \frac1n \big| \mathbf{E}_x\phi( X_{n} ) 
    - \phi(x)\big| \rightarrow 0
  \end{align*}
  as $n\rightarrow \infty$. 
   Hence if $(\mu_{n_k})$ is a subsequence converging weakly to $\mu$, then the Fubini-Tonelli theorem, the Feller property and weak convergence implies
  \begin{align*}
 \mu\mathcal{P}(\phi) = \mu (\mathcal{P}\phi)= \lim_{k\rightarrow \infty}  \mu_{n_k}(\mathcal{P}\phi) = \lim_{k\rightarrow \infty} (\mu_{n_k}\mathcal{P}(\phi) - \mu_{n_k}(\phi)) +  \lim_{k\rightarrow \infty} \mu_{n_k} (\phi) =\mu(\phi).
  \end{align*}
Thus, $\mu \mathcal{P}(\phi)  = \mu(\phi)$ for all bounded, continuous $\phi: \mathbf{X}\rightarrow \R$, finishing the proof.

\end{proof}

\begin{remark}
 Suppose that~\eqref{eq:LyapGen} is satisfied and $G(t)\rightarrow \infty$ as $t\rightarrow \infty$.  Then for any $\epsilon \in (0,1)$ there exists $R>0$ large enough so that the global bound holds
  \begin{align*}
     \mathcal{P} H - H \leq -
    (1-\epsilon) G(H) + (1+\epsilon) f\one_{\{H \leq R\}}.  
  \end{align*}
  Thus under this hypothesis on $G$, the bound~\eqref{eq:LyapGen} is equivalent to the bound
  \begin{align}\label{eq:LyapGen2}
    \mathcal{P} H - H \leq - G(H) +  f\one_{\{H \leq R\}}
      \end{align}
     provided the terms $G$ and $f$ on the righthand side of~\eqref{eq:LyapGen2} are redefined accordingly.  
\end{remark}

\subsection{Return Times and Convergence}
Let $R, f\geq 0$ and suppose that~\eqref{eq:LyapGen2} is satisfied for some $H: \mathbf{X}\rightarrow [1, \infty)$ measurable and some concave nondecreasing differentiable function $G:[1, \infty)\rightarrow (0, \infty)$.  We next consider implications of this estimate for return times of the Markov chain $( X_n )$ to the set\begin{align}
H_{\leq R}=\{ x\in \mathbf{X} \, : \, H(x) \leq R\}. 
\end{align}    
Let
\begin{align}
T_{R}= \min\{ n \geq 0 \, : \, H(X_n) \leq R\}
\end{align}
and note that in the special case when $G(t)=\alpha t$ for some constant $\alpha>0$, relation~\eqref{eq:LyapGen2} implies 
\begin{align}
\label{eqn:expmom}
\E_x e^{\alpha T_R}  <\infty \,\, \text{ for all } \,\, x\in \mathbf{X}. 
\end{align}
This fact is well-known.  However, when $G$ is a general concave, nondecreasing function, the corresponding form of the moment bound changes.        
 
To describe how the estimates change in this more general context, define $K:[1, \infty) \rightarrow [0, \infty)$ by 
\begin{align}
\label{def:K}
  K(t)=\int_1^t \frac{1}{G(s)} \, ds.
\end{align}
Since $G$ is assumed concave, we have  that 
\begin{align*}
G(t)\leq G(1) + G'(1)(t-1).  
\end{align*}
This fact, along with the condition that $G$ is nondecreasing and
positive, implies that $K$ is strictly increasing and $K(t)
\rightarrow \infty$ as $t\rightarrow \infty$.  In particular, $K$ has
inverse function $K^{-1}:[0, \infty) \rightarrow [1, \infty)$ with
$K^{-1}(0)=1$ which is also strictly increasing.  Furthermore, $K$ itself is concave as $G$ is nondecreasing.  Define $r:[0, \infty)\rightarrow (0, \infty)$ by
\begin{align}
\label{def:r}
&r(t) := (K^{-1})'(t)= G(K^{-1}(t)).
\end{align}
It is worth noting that the properties of $r$ and $K^{-1}$ imply that
$K^{-1}(n)$, $n\in \Z_{\geq 1}$, is comparable to $\sum_{k=1}^n r(k)$.  In particular, it holds that 
\begin{align}
\label{eqn:Krcomp}
K^{-1}(n)-1 = \int_0^n r(s) \, ds  \leq \sum_{k=1}^n r(k).  
\end{align}

The following result also holds:
\begin{theorem}
\label{thm:moments}
Let $R, f \geq 0$ and suppose that the estimate~\eqref{eq:LyapGen2} holds for some  measurable $H:\mathbf{X}\rightarrow [1, \infty)$ and some concave nondecreasing differentiable function $G:[1, \infty)\rightarrow (0, \infty)$.  Then for all $x\notin H_{\leq R}$
\begin{align}
\label{eqn:timemom}
 \E_x \bigg[\sum_{k=0}^{T_{R} -1} r(k)\bigg] \leq H(x) 
\end{align}
and for all $x\notin H_{\leq R}$, $n\geq 0$
\begin{align}
\label{eqn:tails}
  \P_x( T_R \geq n+1) \leq  \frac{H(x)+1 }{ K^{-1}(n)}.   
\end{align}
\end{theorem}

The bound~\eqref{eqn:timemom} follows directly from~\cite{RGMP_04, TT_94} while relation~\eqref{eqn:tails} follows from~\eqref{eqn:timemom} using~\eqref{eqn:Krcomp}.  Furthermore, if in addition to the conclusions of Theorem~\ref{thm:moments} one has a minorization condition of the form
\begin{align*}
\inf_{x\in H_{\leq R'}} \mathcal{P}(x,\;\cdot\;) \geq c \lambda(\;\cdot\;)
\end{align*}
for some constant $c>0$, $R'>2R$ and $\mathcal{B}$-probability measure $\lambda$, then one can prove that (see~\cite{RGMP_04, TT_94})
$\mathcal{P}_n(x,\;\cdot\;)$ converges in total variation to the unique stationary distribution of the chain as $n\rightarrow \infty$ at a rate proportional to $1/r(n)$. In Example~\ref{ex:normal} and Example~\ref{ex:stretch} below, this will in turn correspond to exponential, polynomial and stretched exponential rates of convergence to equilibrium, depending on the structure of $G$.  This fits with the intuition that one mechanism slowing
convergence to a subexponential rate is slow return times to the
minorizing set $H_{\leq R}$.  See also~\cite{DFG_09, Hairer_10} for subgeometric rates of convergence for Markov processes in the continuous-time setting.

\begin{example}
\label{ex:normal}
Suppose that $G(t)=\alpha t^a$ for some $\alpha>0$ and $a \in (0,1]$.  Thus $G$ is strictly increasing and concave on $[1, \infty)$ and $K$, $K^{-1}$ are given by \begin{align*}
K(t) = \begin{cases}
\frac{1}{\alpha} \log(t) & \text{ if }a =1\\
\frac{ t^{1-a}-1}{\alpha(1-a)} & \text{ if } a \in (0, 1) 
\end{cases} \quad \text{ and } \quad K^{-1}(t) = \begin{cases}
e^{\alpha t} & \text{ if } a =1\\
(\alpha (1-a)t +1)^{1/(1-a)} & \text{ if } a \in (0, 1) 
\end{cases} .
\end{align*}    
Furthermore, 
\begin{align*}
r(t) = \begin{cases}
\alpha e^{\alpha t} & \text{ if } \alpha =1\\
\alpha(\alpha(1-a)t +1 )^{\frac{a}{1-a}}& \text{ if } a \in (0,1)
\end{cases}.
\end{align*}
In particular, in the case when $a=1$, the bound~\eqref{eqn:timemom} recovers~\eqref{eqn:expmom}. 
\end{example}

The next example will be particularly important in this paper. 
\begin{example}
\label{ex:stretch}
Let $\alpha>0$ and $G_0:(1, \infty) \rightarrow (0, \infty)$ be given by 
\begin{align*}
G_0(t) =\frac{ \alpha t}{\log(t)}
\end{align*}
and note that $G_0$ is nondecreasing and concave for $t\geq e^2$.  For $t\geq 1$, we define $G(t) =G_0(t+e^2)$ and observe that 
\begin{align*}
K(t) = \int_1^t \frac{1}{G(s)}\, ds = \int_{1+e^2}^{t+e^2} \frac{1}{G_0(s)} \, ds = \frac{1}{2\alpha}\bigg[(\log(t+e^2))^{2} - (\log(1+e^2))^{2} \bigg],\\
\end{align*}   
and it follows that  
\begin{align*}
K^{-1}(t)& = \exp\big((2\alpha t +[\log(1+e^2)]^{2})^{\frac{1}{2}}\big)-e^2,\\
r(t) &= \alpha \frac{\exp\big((2\alpha t +[\log(1+e^2)]^{2})^{\frac{1}{2}}\big)}{(2\alpha t +[\log(1+e^2)]^{2})^{\frac{1}{2}}}.
\end{align*}
\end{example}

\section{Random Splittings}
\label{sec:notation}

We now specialize to the setting of a discrete-time Markov chain $(X_n)$ with state space $(\mathbf{X}, \rho)=(\R^d, | \cdot  |)$ where $| \cdot |$ is the usual Euclidean distance.  In order to describe the specific Markov chains to be studied in the rest of the paper, we first fix some notation convenient to this setting.      

\subsection{Notation} For $x,y\in \R^d$, we use $\langle x  , y\rangle$ to denote the standard inner product between $x$ and $y$ so that $|x|=\sqrt{\langle x,x \rangle}$.  The symbol $\mathscr{V}_d$ denotes the set of complete, $C^\infty$ vector fields on $\R^d$; that is, $\mathscr{V}_d$ denotes the set of $C^\infty$ vector fields whose flows are define globally in time (on $\R$) and space (on $\R^d$).  For any $V\in \mathscr{V}_d$, we let $(t,x) \mapsto \varphi_t^V(x):\R\times \R^d\rightarrow \R^d$ denote the flow map associated to the differential equation 
 \begin{align}
 \label{eqn:genODE}
 \dot{x}_t= V(x_t).
 \end{align}
That is, for $(t,x) \in \R \times \R^d$, $\varphi_t^V(x)$ denotes the solution of~\eqref{eqn:genODE} at time $t$ with initial condition $x_0=x\in \R^d$.  We emphasize that since $V$ is complete, the mapping $(t, x)\mapsto \varphi_t(x)$ is defined for all $t\in \R$ and $x\in \R^d$. For any $V\in \mathscr{V}_d$, we say that a finite collection $\mathscr{S}=\{ V_1, V_2, \ldots, V_m\}\subset \mathscr{V}_d$
is a \emph{splitting} of $V$ provided 
\begin{align*}
\sum_{i=1}^m V_i = V.
\end{align*}
  We will use $S_m$ to denote the set of permutations of $\{1,2, \ldots, m\}$.  Given a splitting $\mathscr{S}=\{V_1, \ldots, V_m \} \subset \mathscr{V}_d$ of $V\in \mathscr{V}_d$, we will often use $\varphi_t^i(x)$, $x\in\R^d$, to denote $\varphi_t^{V_i}(x)$.  In this setting, for any $\sigma = (\sigma_1, \ldots,\sigma_m ) \in S_m$, $x\in \R^d$ and $t=(t_1,\ldots, t_m) \in [0,\infty)^m$ we let 
\begin{align}
\label{def:phisig}
\Phi^\sigma_t(x):= \varphi_{t_m}^{\sigma_m} \circ \cdots \circ \varphi_{t_1}^{\sigma_1}(x)  
\end{align}    
where $\circ$ denotes composition.  
In words, $\Phi^\sigma_t(x)$ is the result of flowing along the vector fields in the splitting $\mathscr{S}$ in the order prescribed by the permutation $\sigma$ starting from $x\in \R^d$, with corresponding times spent along each trajectory given by $t=(t_1, t_2, \ldots, t_m)$.  For $m\in \N$ as in $\mathscr{S}= \{ V_1, V_2, \ldots, V_m\}$, we let 
\begin{align}
u_n=(u_{1n}, u_{2n}, \ldots, u_{mn}), \,\,\, n\in \Z_{\geq 0}, 
\end{align}
denote a countably infinite collection of independent, uniformly distributed random variables on $S_m$.  Also, let    
\begin{align}
\label{def:tau}
 \tau_n= (\tau_{1n}, \tau_{2n}, \ldots, \tau_{mn} ), \,\,\, n \in \Z_{\geq 0},
\end{align}
be an independent collection of random vectors such that for each $i$ and $j$, $\tau_{ij}\sim \exp(1/h)$.  We furthermore suppose that the collection $\{ \tau_n \, : \, n \in \Z_{\geq 0}\}$ is independent of the collection $\{ u_n \, : \, n \in \Z_{\geq 0}\}$.  Throughout the paper, we let $\tau$ be a generic  $\exp(1/h)$ distributed random variable.

\subsection{Random Splittings}
\label{sec:gen}

We can construct a discrete-time Markov chain $(X_n)$ with state space $\R^d$ associated to the splitting $\mathscr{S}$ inductively by 
\begin{align}
\label{def:MC}
X_{n+1} = \Phi^{u_{n}}_{\tau_{n}} \circ X_n  = \varphi^{u_{mn}}_{\tau_{mn}} \circ \cdots \circ \varphi^{u_{1n}}_{\tau_{1n}} \circ X_n. 
\end{align} 
That is, starting from state $X_n$, the state $X_{n+1}$ is determined by selecting a uniformly distributed ordering of the vector fields in $\mathscr{S}$ and then flowing along each of these in the selected order for an independent, $\exp(1/h)$ amount of time.  We call the Markov chain $( X_n)$ the \emph{random splitting} associated to $\mathscr{S}$. 
For notational convenience, we will also make use of the random variables $X_{n,k}$, $n\geq 0$, $k=0, 1,2,\ldots, m$, defined by 
\begin{align*}
X_{n,k}=
\begin{cases}
X_n & \text{ if }k=0, n\geq 0\\
\varphi^{u_{kn}}_{\tau_{kn}} \circ \cdots \circ \varphi^{u_{1n}}_{\tau_{1n}} \circ X_n, & \text{ if } k=1,2, \ldots, m, \,\, n \geq 0.
\end{cases}
\end{align*}
Note that the stochastic sequence $(X_{n,k})_{k=0}^{m}$ keeps track of
the dynamics between steps $n$ and $n+1$ of the random splitting.  The
sequence $(X_{n,k})$ will be referred to as the \emph{random dynamics}
associated to the splitting $\mathscr{S}$. Dynamics generated by a
random splitting can be understood as a \emph{random dynamical system} \cite{Arnold} or a collection of random maps~\cite{Kifer}. We use the
term ``random splitting" to emphasize the origin of the random maps and
because we will vary the time step parameter $h$.  The concept of a random splitting is also closely related to piecewise
deterministic Markov processes (PDMPs)~\cite{Bakhtin,PDMP, Davis}.

\subsubsection{Random splitting of Lorenz '96}  
\label{sec:lorint}
As a non-trivial, illustrative example in this section, we consider a splitting of the forced and partially damped Lorenz '96 equation on $\R^d$
\begin{align}
\label{eqn:Lor96}
\dot{x}_i = (x_{i+1}-x_{i-2})x_{i-1}  - \delta_{1i} x_i  + \beta_i  
\end{align}
where $i=1,\ldots, d$, $d\geq 4$.  The indices in equation~\eqref{eqn:Lor96} live on the discrete circle so that, in particular, $x_{d+1}\equiv x_{1}$, $x_{0}\equiv x_{d}$, and $x_{-1}\equiv x_{d-1}$.  We assume that the system~\eqref{eqn:Lor96} is forced externally by real-valued constants $\beta_i\neq 0$, $i=1,\ldots, d$, and that dissipation enters through the first mode only; that is, $\delta_{1i}=1$ if $i=1\, (\text{mod} \,d)$ and $\delta_{1i}=0$ otherwise.

To define the splitting, let $\{e_i\}_{i=1}^{d}$ denote the standard orthornormal basis of $\R^d$ and introduce the following vector fields on $\R^d$\begin{align}
\label{eqn:vf1}
&V_{i}(x) =x_{i+1} x_{i-1} e_i - x_i x_{i-1} e_{i+1},\,\, i=1,\ldots, d, \qquad V_{\star}(x) =  -x_1 e_1 + \sum_{i=1}^d \beta_i  e_i ,\\
\label{eqn:vf2}&\qquad \qquad \qquad V(x)= \sum_{i=1}^d [ (x_{i+1}-x_{i-2})x_{i-1}  - \delta_{1i} x_i  + \beta_i ] e_i. 
\end{align}
One can show that each of the vector fields in~\eqref{eqn:vf1}-\eqref{eqn:vf2} belongs to $\mathscr{V}_d$ since the flows along each $V_i$, $i=1,2,\ldots,d$, conserve the Euclidean length $|\cdot|$.  Furthermore, $\mathscr{S}:= \{ V_1, \ldots, V_d,  V_{\star} \}$ forms a splitting of $V\in \mathscr{V}_d$ since
\begin{align}
V(x)=V_\star(x)+ \sum_{i=1}^{d} V_i(x). 
\end{align}

Let $(X_n)$ denote the random splitting associated to $\mathscr{S}$ and $(\mathcal{P}_n)$ be the associated Markov semigroup.  If $H:\R^d\rightarrow [1, \infty)$ is given by 
\begin{align}
\label{eqn:Lor96H}
H(x)= |x| +1,
\end{align} 
one of our main goals below will be to prove the following result. 
\begin{theorem}
\label{thm:main1}
For every $h>0$ small enough but independent of the dimension $d$, there exist constants $\alpha=\alpha(d,h) \in (0, 1)$ and $f=f(d,h) >0$ such that the bound is satisfied
\begin{align}
\label{eqn:bound}
\mathcal{P}_{d+1} H(x) \leq \alpha H(x) + f
\end{align}
for all $x\in \R^d$. 
\end{theorem}

The proof of Theorem~\ref{thm:main1} will be established partly here in Section~\ref{sec:notation} and partly in Section~\ref{sec:lor} by validating the hypotheses of Proposition~\ref{prop:Lyop} with a constant function $p\equiv p_*>0$.  Surprisingly, even though damping is externally present only on the first mode, it propagates through the random dynamics to produce the globally contractive-type estimate~\eqref{eqn:bound} for the random splitting.  In this regard, the randomization of the ordering of the vector fields is important in the arguments, as it allows us more easily see how energy is transported from one direction to another using a convenient ordering of the fields.  Such an ordering necessarily happens with positive probability.     

\begin{remark}
We recall that Theorem~\ref{thm:main1} ensures the existence of a stationary distribution for the random splitting and also implies that return times to large compact sets have exponential moments (see Section~\ref{sec:GRPD}).  Note that in previous work~\cite{AMM_23} it was shown that the randomly split dynamics in general converges on finite time windows to the underlying deterministic ODE as the switching rate tends to infinity (equivalently, as the mean step size $h>0$ converges to $0$).  This implies that finite segments of the stochastic trajectory will, as $h\rightarrow 0$, increasingly look like the solution of the deterministic ODE.  Consequently, the random splitting will likely affect the structure of the associated invariant measure since this is ``time infinity" object. Nonetheless, studying the random splitting helps illuminate some of the mechanisms involved with the transfer of energy leading to dissipative effects.     
\end{remark}

\subsection{Section overview}Returning to the setting of a general random splitting, in Section~\ref{sec:subcon}, Section~\ref{sec:pd} and Section~\ref{sec:entrances} we describe the important structural features that produce the bounds required by Proposition~\ref{prop:Lyop}, and in particular those that lead to estimates like~\eqref{eqn:bound}. Namely: 
\begin{itemize}
\item A quantity approximately conserved under each vector field in the splitting, i.e. the function $H$ in this discussion.  This will be further elaborated in Section~\ref{sec:subcon} using the notion of \emph{subconservation}.  Importantly, subconservation produces the bound \eqref{eqn:subconest} for some constant $f_2 \geq 0$, for any event $A\in \mathcal{F}$. 
\item Dissipation present in at least one vector field in the splitting $\mathscr{S}$.  This will be discussed in Section~\ref{sec:pd} below.  This structure allows one to establish a bound of the form~\eqref{eqn:PDest} for a collection of events $A=A(x)$.  These events correspond to certain entrances of the random dynamics to a ``dissipative region" $D$ in space. 
\item Such entrances to the dissipative region $D$ must have sufficient probability so that the estimate $\PP_x(A) \geq p(H(x))$ in Proposition~\ref{prop:Lyop} is satisfied for all $x\in \R^d$ for some ``reasonable" function $p:[1, \infty)\rightarrow (0,1]$.  In the case of the splitting for the Lorenz '96 equation above, we need to be able to take $p$ in Proposition~\ref{prop:Lyop} to be a constant function, i.e. $p\equiv p_*>0$.  This point will be discussed in more detail in Section~\ref{sec:entrances}. 
\end{itemize}

Unless otherwise specified, let $\mathscr{S}= \{ V_1, V_2, \ldots, V_m \}\subset \mathscr{V}_d$ be a splitting of $V\in \mathscr{V}_d$, and let $H:\R^d\rightarrow [1, \infty)$ be measureable.  
 \subsection{Subconservation}
 \label{sec:subcon}

Perhaps the most basic structural feature of the splittings considered in this paper is that each vector field in $\mathscr{S}$ ``approximately conserves" the function $H$.  To be more precise:   
 
 \begin{definition}
\label{def:subcon}
We call a splitting $\mathscr{S}$ \emph{subconservative} with respect to $H$ if for every $i=1,2,\ldots, m$ there exists measurable $F_i :[0, \infty)\rightarrow [0, \infty)$ such that $f_i:= \E F_i (\tau)< \infty$ and 
\begin{align}
\label{eqn:subcon}
H(\varphi_t^i(x) ) \leq H(x)+  F_i  (t)  \end{align}
for all $x\in \R^d$, $t\geq 0$. 
\end{definition}
\begin{remark}
Observe that if a vector field $W\in \mathscr{V}_d$ conserves $H$, then $H( \varphi_t^W(x))= H(x)$ for all $x\in \R^d$, $t\geq 0$.  Thus subconservation is a slight deviation from this concept where the function $H$ along all of the vector fields in the splitting does not grow ``too fast".  \end{remark}

 \begin{example}
 \label{ex:lor1}
 Consider the splitting $\mathscr{S}= \{ V_1, \ldots, V_d, V_\star\}$ of the Lorenz '96 equation~\eqref{eqn:Lor96} introduced in~\eqref{eqn:vf1}, and let $H$ be as in~\eqref{eqn:Lor96H}.  We claim that $\mathscr{S}$ is subconservative with respect to $H$.  To see why, note first that for $i=1,2,\ldots, d$ and $t\geq 0$ 
\begin{align*}
H(\varphi^i_t(x))= H(x)
\end{align*}   
since each of the vector fields $V_1, V_2, \ldots, V_d$ conserves length $|x|$.  In particular, we can take $F_i \equiv 0$ in~\eqref{eqn:subcon}.  On the other hand, letting $\varphi_t^\star(x):= \varphi_t^{V_\star}(x)$ and using the triangle inequality we have  
\begin{align*}
H(\varphi_t^{\star}(x)) &= | ( e^{-t} x_1+ (1-e^{-t}) \beta_1 ,x_2+\beta_2 t, \ldots, x_d+ \beta_d t)| +1 \\
&  \leq H(x) + |(\beta_1, \beta_2, \ldots, \beta_d)| t , \quad t\geq 0. 
\end{align*}
Hence we can take $F_\star(t) = |(\beta_1, \ldots, \beta_d)|t $ as in~\eqref{eqn:subcon} and the claim is established.  

 \end{example}

As a basic consequence of subconservation, we have the following result.  See also~\eqref{eqn:subconest}.   

\begin{proposition}
\label{prop:subcon}
Suppose that $H$ is subconservative for $\mathscr{S}$ and let $(X_n)$ denote the random splitting associated to $\mathscr{S}$.  Then 
\begin{align}
\label{eqn:subconservativebound}
\E_x \Big[\one_A H(X_{n})\Big] \leq \PP_x(A) H(x) + n \sum_{i=1}^m f_i
\end{align}    
for all $x\in \R^d$, $n\geq 0$ and $A\in \mathcal{F}$.  In the above, $f_i=\E F_i(\tau) \geq 0$ is as in~\eqref{eqn:subcon}.  
\end{proposition}

\begin{proof}
The bound~\eqref{eqn:subconservativebound} is trivially satisfied when $n=0$.  Thus suppose $n\geq 1$.  Since $\mathscr{S}$ is subconservative with respect to $H$, let $F_i$ and $f_i$ be as in~\eqref{eqn:subcon} and observe that for any $i \in \{1,2, \ldots, m-1\}$ we have by definition of the random splitting and subconservation
\begin{align}
\label{eqn:sub1} H(X_{n}) = H(X_{n-1,m})&\leq H( X_{n-1, m-i}) + \textstyle{\sum_{k=m-i+1}^m} F_{u_{k (n-1)}}(\tau_{k (n-1)}) \\
\nonumber & \leq H(X_{n-1}) +\textstyle{\sum}_{k=1}^{m}  F_{u_{k(n-1) } }(\tau_{k (n-1) }).  
\end{align}
Hence using independence it follows for any $A\in \mathcal{F}$   
\begin{align}
\label{eqn:sub3}
\E_x [\mathbf{1}_A H(X_{n})  ] \leq \E_x [\mathbf{1}_A H(X_{n-1})  ]+  \textstyle{\sum}_{i=1}^m f_i . 
\end{align}
The bound~\eqref{eqn:subconservativebound} follows by iterating~\eqref{eqn:sub3}.
\end{proof}

\subsection{Partial dissipation}
\label{sec:pd}

We next explore the basic structural feature of the random splittings in this paper which give rise to estimates of the form~\eqref{eqn:PDest}.  Recall that $\mathcal{B}$ denotes the Borel measurable subsets of the state space $\mathbf{X}=\R^d$.
\begin{definition}
\label{def:pd}
Let $D\in \mathcal{B}$ be non-empty.  We call the splitting $\mathscr{S}$ \emph{partially dissipative for} $H$ \emph{on the set} $D$ \emph{with index} $\ell\in \{1,2,\ldots, m\}$ if there exists a measurable function $G_\ell:[0, \infty)\rightarrow [0, \infty)$ with $g_\ell:= \E G_\ell(\tau) < \infty$ and a strictly decreasing function $\alpha_\ell: [0, \infty) \rightarrow (0,1]$  for which the bound holds  
\begin{align}
\label{eqn:pd}
H(\varphi^\ell_t(x)) \leq \alpha_\ell(t) H(x) + G_\ell(t) , \,\,\, x\in D, \,\, t\geq 0.
\end{align}
Notationally, we set $a_\ell = \E \alpha_\ell (\tau)\in (0,1)$.  The set $D$ will be called a \emph{dissipative region} for the random splitting $(X_n)$.  The index $\ell$ will be called a \emph{dissipative direction}.   
\end{definition}

\begin{example} 
\label{ex:lor2}
 Consider again the splitting $\mathscr{S}= \{ V_1, \ldots, V_d, V_\star\}$ of the Lorenz '96 equation~\eqref{eqn:Lor96} introduced in~\eqref{eqn:vf1}, and let $H$ be as in~\eqref{eqn:Lor96H}.  Fixing $\eta \in (0,1)$, we claim that $\mathscr{S}$ is partially dissipative with respect to $H$ on the set 
 \begin{align}
 \label{def:Deta}
 D_\eta= \{x\in \R^n \, : \, |x_1|^2 \geq \eta |x|^2\} 
 \end{align}
 with index $\ell= \star$.  Indeed, if $F_\star(t) = |(\beta_1, \ldots, \beta_d)| t$, then for $t\geq 0$ we have  
\begin{align*}
H(\varphi_t^{\star}(x))   &= | ( e^{-t} x_1+ (1-e^{-t}) \beta_1 ,x_2+\beta_2 t, \ldots, x_d+ \beta_d t)| +1\\
&\leq H( e^{-t} x_1, x_2,\ldots, x_d)  + F_\star(t) \\
& =\sqrt{(e^{-2t}-1) x_1^2+|x|^2} +1+ F_\star(t) \\
& \leq \sqrt{1-\eta(1-e^{-2t})} |x| +1 + F_\star(t)  \leq  \alpha_\star(t) H(x) + F_\star(t) +1
\end{align*}
where $\alpha_\star(t) = \sqrt{1-\eta(1-e^{-2t})}$ is strictly decreasing on $[0, \infty)$ and maps $[0, \infty)$ into $(0,1]$.  This finishes the proof of the claim. 

\end{example}    
  
 Our next result relates partial dissipation and subconservation to an estimate of the form~\eqref{eqn:PDest}.  To state it, recalling that $(X_{n,k})$ denotes the random dynamics associated to the splitting $\mathscr{S}$, for any $D\in \mathcal{B}$, $\ell \in \{1,2,\ldots, m\}$ and $n\geq 0$, let 
 \begin{align*}
 T_{n,D}= \min\big\{ k\in \{ 0,1,\ldots, m-1\} \,: \, X_{n,k} \in D\big\}, 
 \end{align*}
 where we understand $\min \emptyset = \infty$,
 and    
 \begin{align}
 \label{eqn:union}A_{n}^\ell(D) &= \bigcup_{k=0}^{m-1}\Big\{ T_{n,D}=k \Big\} \cap \Big\{ \ell = u_{(k+1)n}   \Big\}=  : \bigcup_{k=0}^{m-1} A_{n,k}^\ell(D).
 \end{align} 
 Note that $A_{n,k}^\ell(D)$ is the event that the random sequence $(X_{n,j})_{j=0}^{m-1}$ first enters the set $D$ after the $k$th flow and the next flow, namely the $(k+1)$st flow, follows the vector field $V_\ell$.  
 \begin{proposition}
 \label{prop:PDest}
Let $D\in \mathcal{B}$ be nonempty and suppose that the splitting $\mathscr{S}$ is subconservative for $H$ and partially dissipative for $H$ on the set $D$ with index $\ell$, and let $a_\ell \in (0,1)$ be as in Definition~\ref{def:pd}.  Then there exists constants $f_{n,k}\geq 0$ such that 
\begin{align}
\label{eqn:PDest1}
\E_x\big[ \one_{A_{n}^\ell(D)} H(X_k)\big] \leq a_\ell \PP_x(A_{n}^\ell(D))  H(x) + f_{n,k}  
\end{align}
for all $x\in \R^d$, $k\geq n+1$. 
 \end{proposition}

 \begin{proof}
Let $D\in \mathcal{B}$ be nonempty and suppose that the bound~\eqref{eqn:pd} holds for some index $\ell \in \{1,2,\ldots, m\}$, measurable $G_\ell:[0, \infty)\rightarrow [0, \infty)$ with $g_\ell= \E G_\ell(\tau)<\infty$ and strictly decreasing function $\alpha_\ell: [0, \infty) \rightarrow (0,1]$ with $a_\ell = \E \alpha_\ell (\tau)$.  Since $\mathscr{S}$ is subconservative, suppose the bounds~\eqref{eqn:subcon} hold for some $F_i:[0, \infty)\rightarrow [0, \infty)$ measurable with $f_i = \E F_i(\tau)<\infty$.  For simplicity in the arguments below, set $A=A_{n}^\ell(D)$ and $A_j= A_{n,j}^\ell(D)$. 

Since $k\geq n+1$ and the union~\eqref{eqn:union} is disjoint, relation~\eqref{eqn:sub3} gives
 \begin{align}
\nonumber \E_x \one_{A} H(X_k)& \leq \E_x \one_{A} H(X_{n+1}) + (k-n-1) \textstyle{\sum_{i=1}^m f_i}\\
\nonumber & = \textstyle{\sum_{j=0}^{m-1}} \E_x \one_{A_{j}} H(X_{n+1}) + (k-n-1)  \textstyle{\sum_{i=1}^m f_i} \\
\label{eqn:prop:PDest1}& \leq \textstyle{\sum_{j=0}^{m-1}} \E_x \one_{A_{j}} H(X_{n,j+1}) + (k-n)  \textstyle{\sum_{i=1}^m f_i} . \end{align}   
 Next, observe that 
 \begin{align*}
\one_{A_{j}} H(X_{n,j+1}) \leq \one_{A_{j}} \alpha_\ell(\tau_{(j+1)n}) H(X_{n,j}) + \one_{A_j}G_\ell(\tau_{(j+1)n}).
\end{align*}
Using independence of $\tau_{(j+1)n}$ and $\one_{A_{j}} H(X_{n,j})$, and independence of $\tau_{(j+1)n}$ and $\one_{A_{j}}$, we thus obtain\begin{align*}
\E_x \one_{A_{j}} H(X_{n, j+1})  \leq a_\ell \E_x[ \one_{A_{j}} H(X_{n,j})] +  g_\ell \PP(A_j ).
\end{align*}
Hence, this estimate with subconservation implies
\begin{align}
\nonumber \E_x \one_{A_{j}} H(X_{n,j+1}) & \leq a_\ell \E_x[ \one_{A_j} H(X_{n,j})] + g_\ell  \PP(A_j) \\
\label{eqn:prop:PDest2}& \leq a_\ell \PP_x( A_j) H(x)  + (n+1)\textstyle{\sum_{j=1}^m} f_i + g_\ell  \PP(A_j).
\end{align}
Combining~\eqref{eqn:prop:PDest1} with~\eqref{eqn:prop:PDest2}, we obtain the claimed estimate~\eqref{eqn:PDest1} with $f_{n,k}$ given by 
\begin{align}
\label{eqn:defalphaf_{n,k}}
f_{n,k} = g_\ell+ ((m-1)n + m +k) \textstyle{\sum_{i=1}^m f_i}.  
\end{align}

  \end{proof}

 \subsection{Entrances to a dissipative region}
 \label{sec:entrances}
As we have seen in the case of the random splitting  of the Lorenz '96 equation~\eqref{eqn:Lor96}, checking subconservation and partial dissipation for the random splitting is relatively straightforward.  See Example~\ref{ex:lor1} and Example~\ref{ex:lor2} for further details of this point.  By way of Proposition~\ref{prop:Lyop}, the central difficulty in showing a bound like~\eqref{eqn:contract} is estimating $\PP_x(A_n^\ell(D))$ for a dissipative set $D$ and a dissipative direction $\ell \in \{1,2,\ldots, m \}$.  Such an estimate comes from the specific nature of the dynamics of the splitting, as the system transports an initial condition $x\in \R^d$ to $D$ in a certain number of steps $n$.  The way in which this transport happens inherently depends on $x$.  In this section, we provide the condition (Theorem~\ref{thm:check} below) we will check in subsequent sections to establish the needed lower bound on this probability.

To state the result, define for $R>0$
 \begin{align}
 H_{>R} = \{ x\in \R^d\, : \, H(x) >R \}.   
 \end{align}
 
% \begin{definition}
% Let $D\in \mathcal{B}$ be non-empty, $\ell \in \{ 1,2,\ldots, m\}$, and $p:[1, \infty) \rightarrow (0,1]$ be measurable.  We say that $D$ is $(\ell, p, H)$-\emph{uniformly accessible} if there exists $n\in \Z_{\geq 1}$, $R>0$ such that for every $x\in H_{>R}$   
% \begin{align}
% \label{eqn:lpHunif}
% \PP_x\bigg( \bigcup_{j=0}^{n-1} A_{j}^\ell(D)\bigg) \geq n p(H(x)).
% \end{align}
% \end{definition}
% 
% 

%
%\begin{condition1}
%There exists measurable functions $n:\R^d\rightarrow \Z_{\geq 1}$ and $p:[1, \infty) \rightarrow (0, 1]$ and a constant $R>0$ such that $n_*= \max\{ n(x) \, : \, x\in \R^n \} <\infty$ and such that
% \begin{align*}
% \PP_x(A_{n(x)-1}^\ell(D)) \geq p(H(x)) \,\,\, \text{ for all } \,\, \, x\in H_{> R}.
%  \end{align*}   
%   \end{condition1}  

 \begin{theorem}  

 \label{thm:check}
 Let $D\in \mathcal{B}$ be nonempty and suppose that the splitting $\mathscr{S}$ is subconservative with respect to $H$ and partially dissipative with respect to $H$ on the set $D$ with index $\ell$.  Suppose there exist $n\in \Z_{\geq 1}$, $R>0$ and $p:[1, \infty)\rightarrow (0,1]$ such that 
 \begin{align}
 \label{eqn:lpHunif}
 \PP_x ( A_{j(x)}^\ell(D)) \geq p(H(x)) \,\,\, \text{ for all } \,\,\, x\in H_{>R},
 \end{align}
 for some index $j(x) \leq n-1$.  Then the Markov chain $(\tilde{X}_j):=(X_{j n})$ satisfies the hypotheses of Proposition~\ref{prop:Lyop} with the same choice of $p$.   
  \end{theorem}
  
  \begin{remark}
  We emphasize that the index $j(x)$ in~\eqref{eqn:lpHunif} depends on the initial state $x$ but has the uniform upper bound $j(x) \leq n-1$ over all $x\in H_{>R}$. 
  \end{remark}

 \begin{proof}[Proof of Theorem~\ref{thm:check}]
Since $\mathscr{S}$ is subconservative with respect to $H$, pick $F_i$ and $f_i= \E F_i(\tau)< \infty$ so that the bounds~\eqref{eqn:subcon} hold for $i=1,2,\ldots, m$.  Since $\mathscr{S}$ is partially dissipative with respect to $H$ on $D$ with index $\ell$, suppose that the bound~\eqref{eqn:pd} is satisfied for some measurable $G_\ell:[0, \infty) \rightarrow [0, \infty)$ and $\alpha_\ell:[0, \infty)\rightarrow (0,1]$ with $\alpha_\ell$ strictly decreasing and $g_\ell=\E G_\ell(\tau)< \infty$ and $a_\ell= \E \alpha_\ell (\tau)\in (0,1)$.  By hypothesis, pick $n\in \Z_{\geq 1}$ and $R>0$ such that for every $x\in H_{>R}$ we have~\eqref{eqn:lpHunif}.  We claim that the Markov chain $(\tilde{X}_j):=(X_{jn})$ satisfies the hypothesis of Proposition~\ref{prop:Lyop}.

For $x\in H_{>R}$, define the event $A=A(x)$ in the statement of Proposition~\ref{prop:Lyop} by $A(x)= A_{j(x)}^\ell(D)$.  Then  
\begin{align*}
\PP_x(A)=\PP(A(x))\geq p(H(x)) \,\,\text{ for all } x\in H_{>R}. 
\end{align*}
On the other hand, for $x\in H_{\leq R}$ we define $A=A(x)= \Omega$, so that $$\PP_x(A)= \PP(A(x))=1 \geq p(H(x)) \,\, \text{ for all } \,\, x\in H_{\leq R}.$$  Next, by Proposition~\ref{prop:PDest}, for any $x\in H_{>R}$ we have 
 \begin{align*}
 \E_x\one_A H(\tilde{X}_1)= \E_x \one_A H(X_{n}) &\leq a_\ell \PP_x(A) H(x) + c_n
 \end{align*}
for some constant $c_n\geq 0$.  
 By Proposition~\ref{prop:subcon}, we have for any $B\in \mathcal{F}$ and $x\in \R^d$
 \begin{align*}
\E_x\one_{B} H(\tilde{X}_1)= \E_x \one_{B} H(X_{n})& \leq \PP_x(B) H(x) + n \textstyle{\sum_{i=1}^m f_i}.  
 \end{align*}   
In particular, for any $x\in H_{\leq R}$ we also have the estimate
 \begin{align*}
 \E_x \one_{A} H(\tilde{X}_1) \leq H(x) + n \textstyle{\sum_{i=1}^m f_i} \leq a_\ell \PP_x(A) H(x) + (1-a_\ell) R + n \textstyle{\sum_{i=1}^m f_i}.  \end{align*}  
 \end{proof}

\section{Entrance probabilities to the dissipative region: the Lorenz '96 Splitting}
\label{sec:lor}

Recall the random splitting $(X_n)$ of the Lorenz '96 equation~\eqref{eqn:Lor96} introduced in Section~\ref{sec:lorint} and discussed in Example~\ref{ex:lor1} and Example~\ref{ex:lor2}.   In this section we finish proving Theorem~\ref{thm:main1} by way of Theorem~\ref{thm:check}.  In particular, we have left to show the required lower bound on the entrance probabilities as in~\eqref{eqn:lpHunif}.     

To setup the precise statement of the result to be proven in this section, pick $h_* \in (0,\pi/12)$ small enough so that for $h\leq h_*$ each of the following conditions are met:  
\begin{align}
\label{eqn:h_*choice}
&|\sin(|y|)|\leq h\,\,  \text{ implies  }\,\,| |y| - j \pi | < 3h\,\, \text{ for some }\,\, j \in \Z_{\geq 0},\\
\nonumber & \qquad e^{-6h} \geq \tfrac{3}{4}, \,\,\,  \text{ and } \,\,\, 1-e^{-5h} - 3h e^{-5h} \geq h. 
\end{align}
Throughout this section, we fix $h \in (0, h_*]$ and recall that $h$ corresponds to the exponential clocks $\tau_{ij} \sim \exp(1/h)$.  Also recall the set $D_\eta$ introduced in~\eqref{def:Deta}.  Using Example~\ref{ex:lor1} and Example~\ref{ex:lor2}, to show Theorem~\ref{thm:main1} it suffices to show the following result by Theorem~\ref{thm:check}.
\begin{theorem}
\label{thm:lblor}
There exist $\eta \in (0,1)$, $R>0$ and a constant $p_*>0$ such that for every $x\in H_{>R}$
\begin{align}
\PP_x (A_{j(x)}^\star(D_\eta) ) \geq p_*
\end{align}
for some index $j(x) \leq d$. 
\end{theorem}
To prove Theorem~\ref{thm:lblor}, we first deduce some key lemmata below in Section~\ref{sec:keylemlor}.  These results explain how ``energy" is transferred and/or maintained along each of the vector fields in the splitting $\mathscr{S}$ under certain assumptions on the initial data.  Afterwards, we will use the lemmata to conclude Theorem~\ref{thm:lblor}.  

\begin{remark}
We expect the conclusion of Theorem~\ref{thm:lblor} to hold for general $h>0$ not necessarily satisfying~\eqref{eqn:h_*choice}.  The choice of $h\in (0, h_*]$ is for convenience in the proofs below and because our primary interest is in the regime where $h>0$ is small, fixed.   
\end{remark} 

\begin{remark}
Although we assume nontrivial forcing on every direction in the random splitting for Lorenz '96 (see the Section~\ref{sec:lorint}), we believe a similar result holds even if less forcing is present.  However, the argument used would likely be more complicated because more nontrivial combinations of the $V_i$'s would have to be used to propagate energy to the damped direction.  Furthermore, with less forcing, the structure of the dissipative bound~\eqref{eqn:bound} may also change to reflect longer return times to the center of space.           
\end{remark}

\subsection{The key lemmata}
\label{sec:keylemlor}

In what follows, we let $\pi_j :\R^d\rightarrow \R$ denote the projection onto the $j$th coordinate.  The first lemma describes how energy is transferred downward from mode $j$ to $j-1$ along $V_{j-1}$ as in~\eqref{eqn:vf1}.  An analogous result holds going upward from $j-1$ to $j$, but we provide only the result below for expeditiousness. 

Because of the cyclical nature of the vector fields $V_1, V_2, \ldots, V_d$, we adopt the convention that their indices are defined modulo $d$, e.g. $V_{d+1}=V_1$, with the appropriate shifts in the definition of the vector field.   Recall also the notation
\begin{align*}
\varphi^{\star}_t(x)=\varphi_t^{V_\star}(x), \,\, t\geq 0, \, x\in \R^d. 
\end{align*}
\begin{lemma}[Downward energy transfer]
\label{lem:down}
Let $x\in \R^d_{\neq 0}$, and suppose that for some $j=1,\ldots, d$ \text{\emph{(mod $d$)}}, there exists constants $c_1>c_2>0$ such that $|x_j|\geq c_1 |x|$ and $|x_{j-1}|\leq  c_2 h |x|$.  If $|x_{j-2}|\geq c_3 >0$ for some constant $c_3>0$, then with probability at least $p_1=p_1(c_3)>0$ depending only on $c_3$ we have 
\begin{align*}
|\pi_{j-1} (\varphi_\tau^{j-1}(x))| \geq (c_1-c_2) h |x|= (c_1-c_2) h | \varphi_\tau^{j-1}(x)|. 
\end{align*}   
\end{lemma}  

\begin{proof}
Note that 
\begin{align}
z_1:=\pi_{j-1}(\varphi_\tau^{j-1}(x))= x_{j-1} \cos (| x_{j-2}| \tau)+ \text{sgn}(x_{j-2}) x_{j} \sin(|x_{j-2}| \tau).  
\end{align}
Observe that on the event $B_1=\{ |z_1| \leq (c_1-c_2) h |x|\}$ we have under our assumptions on the initial data that 
\begin{align}
|\sin ( |x_{j-2}| \tau) | \leq h .  
\end{align}
By the choice of $h_*>0$ in~\eqref{eqn:h_*choice}, on the event $B_1=\{ |z_1| \leq  h \}$ for $h\leq h_*$ the event 
\begin{align}
\tilde{B}_1(x_{j-2}):=\bigcup_{k=0}^\infty\big\{| |x_{j-2} | \tau - k\pi| \leq 3 h  \big\}
\end{align}
must occur.  Hence, $B_1\subset \tilde{B}_1(x_{j-2})$. Furthermore, by the choice of $h_*$, the union above is disjoint.  We will estimate $\PP(\tilde{B}_1(x_{j-1})^c)\leq \PP(B_1^c)$.  To this end, setting $a= |x_{j-2}|\geq c_3>0$ for simplicity we obtain\begin{align*}
\PP(\tilde{B}_1(x_{j-2})^c) = \sum_{k=0}^\infty \PP\big(a \tau\in ( 3h+\pi k, -3h+ \pi(k+1) )  \big) &=\sum_{k=0}^\infty   e^{-\frac{\pi k}{h a}}\bigg\{e^{-\frac{3}{a}} - e^{\frac{3}{a}  - \frac{\pi}{ha}} \bigg\}\\&=\frac{e^{-\frac{3}{a}} - e^{\frac{3}{a}  - \frac{\pi}{ha}}}{1-e^{-\frac{\pi}{ha}}}.
\end{align*}  
Since $h_* < \pi/12$ we note that $\pi/h -6 \geq \pi/(2h)$.  Hence 
\begin{align*}
\PP(\tilde{B}_1(x_{j-2})^c) =\frac{e^{-\frac{3}{a}} - e^{\frac{3}{a}  - \frac{\pi}{ha}}}{1-e^{-\frac{\pi}{ha}}} = e^{-\frac{3}{a}}  \frac{1-e^{- ( \frac{\pi}{ha} -\frac{6}{a}) } }{1-e^{-\frac{\pi}{ha}}}&\geq e^{-\frac{3}{a}}  \frac{1-e^{- \frac{\pi}{2 ha} } }{1-e^{-\frac{\pi}{ha}}}\\
&=  \frac{e^{-\frac{3}{a}}}{1+e^{-\frac{\pi}{2ha}}}\geq  \frac{e^{-\frac{3}{c_{3}}}}{2}>0. 
\end{align*}
 This finishes the proof.

%Clearly, for any choice of $|x_{j-2}|\geq c_3>0$ and $h\in (0, h_*]$, $0<\PP (\tilde{B}_1(x_{j-1})^c) < 1$.  However, we note that as $|x_{j-2}|\rightarrow \infty$ for $h$ fixed we have, asymptotically,
%\begin{align}
%\PP(\tilde{B}_1(x_{j-1})^c) \sim 1- \frac{6h}{\pi}>0
%\end{align} 
%by the choice of $h_*$ in~\eqref{eqn:h_*choice}.
%In particular, setting
%\begin{align}
%p_1(h, c_3):=\inf_{|y| \geq c_3} \PP(\tilde{B}_1(x_{j-1})^c)  >0  
%\end{align} 
%finishes the proof.
\end{proof}

The next result is complementary to the previous.  Notice how in Lemma~\ref{lem:down} we use the fact that the speed of rotation along $V_{j-1}$, namely $|x_{j-2}|$, is positive.  The following shows how to make $|x_{j-2}|$ sufficiently large using the forcing terms $\beta_i$ while maintaining the hypotheses of Lemma~\ref{lem:down}.

\begin{lemma}[Speed generation]
\label{lem:speed}
Let $|\beta|:= |(\beta_1, \ldots, \beta_d)|$ and suppose that $2\leq j\leq d$ and that there exists a constant $c\in (0,1/4)$ such that $x\in\R^d$ satisfies
\begin{align}
|x| \geq \frac{36 |\beta| }{c}, \qquad |x_j| \geq 4 c |x|, \qquad \text{ and } \qquad |x_{j-1}| \leq  c h |x|. 
\end{align}
If $y= \varphi_\tau^\star(x)$, then with probability at least $p_2>0$ independent of $c,h$ we have 
\begin{align}
&|y| \geq \frac{|x|}{2}, \quad |y_j| \geq 3 c |y|, \qquad |y_{j-1}| \leq  2c h |y|,\qquad \text{ and } \quad |y_{j-2} | \geq  h|\beta_{j-2}|.
\end{align}  
\end{lemma}

\begin{proof}  
 First note the almost sure bound
\begin{align}
e^{-\tau} |x|- |\beta|  \tau \leq |y|\leq |x| + |\beta| \tau. 
\end{align}  
We will prove the result using cases depending on the value of $j$.

 \emph{Case 1} ($4\leq j \leq d$).  First suppose that $|x_{j-2}| \geq 2  |\beta_{j-2}| h$.  Then using the choice of $h_*$ in~\eqref{eqn:h_*choice} and $|x|$ in the statement, on the event $\{ \tau \leq  h\}$ we have that   
\begin{align}
\label{eqn:lorcor1}|y_j| &= |x_j + \beta_j \tau| \geq |x_j| - |\beta_j| \tau \geq  4 c| x|- |\beta| \tau\geq 3c|y|,\\
\label{eqn:lorcor2}|y_{j-1}|&= |x_{j-1}+\beta_{j-1} \tau| \leq |x_{j-1}|+ |\beta_{j-1}| \tau  \leq ch |x|+ |\beta| \tau \leq 2ch |y|,\\
\label{eqn:lorcor3}|y_{j-2}|&= |x_{j-2}+ \beta_{j-2} \tau| \geq |x_{j-2}|- |\beta_{j-2} |\tau \geq |\beta_{j-2}|  h,\\
\label{eqn:lorcor4} |y|& \geq  e^{-\tau} |x| -|\beta| \tau  \geq \frac{3}{4}|x|  -|\beta| \tau \geq \tfrac{|x|}{2}.
\end{align}
Furthermore, $\PP\{ \tau \leq  h\} =1-e^{-1}$.  On the other hand, if $|x_{j-2}| \leq 2  |\beta_{j-2}|h$, then on the event $\{3 h \leq \tau\leq 4  h\}$, which has probability $e^{-3}-e^{-4}$, we have~\eqref{eqn:lorcor1}, \eqref{eqn:lorcor2}, \eqref{eqn:lorcor4} and 
\begin{align}
\label{eqn:lorcor5}
|y_{j-2}|&= |x_{j-2}+ \beta_{j-2} \tau| \geq |\beta_{j-2} |\tau  -|x_{j-2}|\geq |\beta_{j-2}| h.
\end{align}
This finishes the proof in Case 1. 

 \emph{Case 2} $(j=3)$.  First suppose that $|x_{j-2}| \geq 3|\beta_{j-2}| h$.  Using the choice of $h_*$ as in~\eqref{eqn:h_*choice}, on the event $\{\tau \leq  h \}$ we have~\eqref{eqn:lorcor1}, \eqref{eqn:lorcor2}, \eqref{eqn:lorcor4} and 
 \begin{align*}
|y_{j-2}|&= |e^{-\tau} x_{j-2}+ (1-e^{-\tau}) \beta_{j-2} | \geq  |\beta_{j-2}| h \{e^{-h} 3    -1 \} \geq  |\beta_{j-2}| h
\end{align*}
where we used the bound $1-e^{-h} \leq h$.  
On the other hand, if $|x_{j-2}| \leq 3 |\beta_{j-2}| h$, then on the event $\{5h\leq  \tau \leq 6h \}$, which has probability $e^{-5}-e^{-6}$, we have~\eqref{eqn:lorcor1}, \eqref{eqn:lorcor2}, \eqref{eqn:lorcor4} and 
  \begin{align*}
|y_{j-2}|&= |e^{-\tau} x_{j-2}+ (1-e^{-\tau}) \beta_{j-2} | \geq  |\beta_{j-2}|\{ 1-e^{-5h}-3h e^{-5h} \} \geq |\beta_{j-2}|h  
\end{align*}
where again we used the choice of $h_*$ as in~\eqref{eqn:h_*choice}.

\emph{Case 3} $(j=2)$.  Suppose first that $|x_{j-2}| \geq 2 |\beta_{j-2}|h$.  Then on the event $\{ \tau \leq h\}$ we have~\eqref{eqn:lorcor1}, \eqref{eqn:lorcor3}, \eqref{eqn:lorcor4} and
\begin{align}
\label{eqn:lorcor6}
|y_{j-1}|&= |e^{-\tau} x_{j-1}+(1-e^{-\tau})\beta_{j-1} \tau| \leq |x_{j-1}|+ |\beta_{j-1}| \tau  \leq 2 c h |y|.
\end{align}    
On the other hand, if $|x_{j-2}| \leq 2 |\beta_{j-2}| h$, then on the event $\{ 3h \leq \tau \leq 4 h \}$ we have~\eqref{eqn:lorcor1}, \eqref{eqn:lorcor4}, \eqref{eqn:lorcor5}, and \eqref{eqn:lorcor6}.  This finishes the proof of the result. 

\end{proof}

Finally, the last lemma below shows how we can keep energy previously stored while going through the entire splitting.

\begin{lemma}[Energy maintenance]
\label{lem:energysame}
Let $r>0$, $|\beta| = |(\beta_1, \ldots, \beta_d)|$ and suppose $x\in \R^d$ has $|x| \geq r$.  Suppose that for some $j=2,3,\ldots, d$ and constant $c>0$ we have $|x_j| \geq c |x| $.  Then with probability at least $p_3=p_3(r,c,|\beta|)>0$ depending only on $r,c,|\beta|$ we have that 
\begin{align}
|\pi_{j} \varphi_\tau^{k}(x)| \geq \frac{c}{2}| \varphi^k_\tau(x)|, \,\, k=j-1, j,\qquad \text{ and } \qquad| \pi_j \varphi^{\star}_\tau(x)| \geq \frac{c}{2}| \varphi^{\star}_\tau(x)|.  
\end{align}
\end{lemma}

\begin{remark}
Observe that if $k \neq j-1, j, \star$, then by conservation we have
\begin{align}
|\pi_j\varphi_\tau^k(x)|=|x_j|\geq c |x|=c |\varphi^k_\tau(x)| 
\end{align}
almost surely.   
\end{remark}

\begin{proof}[Proof of Lemma~\ref{lem:energysame}]
We omit the argument for $k=j$ since it is similar to the case when $k=j-1$. If $k=j-1$, observe that  
\begin{align*}
\pi_j \varphi^k_\tau(x)= x_j \cos( x_{j-2} \tau) - x_{j-1} \sin( x_{j-2} \tau).
\end{align*}
Thus we consider the event $$B_2(j):=\big\{ | \pi_j \varphi^k_\tau(x)| \geq \tfrac{c}{2}|\varphi_\tau^k(x)|\big\}=\big\{| x_j \cos( x_{j-2} \tau) - x_{j-1} \sin( x_{j-2} \tau)|\geq \tfrac{c}{2}|x|\big\}.$$ Now for $\delta>0$ to be determined, if $|x_{j-2}|\leq \delta$, then on the event $\{ \tau \leq h \}$ we can pick $\delta >0$ small enough depending only on $c$ so that $B_2(j)$ occurs.  On the other hand, if $|x_{j-2}| \geq \delta >0$, then there exists $\epsilon\in (0,\pi/2)$ small enough depending only on $c$ so that  
\begin{align}
\label{eqn:unionb}
B_3(x_{j-2}):=\bigcup_{m=0}^\infty \big\{| |x_{j-2}|\tau- \pi m| \leq \epsilon \big\} \subset B_2(j)
\end{align}  
and the union in~\eqref{eqn:unionb} above is disjoint.  Now, for any $a=|x_{j-2}| \geq \delta$ we have 
\begin{align*}
\PP (B_3(x_{j-2})) &= 1-e^{-\frac{\epsilon}{h a}}+\sum_{m=1}^\infty e^{-  \frac{\pi m- \epsilon}{h a}} - e^{-\frac{\pi m+ \epsilon}{h a}}\\
&= 1-e^{-\frac{\epsilon}{h a}} + \frac{e^{-\frac{\pi}{ha}}}{1- e^{-\frac{\pi}{ha}}} \bigg[ e^{\frac{\epsilon}{ha}} - e^{\frac{-\epsilon}{ha}}\bigg]:=g(ha) \in (0, 1]. \end{align*}  
Note that $g(x)\rightarrow 2\epsilon/\pi$ as $x\rightarrow \infty$ and $g(x)\rightarrow 1$ as $x\rightarrow 0^+$.  Thus there exists a constant $p_3>0$ depending only on $\epsilon, \delta$ so that   
\begin{align*}
\inf_{|x_{j-2}| \geq \delta } \PP(B_3(x_{j-2})) \geq p_3. 
\end{align*}
This finishes the proof in the case when $k=j-1$.

Lastly, recall that $j\neq 1$.  Thus observe that on the event $\{ \tau \leq \tfrac{ r}{(1+2/c)|\beta|}\}$
\begin{align*}
|\pi_j \varphi^{\star}_\tau(x)| = | x_j + \beta_j \tau| \geq |x_j|- |\beta_j| \tau \geq c|x| - |\beta_j| \tau\geq \tfrac{c}{2}| \varphi_\tau^\star(x)|,  
\end{align*}
finishing the proof.

\end{proof}

\subsection{Proof of Theorem~\ref{thm:lblor}}

Given the previous lemmata, we can now conclude Theorem~\ref{thm:lblor}.  The needed bound will be verified region-by-region in space, essentially depending on how ``far" away the process is from the dissipative region.

To define the relevant regions, let $R_0= 1/\sqrt{d}$, define constants $R_{j}$, $j=1,2,\ldots, d-1$, inductively by
\begin{align}
R_j= \frac{R_{j-1} h}{16}
\end{align} 
and let 
\begin{align}
\label{def:r_*}
R= \frac{2^{d+5} (|\beta|\vee 1)}{R_{d-1}}
\end{align}
where we recall that $|\beta| = |(\beta_1, \beta_2, \ldots, \beta_d)|$.  
Let $U_j\subset \R^d$, $j=1,2, \ldots, d$, be defined by 
\begin{align}
U_1&= \bigg\{|x| \geq R\, : \, |x_1 | \geq R_{d-1}|x|\bigg\},\\
 U_j &= \bigg\{ |x|\geq 2^{j-1}R\, : \,  |x_j| \geq R_{d-j} |x|, \, \,\,|x_{j-1}| < R_{d-j+1} |x| \bigg\}, \qquad 2\leq j\leq d. 
\end{align}

We first note the following. 
\begin{lemma}
Let $R_*=2^{d-1} R+1$.  Then we have 
\begin{align*}
H_{> R_*} \subset \bigcup_{j=1}^d U_j. 
\end{align*}
\end{lemma}
\begin{proof} 
Suppose $x\in H_{> R_*}$ so that $|x| \geq 2^{d-1} R $.  Hence there must exist $j=1,2,\ldots,d$ such that $|x_j| \geq R_{d-j} |x|$.  We will show inductively that if $|x_j| \geq R_{d-j}|x|$ with $|x|\geq R_*$, then 
\begin{align}
\label{eqn:containment}
x\in \bigcup_{\ell=1}^{j} U_\ell. 
\end{align}
 If $j=1$, then $|x_1| \geq   R_{d-1}|x|$ and $|x|\geq R_*\geq R$ so that $x\in U_1$.  Inductively if $|x_j| \geq R_{d-j}|x|$ with $|x| \geq R_* \geq  2^{j-1} R$, then either $|x_{j-1}|< R_{d-j+1}|x|$, in which case $x\in U_j$, or $|x_{j-1}|\geq R_{d-j+1}|x|$, in which case the inductive hypothesis implies~\eqref{eqn:containment}.         
\end{proof}

\begin{proof}[Proof of Theorem~\ref{thm:lblor}]  
Note that if $x\in U_1$, then $x\in D_{R_{d-1}}$.  Thus we have 
\begin{align*}
\inf_{x\in U_1} \PP_x( A_0^\star (D_{R_{d-1}})) \geq \PP( u_{10}=\star)= \frac{1}{d+1}.  
\end{align*}   
Using the Markov property, it therefore suffices to show that for all $j \in \{2,\ldots, d\}$
\begin{align}
\inf_{x\in U_j}\PP_x(X_1 \in U_{j-1}) >0.  
\end{align}
To this end, let $j\in \{2, \ldots, d\}$ and assume that $x\in U_j$ so that 
\begin{align}
|x|\geq 2^{j-1}R, \,\,\, |x_j| \geq R_{d-j}|x|\,\,\,\text{ and } \,\,\, |x_{j-1}|< R_{d-j+1}|x|.  
\end{align}
Let $\sigma$ be any permutation of $\{ 1,\ldots, d, \star\}$ with $\sigma_1=\star$ and $\sigma_2=j-1$.   First applying Lemma~\ref{lem:speed} and then applying Lemma~\ref{lem:down}, we see that with probability $q_0(j, d, h,\beta_{j-2})>0$ the point $y= \varphi_{\tau_{20}}^{\sigma_2} \circ \varphi_{\tau_{10}}^{\sigma_1}(x)$ satisfies \begin{align}
\label{eqn:lorlb1}
|\pi_{j-1} y| \geq \frac{R_{d-j}}{4}h |y| \qquad \text{ and } \qquad |y| \geq \frac{|x|}{2} \geq 2^{j-2} R. 
\end{align}  
Applying Lemma~\ref{lem:energysame} and setting $z= \varphi_{\tau_{(d+1)0}}^{\sigma_{d+1}} \circ \cdots \circ \varphi_{\tau_{30}}^{\sigma_3}\circ y$, we find that 
\begin{align}
| \pi_{j-1} z| \geq \frac{R_{d-j}}{8} h| z| \geq R_{d-j +1} |z| \qquad \text{ and } \qquad |z| \geq \frac{|x|}{2} \geq 2^{j-2} R.
\end{align}
with probability $q(j, d, h, \beta_{j-1})>0$.  That is, with probability $q(j, d, h, \beta_{j-2})$, $\Phi_{\tau_0}^{\sigma} (x) \in U_{j-1}$ so that, by independence,
\begin{align*}
\inf_{x\in U_j} \PP_x(X_1 \in U_{j-1} ) &\geq \inf_{x\in U_j} \PP_x( \Phi_{\tau_0}^{u_0}(x) \in U_{j-1}, u_{10}=\star, u_{20}=1 ) \\
& \geq \frac{q(j, d,h,\beta_{j-2})}{d(d+1)}>0.  
\end{align*} 
This finishes the proof. 
  
\end{proof}

\section{Random Splittings of Galerkin Truncations of 2D Euler}
\label{sec:NSE}

In this section, we study random splittings of finite-dimensional
projections of the two-dimensional partially forced and damped Euler equation on the periodic box $\T^2=[0,2\pi]^2$.  A full derivation of the splitting below, starting from the infinite-dimensional partial differential equation, can be found in the Appendix.

Let $N\geq 4$, $d=2N(N+2)$ and define 
\begin{align}
\label{eqn:Z2def}
\ZZ^2:= \big\{ \bj=(j_1, j_2)\in \Z^2_{\neq 0} \,: \, \,  0\leq j_1\leq N, \, 0\leq j_2 \leq N\big\}.
\end{align} 
Our phase space for the splitting below is $\R^{d}$ where points $q\in \R^{d}$ are represented by $q=(a,b)$ where $a,b\in \R^{d/2}$ are each indexed by $\ZZ^2$; that is,
\begin{align}
a=(a_\bj)_{\bj\in \ZZ^2} \qquad \text{ and } \qquad b = (b_\bj)_{\bj\in \ZZ^2}.
\end{align}
Using this convention, for $\bj\in \ZZ^2$ we let $e_{a_\bj}$, respectively $e_{b_\bj}$, denote the vector on $\R^{d}$ which is $1$ in the $a_\bj$th entry, respectively $1$ in the $b_\bj$th entry, and $0$ elsewhere.  For any $\bk, \bl \in \ZZ^2$, define real-valued constants $\theta_{\bk\bl}$ by   
\begin{align}\label{eqn:coef}
\theta_{\bk\bl} := \frac{ \bk\cdot \bl^\perp}{4\pi}\bigg(\frac{1}{|\bk|^2}- \frac{1}{|\bl|^2} \bigg),
\end{align}
where $\cdot$ denotes the dot product on $\R^2$ while $\bl^\perp= (l_1, l_2)^\perp:=(l_2, -l_1)$.  For any $\bj,\bk, \bl \in \ZZ^2$ with $\bj+\bk=\bl$ and $q=(a,b)\in \R^{d}$, we define smooth vector fields $V_{a_\bj a_\bk a_\bl}$, $V_{a_\bj b_\bk b_\bl}$, $V_{b_\bj a_\bk b_\bl}$, $V_{b_\bj b_\bk a_\bl}$ on $\R^d$ by  
\begin{align}
\label{def:trips}
V_{a_\bj a_\bk a_\bl}(a,b) &= \theta_{\bk\bl} a_\bk a_\bl e_{a_\bj} + \theta_{\bj\bl} a_\bj a_\bl e_{a_\bk}  - \theta_{\bj\bk} a_\bj a_\bk e_{a_\bl},\\
\nonumber V_{a_\bj  b_\bk b_\bl}(a,b)  &= \theta_{\bk\bl} b_\bk b_\bl e_{a_\bj} +\theta_{\bj\bl} a_\bj b_\bl e_{b_\bk} -\theta_{\bj \bk} a_\bj b_\bk e_{b_\bl},\\
\nonumber V_{b_\bj a_\bk b_\bl}(a,b) &= \theta_{\bk\bl} a_\bk b_\bl e_{b_\bj}  + \theta_{\bj\bl} b_\bj b_\bl e_{a_\bk} -\theta_{\bj\bk} b_\bj a_\bk e_{b_\bl},\\
\nonumber V_{b_\bj  b_\bk a_\bl}(a,b)&=-\theta_{\bk\bl} b_\bk a_\bl e_{b_\bj} - \theta_{\bj\bl} b_\bj a_\bl e_{b_\bk} + \theta_{\bj\bk} b_\bj b_\bk e_{a_\bl}.  
 \end{align} 
As a simple consequence of conservation laws (see Section~\ref{sec:energyt} below), the vector fields above are moreover complete, i.e. belong to $\mathscr{V}_{d}$.

 Let $\mathfrak{m} \in\N$. For $\bj \in \ZZ^2$ and $\ell=1,2,\ldots, \mathfrak{m}$, fix constants
 \begin{align*}
 \lambda_\bj\geq 0  \quad \text{ and } \quad \beta_{a_\bj}^{\ell}, \beta_{b_\bj}^\ell\in  \R. 
 \end{align*}
  We also define vector fields $V_\text{damp}, V_{\ell} \in \mathscr{V}_{d}$, $\ell=1,2,\ldots, \mathfrak{m}$, for $q=(a,b)\in \R^{d}$ by 
 \begin{align}
 V_\text{damp}(a,b)&:= -\sum_{\bj\in \ZZ^2}\lambda_\bj (a_\bj e_{a_\bj}+ b_\bj e_{b_\bj}) \\
\label{def:force} V_{\ell}(a,b)&:= \sum_{\bj \in\ZZ^2} \beta_{a_\bj}^\ell e_{a_\bj}+ \beta_{b_\bj}^\ell e_{b_\bj}.   \end{align}
Letting 
\begin{align}
\mathscr{I}= \{ (\bj, \bk, \bl) \in  (\ZZ^2)^3\,: \, \bl=\bj+\bk, \, \bj \cdot \bk^\perp \neq 0 \},
\end{align}
our sought after splitting in this section is given by 
\begin{align}
\label{eqn:NSEsplit}
\mathscr{S}:= \big\{ V_\text{damp}, V_{1}, \ldots, V_{\mathfrak{m}}\big\}  \cup \bigcup_{(\bj, \bk, \bl) \in \mathscr{I}} \big\{V_{a_\bj a_\bk a_\bl}, V_{a_\bj b_\bk b_\bl}, V_{b_\bj a_\bk b_\bl}, V_{b_\bj b_\bk a_\bl}  \big\}.
\end{align} 
Clearly, $\mathscr{S}$ is a splitting of the sum of its vector fields.  This sum corresponds to Galerkin projections of the two-dimensional, damped and forced, Euler equation on $\T^2$, as derived in the Appendix. 

Throughout, we let 
\begin{align}
\label{eqn:dmodes}
\mathscr{D}= \{ \bj \in \ZZ^2 \, : \, \lambda_{\bj} >0 \}\end{align} 
denote the collection of \emph{damped} modes.

\subsubsection*{Notation}  Below, for $q\in \R^{d}$, $t\geq 0$, and $(\bj,\bk, \bl) \in \mathscr{I}$, we will use the notation 
\begin{align*}
&\varphi_t^{a_\bj a_\bk a_\bl}(q):= \varphi_t^{V_{a_\bj a_\bk a_\bl}}(q), \quad \varphi_t^{a_\bj b_\bk b_\bl}(q):=\varphi_t^{V_{a_\bj b_\bk b_\bl}}(q), \quad \varphi_t^{b_\bj a_\bk b_\bl}(q):= \varphi_t^{V_{b_\bj a_\bk b_\bl}}(q), \\
&\varphi_t^{b_\bj b_\bk a_\bl}(q) := \varphi_t^{V_{b_\bj b_\bk a_\bl}}(q), \quad \varphi_t^{\text{damp}}(q):= \varphi_t^{V_\text{damp}}(q), \quad \varphi_t^{\ell}(q)= \varphi_t^{V_{ \ell }}(q), 
\end{align*}
where $\ell=1,2,\ldots, \mathfrak{m}. $
We offer the convenient abuse of notation by letting $S_{|\mathscr{S}|}$ in this section denote the set of permutations of the set of symbols
\begin{align}
 \big\{\text{damp}, 1, \ldots,  \mathfrak{m}\big\}\cup\bigcup_{(\bj, \bk, \bl) \in \mathscr{I}}\big\{ a_\bj a_\bk a_\bl, \, a_\bj b_\bk b_\bl, \, b_\bj a_\bk b_\bl, \, b_\bj b_\bk a_\bl \big\}. \end{align}
With this abuse of notation, $\sigma_1, \sigma_2, \ldots, \sigma_j, \ldots$ will denote elements of $S_{|\mathscr{S}|}$ with $u_1, u_2, \ldots, \\u_j, \ldots$ being independent, uniformly distributed elements on $S_{|\mathscr{S}|}$. We also offer slight abuses of notation (see the $\beta^\ell$ and $\Lambda$ in~\eqref{eqn:vort} in the appendix) by letting $\beta^\ell\in \R^{d}$ be the vector
\begin{align}
\label{eqn:betaL}
\beta^\ell= \sum_{j\in \ZZ^2}\beta_{a_\bj}^\ell e_{a_\bj} + \beta_{b_\bj}^\ell e_{b_\bj}
\end{align}
and letting $\Lambda$ be the $d\times d$ diagonal matrix defined by
\begin{align}
\Lambda e_{a_\bj} \cdot e_{a_\bj}= \Lambda e_{b_\bj} \cdot e_{b_\bj} = \lambda_\bj, \,\,\, \bj\in \ZZ^2.  
\end{align}

\subsection{Statement of the main result}

In order to connect with the setup in Section~\ref{sec:notation} and state the main result to be proven in this section, let $H:\R^{d}\rightarrow [1, \infty)$ be given by
\begin{align}
\label{def:NSH}
 H(q)= \sqrt{\sum_{\bj\in \ZZ^2} a_\bj^2+ b_\bj^2} +1, \,\,\,\, q=(a,b) \in \R^d. 
\end{align}
We first establish subconservation and partial dissipation of the splitting. 
\begin{lemma}
\label{lem:scpdNSE}
Suppose that the set of damped modes $ \mathscr{D}$ as in~\eqref{eqn:dmodes} is nonempty, and let $\eta \in (0,1)$.  Then the splitting $\mathscr{S}$ as in~\eqref{eqn:NSEsplit} is:
\begin{itemize}
\item[(i)] subconservative with respect to $H$ with corresponding measurable functions as in~\eqref{eqn:subcon} given by 
\begin{align*}
  &F_\ell(t)= |\beta^\ell|t,\,\,  \ell=1,\ldots, \mathfrak{m}, \quad \text{ and }\\
  &F_{\text{\emph{damp}}}(t) =F_{a_\bj a_\bk a_\bl}(t)=F_{a_\bj b_\bk b_\bl}(t)= F_{b_\bj a_\bk b_\bl}(t)=F_{b_\bj b_\bk a_\bl}(t)\equiv 0,
\end{align*}
for all $(\bj,\bk,\bl) \in \mathscr{I}$.  
\item[(ii)] partially dissipative with respect to $H$ with index $\text{\emph{damp}}$ on the set 
\begin{align}
D_\eta  = \Big\{ q\in \R^{d}\, : \,  \sum_{\bj \in \mathscr{D}} |q_\bj|^2 \geq \eta |q|^2 \Big\}.
\end{align}
\end{itemize}
\end{lemma}
\begin{proof}
For part (i), note that for any $(\bj,\bk, \bl)\in \mathscr{I}$, the dynamics along the vector fields $V_{a_\bj a_\bk a_\bl}$, $V_{a_\bj b_\bk b_\bl}$, $V_{b_\bj a_\bk b_\bl}$, and $V_{b_\bj b_\bk a_\bl}$ conserves $H(q)$, implying 
\begin{align*}
F_{a_\bj a_\bk a_\bl}(t)=F_{a_\bj b_\bk b_\bl}(t)= F_{b_\bj a_\bk b_\bl}(t)=F_{b_\bj b_\bk a_\bl}(t)=0
\end{align*}
for all $t\geq 0$. 
Also, for any $q\in \R^{d}$, $\ell=1,2,\ldots, \mathfrak{m}$ and $t\geq 0$ we have 
\begin{align*}
H( \varphi_t^{\text{damp}}(q))  = \big| e^{-\Lambda t} q  \big|+1 \leq H(q) \quad \text{ and }\quad  H(\varphi_t^{\ell}(q)) = |q+ t \beta^\ell | \leq H(q) + t |\beta^\ell|.
\end{align*} 
This establishes part (i). 

To establish (ii),  observe that for any $t\geq 0$ and any $q\in \R^{d}$ with $q\in D_\eta$
\begin{align*}
H(\varphi_t^{\text{damp}}(q))= | e^{-\Lambda t}q | +1 &=  \sqrt{\sum_{\bj\in\mathscr{D}} (e^{-\lambda_\bj t}-1)|q_\bj|^2 + |q|^2} +1 \\
&\leq  \sqrt{\sum_{\bj\in\mathscr{D}} (e^{-\lambda t}-1)|q_\bj|^2 + |q|^2} +1 \\
& \leq |q|\sqrt{(e^{-\lambda t} -1)\eta +1}+1 \leq \alpha_{\text{damp}}(t) H(x)+ 1\end{align*}
where $\alpha_{\text{damp}}(t) = \sqrt{(e^{-\lambda t} -1)\eta +1}$.  \end{proof}

In order to avoid certain invariant sets in the splitting, we need further structure in our forcing terms.  Specifically,  for any $\bj, \bk, \bl\in \mathscr{Z}^2$  and any $\beta\in \R^{d}$ let 
\begin{align}
\label{def:NRC}
\Delta_{\bj\bk\bl}^1(\beta)&:= (\beta_{a_\bj})^2 \bigg( \frac{1}{|\bj|^2}- \frac{1}{|\bk|^2}\bigg) - (\beta_{a_\bl})^2\bigg( \frac{1}{|\bk|^2}- \frac{1}{|\bl|^2}\bigg), \\ 
\nonumber \Delta_{\bj\bk\bl}^2(\beta)&:=(\beta_{a_\bj})^2 \bigg( \frac{1}{|\bj|^2}- \frac{1}{|\bk|^2}\bigg) -( \beta_{b_\bl})^2\bigg( \frac{1}{|\bk|^2}- \frac{1}{|\bl|^2}\bigg).\end{align}
\begin{definition}
We say that the forcing $V_1, V_2, \ldots, V_\mathfrak{m}$ as in~\eqref{def:force} in the splitting $\mathscr{S}$ is \emph{non-resonant for} $\bj\in \mathscr{Z}^2$, denoted by $\bj \in \mathscr{F}$, if there exists $\ell=1,2, \ldots, \mathfrak{m}$ such that for any $\bk, \bl \in \mathscr{Z}$ with $\bj+\bk=\bl$ and $|\bj|\neq |\bk|$ we have 
\begin{align}
\beta^\ell_{a_\bj}\neq 0, \quad \Delta_{\bj\bk\bl}^1(\beta^\ell)\neq 0, \quad \text{ and } \quad \Delta_{\bj\bk\bl}^2(\beta^\ell)\neq0.\end{align}
\end{definition}

\begin{remark}
\label{rem:triad}
Non-resonant forcing will be used below to ensure energy transfer for the dynamics along the triads $V_{a_\bj a_\bk a_\bl}, V_{a_\bj b_\bk b_\bl}, V_{b_\bj a_\bk b_\bl}, V_{b_\bj b_\bk a_\bl}$ for $(\bj, \bk, \bl) \in \mathscr{I}$. In particular, there are ``bad" initial conditions in the associated dynamics where the energy concentrates in one mode and does not transfer energy to the other two.  Such initial conditions exist even at large values of the phase space.  The use of non-resonant forcing will allow us to push the dynamics sufficiently far away from such initial states.  See Remark~\ref{rem:tripinitialdata}, Figure~1 and Figure~2 below for more details of this point. 
\end{remark}

We next provide some simple conditions which ensure that the forcing $V_1, \ldots, V_\mathfrak{m}$ is non-resonant for a given $\bj \in \mathscr{Z}^2$.  

\begin{example}
Suppose that $\beta^1$ as in~\eqref{eqn:betaL} is of the form 
\begin{align*}
\beta^1= \beta_{a_\bj}^1e_{a_\bj} 
\end{align*}
 for some $\beta_{a_\bj}^1\neq0$. In particular, the only non-zero term in~\eqref{eqn:betaL} with $\ell=1$ is $\beta^1_{a_\bj}$. It then follows that the forcing $V_1, \ldots, V_\mathfrak{m}$ is non-resonant for $\bj$; that is, $\bj \in \mathscr{F}$.   \end{example}
\begin{example}
Suppose that $\bj, \bk\in \mathscr{Z}^2$ are such that $|\bj|=|\bk|$ and $\bj\neq \bk$.  If $\beta^1$ as in~\eqref{eqn:betaL} has the form 
\begin{align*}
\beta^1= \beta_{a_\bj}^1 e_{a_\bj}+ \beta^1_{a_\bk} e_{a_\bk}
\end{align*}
for some $\beta_{a_\bj}^1\neq0$ and $\beta_{a_\bk}^1\neq 0$, then the forcing $V_1, V_2, \ldots, V_\mathfrak{m}$ is non-resonant for $\bj$ and $\bk$; that is, $\bj,\bk \in \mathscr{F}$. 
\end{example}

Throughout this section, we will employ the following assumptions on $\mathscr{D}$ and $\mathscr{F}$. 
\begin{assumption}
\label{assump:SNS}
One of the following conditions is met:
\begin{itemize}
\item[(DF1)]  $\{(0,1), (1,0 )\} \subset \mathscr{F}$ and $\{(1,0), (0,1), (N,N)\} \subset \mathscr{D}.$ 
\item[(DF2)]  $\{(0,1), (1,0), (1,1) \} \subset\mathscr{F}$ and either $\{ (1,0), (N, N) \} \subset \mathscr{D}$ or  $\{ (0,1), (N, N) \} \subset \mathscr{D}$.\end{itemize} 
\end{assumption}

\begin{remark}
Note that Assumption~\ref{assump:SNS} (DF1) means that the forcing $V_1, \ldots, V_\mathfrak{m}$ in the splitting $\mathscr{S}$ is non-resonant for the lowest modes $(0,1)$ and $(1,0)$, and that the lowest and highest modes are damped, i.e. $(1,0), (0,1), (N,N) \in \mathscr{D}$.   
\end{remark}

\begin{remark}
Assumption~\ref{assump:SNS} aligns with study of out-of-equilibrium fluxes in turbulence.  The damping assumption on high frequencies corresponds to the effect of viscosity on high Fourier modes.  Damping on low modes corresponds to the effect of friction at large scales.  Damping on the large scales is needed to prevent energy build up since our model is on the two-dimensional periodic box. The absence of dissipation on the intermediate modes, which should be thought of as the systems' inertial range in turbulence theory, is important here because we want to study energy transfer between low modes and high modes created by the nonlinearity.  Since the system is dissipative, we must force the system to keep it out of the zero state.  We choose to only force at the large scales to emphasize that we are interested in studying in how energy is transferred through the inertial range (i.e. the middle modes) to the dissipative scale (i.e. the high modes).  \end{remark}

\begin{theorem}
\label{thm:SNS}
Let $h\in (0,1)$.  Suppose that Assumption~\ref{assump:SNS} holds and let $\psi(x):= \log x$ for $x\geq 2$ and $ \psi(x):= \log 2$ for $x\leq 2$.  Then there exist constants $c=c(N, h)>0, C=C(N,h)>0 $ such that the global estimate holds
\begin{align}
\label{eqn:SNSmainbound1}
\mathcal{P} H -H \leq - \frac{c}{\psi(H)} H + C.  
\end{align}  
\end{theorem}

\begin{remark}
The difference between the conclusions of Theorem~\ref{thm:SNS} and the analogous result for Lorenz '96 (Theorem~\ref{thm:main1}) is the presence of the logarithmic correction in the bound~\eqref{eqn:SNSmainbound1}.  Recall that in Section 2, this correction leads to stretched exponential moments, versus exponential moments in the case of Lorenz '96, of the first return time to a large compact set in the phase space.  At the level of the dynamics of the random splitting for Euler, even if one perturbs initial conditions using the forcing away from the ``bad" states (cf. Remark~\ref{rem:triad}), the triad dynamics can still spend the bulk of its time with energy concentrated in a single mode.  Moreover, the fraction of time it spends in the thermalized state where each mode is non-zero, order 1 (as $H\rightarrow \infty$) is at least proportional to $1/|\log H|$ as $H\rightarrow \infty$.  Contrasting this with the random splitting for Lorenz '96, provided the dynamics along $V_i$ as in~\eqref{eqn:vf1} moves at order 1 speed (which happens by using the forcing), the dynamics spends a positive, order 1 time in the thermalized state.  See Remark~\ref{rem:tripinitialdata}, Figure 1 and Figure 2 below for further details of these points.                 
\end{remark}

\begin{remark}
Although the logarithmic correction in~\eqref{eqn:SNSmainbound1} is tantalizing, it is unclear whether a better estimate, e.g. the same qualitative one as in the case of Lorenz '96, is possible for the random splitting of Euler.  Indeed, it may be possible to utilize more complicated dynamical structures than used in the proof below to produce a stronger bound.    
\end{remark}

In order to prove Theorem~\ref{thm:SNS}, combining Lemma~\ref{lem:scpdNSE} with Theorem~\ref{thm:check} it suffices to show the following result.  Below, recall the definition of $A_n^\ell(D)$ given in~\eqref{eqn:union}.
\begin{theorem}  
\label{thm:NSElb}
Let $h\in (0,1)$ and suppose that Assumption~\ref{assump:SNS} is satisfied.  Then there exists constants $p_*>0, R\geq 2$ and $\eta \in (0,1)$ such that 
\begin{align}
\PP_q( A_0^\text{\emph{damp}}(D_\eta)) \geq \frac{p_*}{\log(H(q))}
\end{align} 
for all $q\in H_{> R}$. 
\end{theorem}

In order to prove Theorem~\ref{thm:NSElb}, we need to first establish some key properties of the dynamics defined in~\eqref{def:trips}. Similar to the case of Lorenz '96 in Section~\ref{sec:lor}, we will provide conditions under which energy is transferred among the directions in these vector fields.  All of this analysis is carried out in Section~\ref{sec:prelimob}, Section~\ref{sec:Case1} and Section~\ref{sec:energyt}.  Finally, given the facts deduced in these sections, in Section~\ref{sec:proofpropNSE} we conclude Theorem~\ref{thm:NSElb}.

\subsection{Preliminary observations}  
\label{sec:prelimob}

%\textcolor{red}{Just use the name stochastic triad, refer to Lyapunov paper with you and andrea agazzi.  Euler top see Remark 3.5 or previous two papers for further discussion.}
%
%\textcolor{red}{SInce on finite box, enstrophy could accumulate at low/modes }
Note first that the dynamics along every triad $V_{a_\bj a_\bk a_\bl}$, $V_{a_\bj b_\bk b_\bl}$, $V_{b_\bj a_\bk b_\bl}$, $V_{b_\bj b_\bk a_\bl}$ with $(\bj,\bk, \bl) \in \mathscr{I}$  can be written in the form 
\begin{align}
\label{eqn:trips2}
\dot{X}&= \theta_{\bk\bl} Y Z, \qquad \dot{Y}= \theta_{\bj\bl} XZ, \qquad \dot{Z}= -\theta_{\bj\bk} XY. 
\end{align} 
This is clear for $V_{a_\bj a_\bk a_\bl}$, $V_{a_\bj b_\bk b_\bl}$ and $V_{b_\bj a_\bk b_\bl}$.  For $V_{b_\bj b_\bk a_\bl}$, we simply let $X= b_\bj,  Y=  b_\bk$ and $Z=- a_\bl$ to arrive at~\eqref{eqn:trips2}.  Note that equation~\eqref{eqn:trips2} is the equation of an Euler spinning top, which has two conserved quantities making it integrable.  See Section~\ref{sec:energyt} below and \cite[Remark 3.5]{AMM_22} for further details.   

It will be useful to rescale~\eqref{eqn:trips2} by $H=H(q)=H(a,b)$, and we assume below that $H\geq 3$.  Because $H$ is conserved by the dynamics along every vector field in $\mathscr{S} - \{ V_\text{damp}, V_{1}, \ldots,V_{\mathfrak{m}}  \}$, we treat it as a constant when analyzing~\eqref{eqn:trips2}.  For simplicity, we define $\delta = 1/H$ and set $x=\delta X, y=\delta Y, z=\delta Z$ where $X,Y,Z$ are as in~\eqref{eqn:trips2}.  Note then that the triple $(x,y,z)$ satisfies the equation
\begin{align}
\label{eqn:trips3}
\dot{x}= \frac{\theta_{\bk \bl}}{\delta} yz, \qquad \dot{y}= \frac{\theta_{\bj\bl}}{\delta} xz, \qquad \dot{z}=- \frac{\theta_{\bj\bk}}{\delta}xy.
\end{align}  
Under certain assumptions on the initial data as well as the indices $\bj,\bk$, throughout this section we study the behavior of equation~\eqref{eqn:trips3} as $\delta \rightarrow 0$, i.e. as $H\rightarrow \infty$.

We break the analysis apart according to two cases; namely, when $|\bk|=|\bj|$ or $|\bk|>|\bj|$.  Both cases will always be under the hypotheses that $\bl= \bj+\bk$ and $ \bj \cdot \bk^\perp \neq 0$.   

\subsection{Case 1: $(\bj, \bk, \bl ) \in \mathscr{I}$ with $|\bk|=|\bj|$} 
\label{sec:Case1}
In this case, $\theta_{\bj\bk}=0$ and $\theta_{\bk\bl}=-\theta_{\bj\bl}\neq 0$.  Thus equation~\eqref{eqn:trips3} reduces to
\begin{align}
\label{eqn:NS2}
\dot{x}= -\frac{\theta_{\bj\bl}}{\delta} yz, \qquad \dot{y} =\frac{\theta_{\bj\bl}}{\delta}xz, \qquad  \dot{z}=0. 
\end{align}
The unique solution of~\eqref{eqn:NS2} with initial condition $(x_0, y_0, z_0)\in \R^3$ is explicitly given by 
\begin{align}
\label{eqn:solNS2}
      x_t = \rho_0 \cos\big( \tfrac{\theta_{\bj\bl}z_0}{\delta} t + \theta_* \big),  \qquad y_t= \rho_0 \sin\big( \tfrac{\theta_{\bj\bl}z_0}{\delta} t + \theta_* \big) \qquad \text{ and } \qquad z_t=z_0, \,\,\, t\geq 0,
      \end{align}
      where $\rho_0= \sqrt{x_0^2+y_0^2}$ and $\theta_* \in [0, 2\pi)$ is the unique  point for which 
      \begin{align*}
      (x_0, y_0) = \rho_0 ( \cos \theta_*,  \sin \theta_*).  
      \end{align*}
Recalling that $\tau \sim \exp(1/h)$ with $h\in (0,1)$, using these expressions we will prove the following lemma. 
\begin{lemma}
\label{lem:spinners}
Suppose that $(x_t, y_t, z_t)$ solves~\eqref{eqn:NS2} with initial condition $(x_0, y_0, z_0)$ satisfying $|(x_0, y_0, z_0)|\leq 1$.  If $\max\{ |x_0|, |y_0| \} \geq c_0$ and $|z_0|\geq c_1 \delta$ for some constants $c_0>0, c_1 >0$ independent of $\delta$, then with probability $q=q(c_0, h)>0$ depending only on $c_0, h$ we have
\begin{align*}
\min \{ |x_\tau|, |y_\tau|\} \geq c_0/2.    
\end{align*}
%\begin{itemize}
%\item[(i)]  If $|x_0| \geq c_0$ for some $c_0\in (0, 1] $ independent of $\delta$, then   
%with probability $q_i=q_i(c_0, h)>0$ depending only on $c_0, h$ we have $ |x_\tau| \geq c_0/\sqrt{2}$.
%\item[(ii)] If $|y_0| \geq c_0$ for some $c_0\in (0,1]$ independent of $\delta$, then with probability $q_{ii}=q_{ii}(c_0, h)>0$ depending only on $c_0, h$ we have $ |y_\tau| \geq c_0/2$.
%\item[(iii)] \end{itemize} 

\end{lemma}

\begin{proof}
%The proof of the second part (ii) is nearly identical to the proof of part (i), so we only prove parts (i) and (iii). 
%
%For part (i), suppose first that $|z_0| \leq \epsilon \delta/|\theta_{j\bl}|$ for some small, to be determined, parameter $\epsilon >0$.  Then on the event $\{ \tau \leq h \}$, which has probability $1-e^{-1}$, we have that 
%\begin{align*}
%\Big| \sin \bigg( \frac{\theta_{j\bl} z_0 \tau}{\delta}\bigg)\Big| \leq \frac{|\theta_{j\bl} z_0| h}{\delta}< \epsilon. 
%\end{align*} 
%Hence, using sum-angle formulas, we find that 
%\begin{align*}
%x_\tau= x_\tau- \rho_0 \cos (\theta_*) \cos\bigg( \frac{\theta_{j\bl} z_0 \tau}{\delta}\bigg) + O(\epsilon) = x_\tau- \rho_0 \cos \theta_* + O(\epsilon),
%\end{align*}
%where the constant in the $O(\epsilon)$ is independent of $h, c_0>0$.  Therefore, by picking $\epsilon >0$ small enough, but independent of $h, c_0$, we find that on the event $\{ \tau \leq h \}$:
%\begin{align*} 
% |x_\tau| \geq \frac{\rho_0 |\cos \theta_*| }{\sqrt{2}}\geq \frac{c_0}{\sqrt{2}}.
% \end{align*}
% 
% On the other hand, suppose now that $|z_0| \geq \epsilon \delta/|\theta_{j\bl}|$ where the $\epsilon>0$ is as chosen above in the first case of part (i).  Because the argument is similar otherwise, we suppose without loss of generality that $\theta_{j\bl} z_0 \geq \epsilon \delta$.  

We suppose without loss of generality that $\theta_{\bj\bl} z_0 >0$ as the other case is similar.  Consider the event 
 \begin{align}
 \label{eqn:A_2}
B_4=B_4(\bj,\bl):= \bigcup_{k=1}^\infty \bigg\{ 2\pi k + \frac{\pi}{6} \leq \frac{\theta_{\bj\bl} z_0 \tau}{\delta} + \theta_* \leq \frac{\pi}{4} + 2\pi k \bigg\}.
 \end{align} 
On $B_4$, we have the desired bounds
 \begin{align*}
 |x_\tau| = \rho_0 \bigg| \cos \bigg(  \frac{\theta_{\bj\bl} z_0 \tau}{\delta} + \theta_*\bigg) \bigg| \geq \frac{\rho_0}{\sqrt{2}}\geq \frac{c_0}{2} \,\,\, \text{ and } \,\,\, |y_\tau| = \rho_0 \bigg| \sin \bigg(  \frac{\theta_{\bj\bl} z_0 \tau}{\delta} + \theta_*\bigg) \bigg|\geq \frac{\rho_0}{2}\geq \frac{c_0}{2}.    
 \end{align*}
 We have left to estimate $\PP(B_4)$.  To this end, note that if $w:= \tfrac{\delta}{\theta_{\bj\bl} z_0 h}$, we have that 
 \begin{align*}
 \PP(B_4) &= \sum_{k=1}^\infty \PP \bigg\{\frac{(2\pi k+\pi/6- \theta_*)\delta}{\theta_{\bj\bl} z_0} \leq \tau \leq \frac{(\pi/4 + 2\pi k -\theta_*)\delta}{\theta_{\bj\bl} z_0} \bigg\} \\
 &= \sum_{k=1}^\infty e^{-\frac{(2\pi k +\pi/6 - \theta_*)\delta}{\theta_{\bj\bl} z_0 h}} - e^{-\frac{(\pi/4+2\pi k - \theta_*)\delta}{\theta_{\bj\bl} z_0 h}} \\
&= e^{\theta_* w} \frac{e^{- 2\pi w}}{1- e^{-2\pi w}}  (e^{-\frac{\pi}{6}w}- e^{-\frac{\pi}{4}w}).\end{align*}
By our assumption, $\theta_{\bj\bl} z_0 \geq c_1 |\theta_{\bj\bl}| \delta$.  Hence, $0<w\leq 1/(c_1 |\theta_{\bj\bl}| h)$.  In particular, $\PP(B_4)>0$ for all such $w$.  On the other hand, as $w\rightarrow 0$   
\begin{align*}
e^{\theta_* w} \frac{e^{- 2\pi w}}{1- e^{-2\pi w}}  (e^{-\frac{\pi}{6}w}- e^{-\frac{\pi}{4}w}) \rightarrow \frac{1}{24}>0.
\end{align*}
In particular, 
\begin{align}
\inf_{0<w\leq \frac{1}{c_1 |\theta_{\bj\bl}| h}} \PP(B_4) >0,
\end{align}
which finishes the proof. 
\end{proof}

\subsection{Case 2: $(\bj, \bk, \bl) \in \mathscr{I}$ with $|\bk| >|\bj|$}

%\subsection{Energy transfer}
\label{sec:energyt}
To see when energy is transferred in~\eqref{eqn:trips3} in this case, first observe that equation~\eqref{eqn:trips3} has two conserved quantities; namely, the \emph{relative energy} and the \emph{relative enstrophy}, which are respectively given by
\begin{align}
\label{eqn:relen}E_{\bj \bk\bl}(x,y,z)&=\frac{x^2}{|\bj|^2}+ \frac{y^2}{|\bk|^2}+ \frac{z^2}{|\bl|^2},\\
\label{eqn:relens}\mathcal{E}(x,y,z) &=x^2+y^2+z^2.
\end{align}
For simplicity, we will often use the shorthand notations 
\begin{align}E=E_{\bj\bk\bl}(x,y,z)
\quad \text{ and } \quad \mathcal{E}= \mathcal{E}(x,y,z).
\end{align}  
Define a constant $\zeta_0(\bj, \bk)$ by 
\begin{align}
\label{eqn:zetadef}
\zeta_{0}(\bj,\bk):= \min\bigg\{  \bigg(\frac{1}{|\bj|^2}- \frac{1}{|\bk|^2} \bigg)\frac{1}{2d} 
,   \bigg(\frac{1}{|\bk|^2}- \frac{1}{|\bl|^2}\bigg) \frac{1}{2d}  \bigg\}>0
\end{align}
we will apply the following assumption on the initial data and the indices.     
\begin{assumption}
\label{assump:TNS}
We have that $(\bj, \bk, \bl) \in \mathscr{I}$ with $|\bk| > |\bj|$ and either \emph{(A1)} or \emph{(A2)} below holds:
\begin{itemize}
\item[(A1)] There exists constants $\xi \in (0,1)$ and $\zeta\in (0,  \zeta_0(\bj, \bk) d)$ independent of $\delta $ such that
\begin{align*}
\xi\leq \mathcal{E}\leq 1 \qquad \text{ and } \qquad \frac{\mathcal{E}}{|\bl|^2} + \zeta \leq E \leq \frac{\mathcal{E}}{|\bk|^2} -\zeta\delta^2.\end{align*}
\item[(A2)] There exists constants $\xi \in (0,1), \zeta\in (0,  \zeta_0(\bj, \bk)d)$ independent of $\delta $ such that  
\begin{align*}
\xi \leq \mathcal{E}\leq 1 \qquad \text{ and } \qquad \frac{\mathcal{E}}{|\bk|^2} + \zeta\delta^2 \leq E \leq \frac{\mathcal{E}}{|\bj|^2} -\zeta.\end{align*}
\end{itemize}
\end{assumption}

\begin{remark}
\label{rem:tripinitialdata}
Although it is suggestive, Assumption~\ref{assump:TNS} is \emph{not} equivalent to the condition that there exists a constant $c>0$ small, independent of $\delta\in (0,1)$ for which 
\begin{align}
\label{eqn:rely2}
c \leq y^2_0\leq \mathcal{E}- c \delta^2.  
\end{align} 
The above relation~\eqref{eqn:rely2} means that, initially, the middle mode $y$ in~\eqref{eqn:trips3} contains some, but not all of the relative enstrophy $\mathcal{E}$.  To see why they are not equivalent, note that the condition $E=\mathcal{E}/|\bk|^2$ is equivalent to the initial data $x_0$ and $z_0$ in~\eqref{eqn:trips3} satisfying
\begin{align}
\label{eqn:F=1/k^2}
x^2_0= z^2_0 \frac{\frac{1}{|\bk|^2}- \frac{1}{|\bl|^2} }{\frac{1}{|\bj|^2}- \frac{1}{|\bk|^2}}.  
\end{align}  
Note that if $y_0$ satisfies~\eqref{eqn:rely2}, then using the fact that $|\bl|^2 > |\bk|^2>|\bj|^2$ we can find $x_0, z_0$ satisfying~\eqref{eqn:F=1/k^2}.  When returning back to the randomly switched system, in order to deal with this family of \emph{bad} initial conditions, we will employ the assumption that $\bj\in \mathscr{F}$.  In particular, this will allow us to ensure, with positive probability, that we can initially perturb the dynamics away from the surface~\eqref{eqn:F=1/k^2} so that for some constant $\zeta'>0$ independent of $\delta$, the resulting coordinates $(x_0, y_0, z_0)$ satisfy either 
\begin{align*}
x^2_0 \bigg(\frac{1}{|\bj|^2}- \frac{1}{|\bk|^2} \bigg) \leq z^2_0\bigg( \frac{1}{|\bk|^2}- \frac{1}{|\bl|^2}\bigg) - \zeta' \delta^2, 
\end{align*}
which is equivalent to $E \leq \tfrac{\mathcal{E}}{|\bk|^2}-\zeta'\delta^2$, or 
     \begin{align*}
x^2_0 \bigg(\frac{1}{|\bj|^2}- \frac{1}{|\bk|^2} \bigg) \geq z^2_0\bigg( \frac{1}{|\bk|^2}- \frac{1}{|\bl|^2}\bigg) + \zeta' \delta^2, 
\end{align*}
which is equivalent to $E \geq \frac{\mathcal{E}}{|k|^2}+ \zeta'\delta^2.$  This will then imply Assumption~\ref{assump:TNS}. 
\end{remark}

\begin{figure} 
\includegraphics[width=8cm, height=6cm]{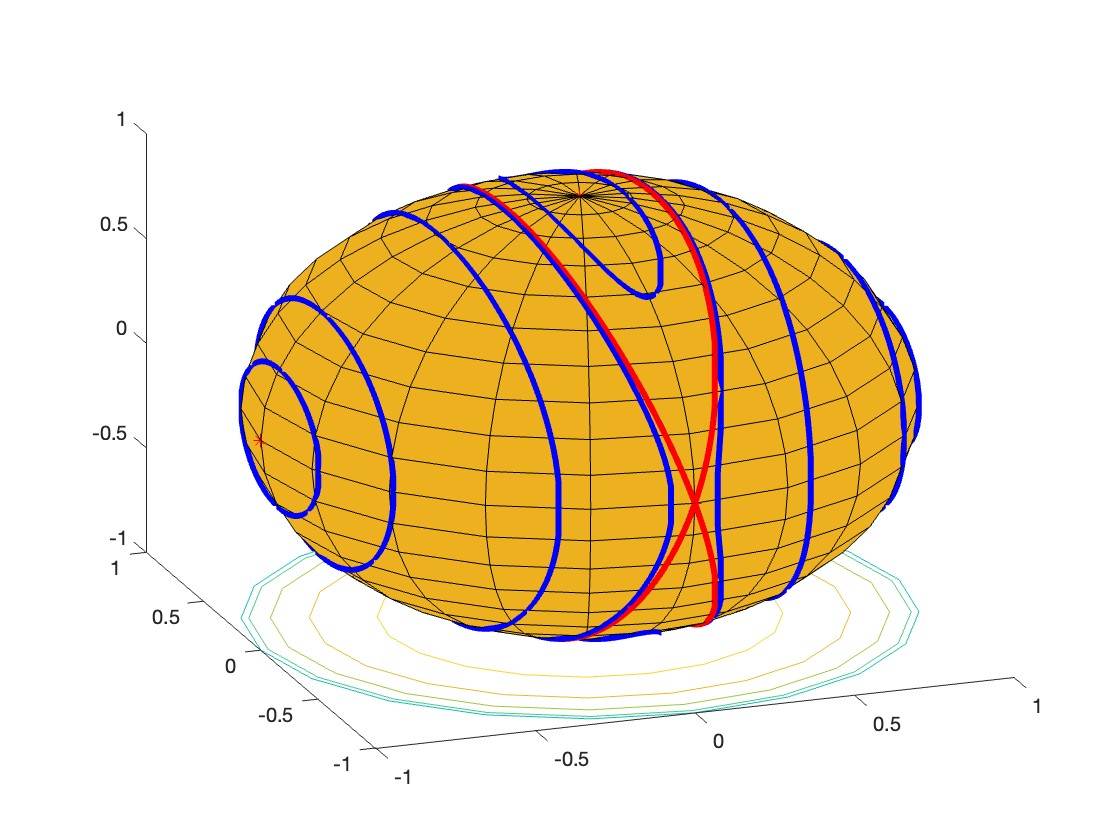}
\caption{Phase portrait of solutions of equation~\eqref{eqn:trips3} initialized on the unit sphere in $\R^3$.  When started on the interface $E= \mathcal{E}/|\bk|^2$ (plotted in red), solutions are not periodic and, in this case, approach one of the equilibria $(0, \pm1, 0)$.  Thus the fraction of time spent in the `thermalized' state where each of the modes $x$, $y$ and $z$ is non-zero, order 1 is of order $1/H$ as $H \rightarrow \infty$.  Away from the interface $E= \mathcal{E}/|\bk|^2$ (plotted in blue), solutions are periodic.  Provided solutions start sufficiently far from this interface (cf. Assumption~\ref{assump:TNS}), the fraction of time in the `thermalized' state is at least order $1/|\log H|$ as $H \rightarrow \infty$. } 
\end{figure}

\begin{figure}
\includegraphics[width=7cm, height=6cm]{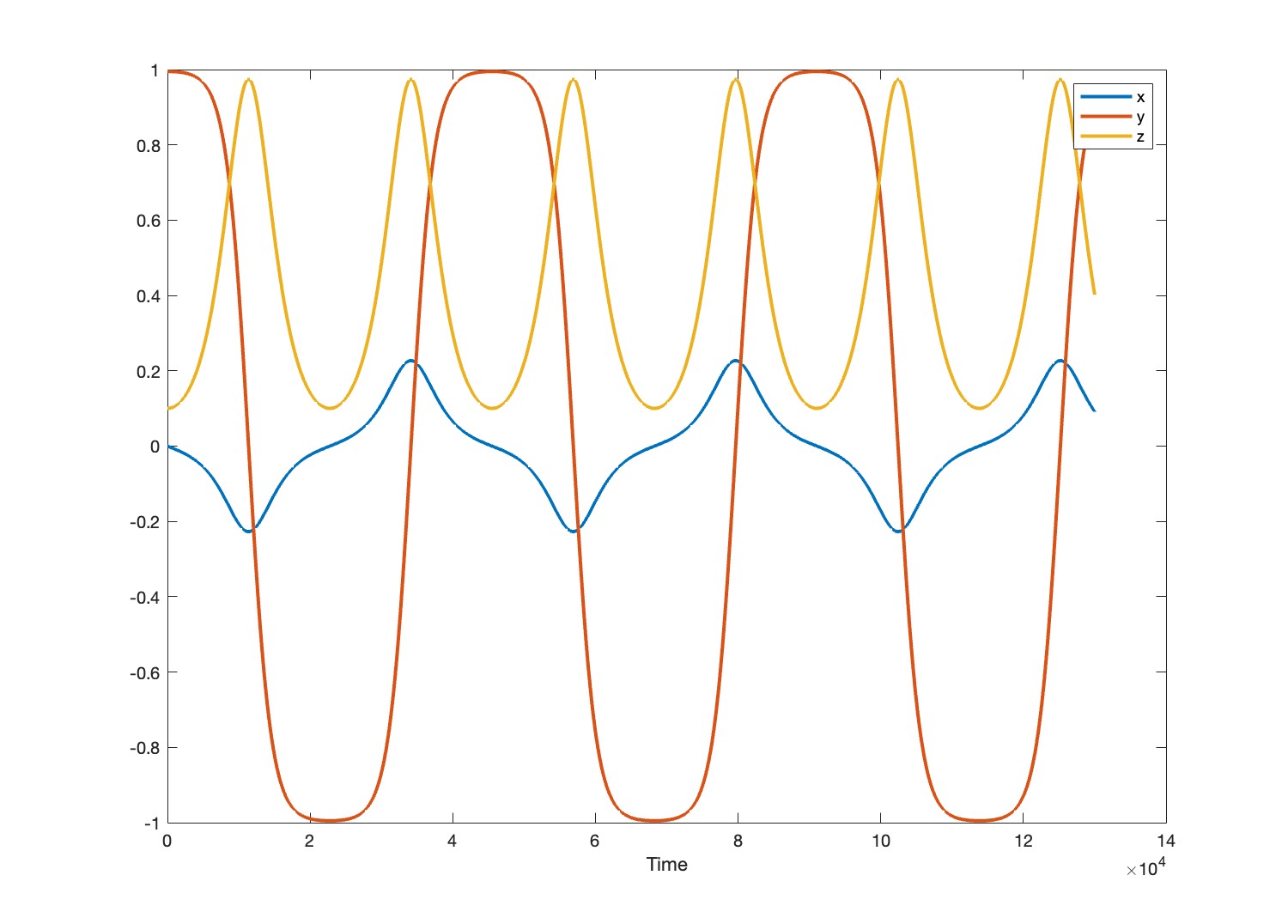}\includegraphics[width=7cm, height=6cm]{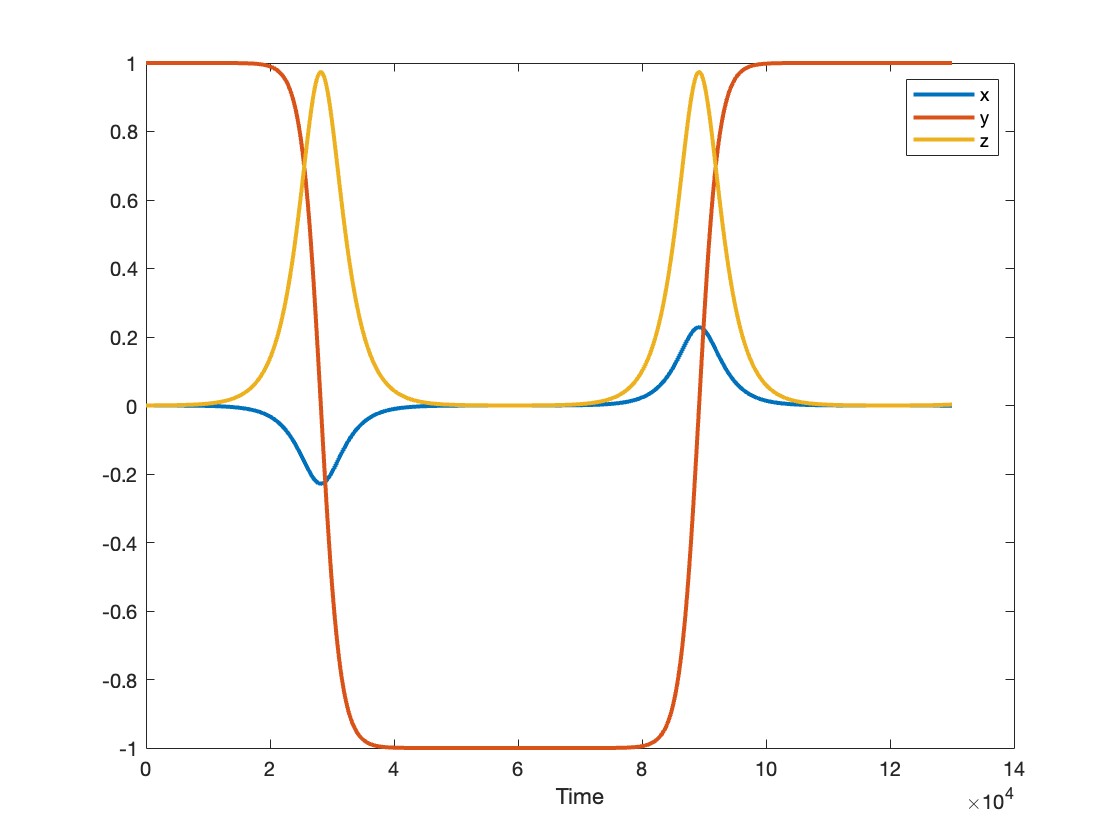}
\includegraphics[width=7cm, height=6cm]{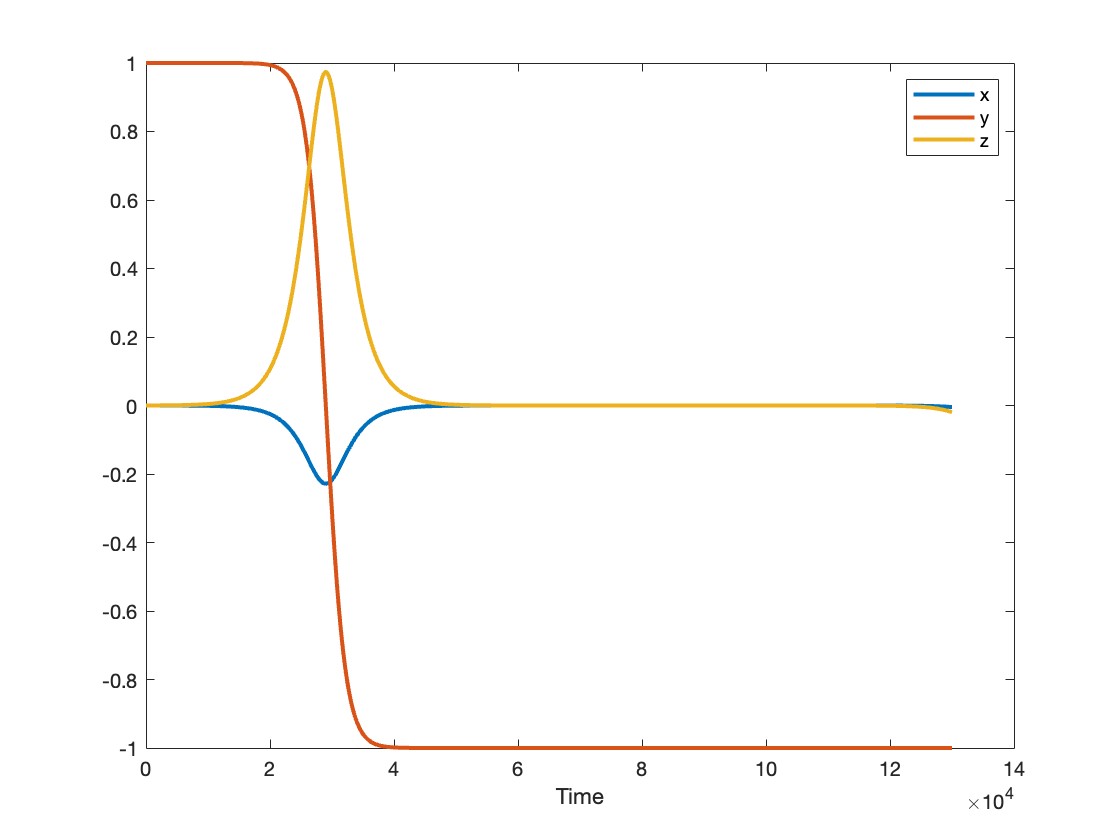}
\caption{Solutions of equation~\eqref{eqn:trips3} plotted over 140 units of time which are progressively closer to the interface $E=\mathcal{E}/|\bk|^2$ from left to right and then down.  As initial data tend to $E=\mathcal{E}/|\bk|^2$, the period of the solutions tends to infinity and the fraction of time spent in the `thermalized' state where $x$, $y$, and $z$ are all non-zero, order $1$ tends to zero.}
\end{figure}

In order to state our main result of this section, for $\eta \in (0,1/2)$ and $\delta \in (0,1)$, consider the following set \begin{align}
B_{\eta, \delta} = \{ s\geq 0 \, : \, \min\{ |x_s|, |y_s |, |z_s |  \} \geq \eta\} \end{align}
where $(x_s,y_s,z_s)$ denotes the solution of~\eqref{eqn:trips3}.  Of course, the solution $(x_s, y_s, z_s)$ depends on $\delta$, but we suppress this dependence for notational simplicity.   
\begin{theorem}[Thermalization]
\label{thm:therm}
Suppose that Assumption~\ref{assump:TNS} is satisfied.  Then there exist constants $\eta\in (0,1/4), c>0$ and $\delta_* \in (0,1)$ depending only on $|\bj|, |\bk|, |\bl|, \zeta, \xi, h$ so that 
 \begin{align}
 \label{eqn:timesel}
 \PP \{ \tau \in B_{\eta, \delta} \} \geq \frac{c}{|\log \delta|} \,\,\, \text{ for all } \delta \in (0, \delta_*). 
 \end{align} 
 \end{theorem}   
 \begin{remark}
 Note that conclusion~\eqref{eqn:timesel} is \emph{thermalization} of the triple $(x_s, y_s, z_s)$ with probability $\geq c/|\log \delta|$.  
 
 %Under Assumption~\ref{assump:TNS} with the same $\eta$ and $\delta_*$ in the statement of the result above, one can also show that 
% \begin{align*}
% \liminf_{m\rightarrow \infty} \frac{1}{m}\int_0^m \mathbf{1}_{B_{\eta, \delta}}(s) \, ds \geq \frac{d}{|\log \delta|}, \quad \delta \in (0, \delta_*),
% \end{align*}
% for some constant $d>0$ depending only on $|\bj|, |\bk|, |\bl|, \zeta, \xi, h$.  That is, the fraction of time on the interval $[0, \infty)$ spent in $B_{\eta, \delta}$ is bounded below by $d/|\log \delta|$.  
% 
 \end{remark}

In order to prove Theorem~\ref{thm:therm}, we first introduce and deduce properties of the so-called \emph{Jacobi elliptic functions}.  As we will see below, these functions can be used to arrive at explicit expressions for solutions of equation~\eqref{eqn:trips3}.     

To begin, for $(\bj,\bk, \bl) \in \mathscr{I}$ with $|\bk| >|\bj|$, we first define the following parameters 
\begin{align}
\kappa_1 = \frac{E-\frac{\mathcal{E}}{|\bl|^2}}{\frac{1}{|\bj|^2}- \frac{1}{|\bl|^2}}, \qquad \gamma_1 = \frac{\frac{1}{|\bk|^2}- \frac{1}{|\bl|^2}}{E- \frac{\mathcal{E}}{|\bl|^2}}, \qquad \kappa_2 = \frac{\frac{\mathcal{E}}{|\bj|^2}-E}{\frac{1}{|\bj|^2}- \frac{1}{|\bl|^2}}, \qquad \gamma_2 = \frac{{\frac{1}{|\bj|^2} - \frac{1}{|\bk|^2}}}{\frac{\mathcal{E}}{|\bj|^2}- E}
\end{align}  
where $E,\mathcal{E}$ are the conserved quantities in~\eqref{eqn:relen}.  Observe that each of these parameters is strictly positive under Assumption~\ref{assump:TNS}.  Furthermore, these parameters arise from relations~\eqref{eqn:relen}-\eqref{eqn:relens} by solving for $x$ and $z$ in terms of $y$.  Indeed, we see that~\eqref{eqn:relen}-\eqref{eqn:relens} together imply
\begin{align}
\label{eqn:xzintermsy}
x^2= \kappa_1(1-\gamma_1 y^2) \qquad \text{ and } \qquad z^2 = \kappa_2(1-\gamma_2 y^2). 
\end{align}  
Next, define
\begin{align}
\gamma_{\text{min}} = \gamma_1 \wedge \gamma_2, \qquad \gamma_{\text{max}} = \gamma_1 \vee \gamma_2, \qquad \text{ and }\qquad \rho= \frac{\gamma_\text{min}}{\gamma_{\text{max}}}.  
\end{align}
One can check that 
\begin{align*}
\frac{\mathcal{E}}{|\bl|^2}<E< \frac{\mathcal{E}}{|\bk|^2} \,\,\, \text{ implies } \,\,\, \gamma_1 >\mathcal{E}^{-1} \text{ and } \gamma_2 <\mathcal{E}^{-1}, \text{ so } \rho= \frac{\gamma_2}{\gamma_1}<1.   
\end{align*}
Similarly, 
\begin{align*}
\frac{\mathcal{E}}{|\bk|^2}<E< \frac{\mathcal{E}}{|\bj|^2} \,\,\, \text{ implies } \,\,\, \gamma_2 >\mathcal{E}^{-1} \text{ and } \gamma_1 < \mathcal{E}^{-1}, \text{ so } \rho= \frac{\gamma_1}{\gamma_2}<1.
\end{align*}

\begin{remark}
To motivate the use of Jacobi elliptic functions with the parameters above, we observe that, formally, if $b:= y \sqrt{\gamma_\text{max}}$ where $y$ is as in~\eqref{eqn:trips3}, then 
\begin{align*}
\dot{b} = \frac{\sqrt{\gamma_\text{max} \kappa_1 \kappa_2} \theta_{\bj\bl}}{\delta} \sqrt{1- \rho b^2}\sqrt{1-b^2}. 
\end{align*}
Since $\rho\in (0,1)$, the above system is explicitly (locally) integrable but we need to be careful about the signs in the square roots.  These signs depend on  the initial data as well as the periodicity in the system~\eqref{eqn:trips3}.  A similar line of reasoning can also be used to arrive at formal explicitly integrable equations for $x$ and $z$ in~\eqref{eqn:trips3}.  
\end{remark} 

Following the previous remark, define a function $T_\rho$ on $[0,1]$ and a number $K_\rho$ by 
\begin{align}
T_\rho(s) = \int_0^s \frac{db}{\sqrt{1-\rho b^2}\sqrt{1-b^2}}   \qquad \text{ and } \qquad K_\rho :=T_\rho(1).
\end{align}
Observe that since $\rho<1$, $T_\rho$ is a well-defined, strictly increasing function on $[0,1]$ with range $[0, K_\rho]$.  Let $\text{sn}_\rho:[0, K_\rho]\rightarrow [0,1]$ denote the inverse of $T_\rho$.   We extend $\text{sn}_\rho$ to the larger interval $[0, 4K_r]$ by defining  
\begin{align}
&\text{sn}_\rho(x):=\text{sn}_\rho(2K_\rho-x), \qquad x\in [K_\rho, 2K_\rho],\\
\nonumber &\text{sn}_\rho(x):=- \text{sn}_\rho(4K_\rho-x) \qquad x\in [2K_\rho, 4K_\rho].
\end{align}
We extend $\text{sn}_\rho$ periodically to all of $\R$ with period $4K_\rho$.  Similarly, the function $\text{cn}_\rho$ is defined initially on the interval $[0, 4K_r]$ by
\begin{align}
\text{cn}_\rho(x)&:= \sqrt{1-\text{sn}_\rho^2(x)}, \qquad x\in [0, K_\rho],\\
\nonumber \text{cn}_\rho(x)&:= -  \text{cn}_\rho(2K_\rho-x), \quad x\in [K_\rho, 2K_\rho],\\
\nonumber \text{cn}_\rho(x)&:= \text{cn}_\rho(4K_\rho-x), \quad x\in [2K_\rho, 4K_\rho]
\end{align}
 and then extended to all of $\R$ to a function which is periodic with period $4K_\rho$.  The function $\text{dn}_\rho$ is defined by
 \begin{align}
 \text{dn}_\rho(x):= \sqrt{1-\rho \text{sn}_\rho(x)^2}, \qquad  x\in \R.  
 \end{align}
 Note that because $\rho<1$, $\text{dn}_\rho$ remains positive for all values of $x$ whereas $\text{sn}_\rho$ and $\text{cn}_\rho$ oscillate between $1$ and $-1$, similar to the usual sine and cosine. The function $\text{sn}_\rho$, 
 $\text{cn}_\rho$, and $\text{dn}_\rho$ are called \emph{Jacobi elliptic functions}.    One can check that (see~\cite{AS_64}) the functions $\text{sn}_\rho$, $\text{cn}_\rho$ and $\text{dn}_\rho$ are differentiable and satisfy the following differential identities  
 \begin{align}
 \text{sn}_\rho'(x)= \text{cn}_\rho(x) \text{dn}_\rho(x), \quad \text{cn}_\rho'(x)= - \text{sn}_\rho(x) \text{dn}_\rho(x), \quad \text{dn}_\rho'(x) =-\rho\text{sn}_\rho(x) \text{cn}_\rho(x).  
 \end{align}
 Using the above facts, we can now deduce explicit expressions involving the Jacobi elliptic functions for solutions of~\eqref{eqn:trips3}.  The explicit expressions vary slightly depending on the value of the relative energy $\mathcal{E}$ and the initial data, as we now note.

\subsubsection*{Assumption \emph{(A1)}}  Suppose that Assumption~\ref{assump:TNS} (A1) is satisfied.  First observe that $\mathcal{E}/|\bl|^2 <E< \mathcal{E}/|\bk|^2$ implies that the initial condition $z_0$ in \eqref{eqn:trips3} is nonzero.  Indeed, if $z_0=0$, then using $|\bl| > |\bk|> |\bj|$ gives 
 \begin{align*}
 E= \frac{x^2_0}{|\bj|^2}+ \frac{y^2_0}{|\bk|^2} \geq \frac{\mathcal{E}}{|\bk|^2}.  
 \end{align*}
 Now, if $z_0 >0$, then we can express $x$, $y$ and $z$ as
\begin{align*}
x_t= \sqrt{\kappa_1} \text{cn}_\rho \bigg[& \frac{\sqrt{\gamma_1 \kappa_1 \kappa_2} \theta_{\bj\bl}}{\delta} t + \theta_0\bigg], \quad y_t= \frac{1}{\sqrt{\gamma_1}} \text{sn}_\rho\bigg[ \frac{\sqrt{\gamma_1 \kappa_1 \kappa_2} \theta_{\bj\bl}}{\delta} t + \theta_0\bigg],\\
&z_t= \sqrt{\kappa_2} \text{dn}_\rho  \bigg[ \frac{\sqrt{\gamma_1 \kappa_1 \kappa_2} \theta_{\bj\bl}}{\delta}t + \theta_0\bigg],
\end{align*}  
where $\theta_0 \in [0, 4K_\rho)$ is the unique point so that 
\begin{align}
x_0= \sqrt{\kappa_1}\text{cn}_\rho(\theta_0)\qquad  \text{ and }\qquad  y_0=\frac{1}{\sqrt{\gamma_1}} \text{sn}_\rho(\theta_0) .  
\end{align} 

On the other hand, if $z_0<0$, then the solution of~\eqref{eqn:trips3} can be written as   
\begin{align*}
x_t= \sqrt{\kappa_1} \text{cn}_\rho \bigg[& \frac{\sqrt{\gamma_1 \kappa_1 \kappa_2} \theta_{\bj\bl}}{\delta} t + \theta_0\bigg], \quad  y_t= -\frac{1}{\sqrt{\gamma_1}} \text{sn}_\rho\bigg[\frac{\sqrt{\gamma_1 \kappa_1 \kappa_2} \theta_{\bj\bl}}{\delta}t + \theta_0\bigg], \\
&z_t= -\sqrt{\kappa_2} \text{dn}_\rho  \bigg[ \frac{\sqrt{\gamma_1 \kappa_1 \kappa_2} \theta_{\bj\bl}}{\delta}t + \theta_0\bigg],
\end{align*}  
where $\theta_0 \in [0, 4K_\rho)$ is the unique point so that 
\begin{align}
x_0= \sqrt{\kappa_1}\text{cn}_\rho(\theta_0)\qquad  \text{ and }\qquad  y_0=-\frac{1}{\sqrt{\gamma_1}} \text{sn}_\rho(\theta_0) .  
\end{align}
Additionally, it is not hard to check that the conditions $\mathcal{E}\in [\xi, 1]$ and $\mathcal{E}/|\bl|^2 +\zeta \leq E \leq \mathcal{E}/|\bk|^2 - \zeta\delta^2$ imply that the parameters $\kappa_1, \gamma_1, \kappa_2, \gamma_2$ are bounded above and below by positive constants that do not depend on $\delta$.  Indeed, we have the following bounds   
\begin{align}
\label{eqn:o1par}\frac{\zeta}{\frac{1}{|\bj|^2} - \frac{1}{|\bl|^2}}&\leq \kappa_1 \leq \frac{\frac{1}{|\bk|^2}- \frac{1}{|\bl|^2}}{\frac{1}{|\bj|^2} - \frac{1}{|\bl|^2}}, \qquad  \,\,\,1\leq \gamma_1 \leq \frac{\frac{1}{|\bk|^2}- \frac{1}{|\bl|^2}}{\zeta},\\
\nonumber \frac{\frac{\xi}{|\bj|^2} - \frac{\xi}{|\bk|^2} }{\frac{1}{|\bj|^2}- \frac{1}{|\bl|^2}}&\leq \kappa_2 \leq \frac{\frac{1}{|\bj|^2} - \frac{1}{|\bl|^2} - \zeta}{\frac{1}{|\bj|^2} - \frac{1}{|\bl|^2}}, \qquad \frac{\frac{1}{|\bj|^2} - \frac{1}{|\bk|^2}}{\frac{1}{|\bj|^2} - \frac{1}{|\bl|^2} -\zeta} \leq \gamma_2 \leq 1/\xi.
\end{align}

\subsubsection*{Assumption \emph{(A2)}}  Suppose now that  Assumption~\ref{assump:TNS} (A2) is satisfied.  Then a similar argument to the one used above implies that $x_0\neq 0$.  If $x_0>0$, then the unique solution of~\eqref{eqn:trips3} is given by 
\begin{align*}
x_t=\sqrt{\kappa}_1 \text{dn}_\rho\bigg[&\frac{ \sqrt{\gamma_2 \kappa_1 \kappa_2} \theta_{\bj\bl}}{\delta} t + \theta_0\bigg],\,\,\,\,y_t= \frac{1}{\sqrt{\gamma_2}}\text{sn}_\rho \bigg[\frac{\sqrt{\gamma_2 \kappa_1 \kappa_2} \theta_{\bj\bl}}{\delta} t + \theta_0\bigg], \\
 &z_t= \sqrt{\kappa_2}\text{cn}_\rho \bigg[\frac{\sqrt{\gamma_2 \kappa_1 \kappa_2} \theta_{\bj\bl}}{\delta} t + \theta_0  \bigg],
\end{align*}
where $\theta_0\in [0, 4K_\rho)$ is defined by 
\begin{align}
y_0= \frac{1}{\sqrt{\gamma_2}} \text{sn}_\rho(\theta_0) \qquad \text{ and } \qquad z_0= \sqrt{\kappa_2} \text{cn}_\rho(\theta_0).  
\end{align}
On the other hand if $x_0<0$, then the unique solution of~\eqref{eqn:trips3} is given by 
\begin{align*}
x_t=-\sqrt{\kappa}_1 \text{dn}_\rho\bigg[&\frac{ \sqrt{\gamma_2 \kappa_1 \kappa_2} \theta_{\bj\bl}}{\delta}  t + \theta_0\bigg],\,\,\,\, y_t=- \frac{1}{\sqrt{\gamma_2}}\text{sn}_\rho\bigg[\frac{ \sqrt{\gamma_2 \kappa_1 \kappa_2} \theta_{\bj\bl}}{\delta} t + \theta_0\bigg], \\
&z_t= \sqrt{\kappa_2}\text{cn}_\rho\bigg[\frac{ \sqrt{\gamma_2 \kappa_1 \kappa_2} \theta_{\bj\bl}}{\delta} t + \theta_0\bigg],
\end{align*}
where $\theta_0\in [0, 4K_\rho)$ is defined by 
\begin{align}
y_0= -\frac{1}{\sqrt{\gamma_2}} \text{sn}_\rho(\theta_0) \qquad \text{ and } \qquad z_0= \sqrt{\kappa_2} \text{cn}_\rho(\theta_0).  
\end{align}

In this case, we also have bounds for the positive parameters $\kappa_1, \gamma_1, \kappa_2, \gamma_2$ analogous to those in~\eqref{eqn:o1par}:
\begin{align}
\label{eqn:o1par2}
\frac{\frac{\xi}{|\bk|^2}- \frac{\xi}{|\bl|^2}}{\frac{1}{|\bj|^2}- \frac{1}{|\bl|^2}} &\leq \kappa_1 \leq \frac{\frac{1}{|\bj|^2} - \frac{1}{|\bl|^2} -\zeta}{\frac{1}{|\bj|^2}- \frac{1}{|\bl|^2}}, \qquad \frac{\frac{1}{|\bk|^2}- \frac{1}{|\bl|^2}}{\frac{1}{|\bj|^2}- \frac{1}{|\bl|^2} -\zeta }\leq \gamma_1 \leq \frac{1}{\xi} ,\\
\nonumber \frac{\zeta}{\frac{1}{|\bj|^2}- \frac{1}{|\bl|^2}}& \leq \kappa_2 \leq \frac{\frac{1}{|\bj|^2}- \frac{1}{|\bk|^2}}{\frac{1}{|\bj|^2}- \frac{1}{|\bl|^2}}, \qquad  1\leq \gamma_2 \leq \frac{\frac{1}{|\bj|^2}- \frac{1}{|\bk|^2}}{\zeta}.  
\end{align}

Given the expressions above for solutions of~\eqref{eqn:trips3} as well as the bounds~\eqref{eqn:o1par} and \eqref{eqn:o1par2}, we are almost ready to prove Theorem~\ref{thm:therm}.  The last ingredient needed is a lower asymptotic bound on the quantity $K_\rho$ in certain regimes in Assumption~\ref{assump:TNS}.  We recall that $K_\rho$ is $1/4$ of the period of the Jacobi elliptic functions $\text{sn}_\rho$ and $\text{cn}_\rho$.  We do this in the following lemma.    
\begin{lemma}
\label{lem:asyp}
Suppose that Assumption~\ref{assump:TNS} is satisfied and, furthermore, 
\begin{align}
E=\frac{\mathcal{E}}{|\bk|^2}+\epsilon(\delta) \qquad \text{ or } \qquad E= \frac{\mathcal{E}}{|\bk|^2}- \epsilon(\delta)
\end{align}
where $\epsilon (\delta) >0$ and $\epsilon(\delta) \rightarrow 0$ as $\delta \rightarrow 0$.  Then as $\delta \rightarrow 0$
\begin{align}
K_\rho = - \frac{\log \epsilon(\delta)}{2} + O(1). 
\end{align}
\end{lemma}

\begin{remark}
If Assumption~\ref{assump:TNS} is satisfied and $E$ is uniformly bounded away from $\mathcal{E}/|\bk|^2$ as $\delta \rightarrow 0$, then the parameter $\rho=\gamma_\text{min}/\gamma_\text{max}<1$ is then uniformly bounded away from $1$ as $\delta \rightarrow 0$.  Hence, $K_\rho$ is order one, positive as $\delta \rightarrow 0$.  Thus the analysis in the lemma above focuses on the case when $E\rightarrow \mathcal{E}/|\bk|^2$ as $\delta \rightarrow 0$.   
\end{remark}

\begin{proof}[Proof of Lemma~\ref{lem:asyp}]
Let $E= \mathcal{E}/|\bk|^2 \pm \epsilon(\delta)$ where $\epsilon(\delta) >0$ satisfies $\epsilon(\delta) \rightarrow 0$ as $\delta \rightarrow 0$.  Without loss of generality, we study the case when $E= \mathcal{E}/|k|^2 - \epsilon(\delta)$ as the other case is nearly identical.  Observe that as $\delta \rightarrow 0$
\begin{align*}
\rho= \frac{\gamma_\text{min}}{\gamma_\text{max}}=\frac{ \frac{1}{1+ \frac{\epsilon(\delta)}{\frac{\mathcal{E}}{|\bj|^2}- \frac{\mathcal{E}}{|\bk|^2}}}}{\frac{1}{1- \frac{\epsilon(\delta)}{\frac{\mathcal{E}}{|\bk|^2}- \frac{\mathcal{E}}{|\bl|^2}}}} = 1 - \bigg(\frac{1}{\mathcal{E}(\frac{1}{|\bj|^2} -\frac{1}{|\bk|^2})} + \frac{1}{\mathcal{E}(\frac{1}{|\bk|^2} -\frac{1}{|\bl|^2})} \bigg)  \epsilon(\delta) + O(\epsilon^2(\delta)).
\end{align*}
Using the asymptotics in~\cite{Ki_15}, we thus obtain
\begin{align*}
K_\rho = \int_0^1 \frac{db}{\sqrt{1-\rho b^2}\sqrt{1-b^2}}\, db = - \frac{\log \epsilon(\delta)}{2} +O(1)
\end{align*}
as $\delta \rightarrow 0$.  
\end{proof}

Given the previous result, we now turn to the proof of Theorem~\ref{thm:therm}.  

\begin{proof}[Proof of Theorem~\ref{thm:therm}]
Set $C= \theta_{\bj\bl} \sqrt{\gamma_\text{max} \kappa_1 \kappa_2}$ and assume without loss of generality that $C>0$.  Observe that for $\eta>0$, $\delta_*>0$ small enough depending only on $|\bj|, |\bk|, |\bl|, \zeta, \xi$, there exists $\eta'>0$ with $\eta'<1-\eta'$ depending only on $|\bj|, |\bk|, |\bl|, \zeta, \xi$ such that $\delta \in (0, \delta_*)$ implies
\begin{align*}
\nonumber \P\{ \tau \in B_{\eta, \delta} \} = \int_{B_{\eta, \delta}} \frac{1}{h} e^{-t/h} \, dt & \geq \int_0^\infty \mathbf{1}\{ \eta' \leq | \text{sn}_\rho(C t/\delta + \theta_0) | \leq 1- \eta' \} \frac{1}{h} e^{-t/h} \, dt \\
\nonumber &= \frac{\delta}{C}e^{\theta_0 \frac{\delta}{Ch}} \int_{\theta_0}^\infty \mathbf{1}\{ \eta' \leq | \text{sn}_\rho(t)| \leq 1- \eta' \} \frac{1}{h} e^{-\frac{\delta t}{C h}} \, dt\\
& \geq \frac{\delta}{C}e^{\theta_0 \frac{\delta}{Ch}} \sum_{j=1}^\infty \int_{4j K_\rho}^{(4j+1)K_\rho} \mathbf{1}\{ \eta' \leq  \text{sn}_\rho(t) \leq 1- \eta' \} \frac{1}{h} e^{-\frac{\delta t}{C h}} \, dt\\
& \geq  \frac{\delta}{C}e^{\theta_0 \frac{\delta}{Ch}}  \sum_{j=1}^\infty e^{-\frac{\delta K_\rho}{Ch}(4j+1)}\int_{4j K_\rho}^{(4j+1)K_\rho} \mathbf{1}\{ \eta' \leq  \text{sn}_\rho(t) \leq 1- \eta' \}  dt\\
&= \frac{\delta}{C}\exp\Big( \frac{\theta_0\delta}{Ch}\Big) (T_\rho(1-\eta')- T_\rho (\eta'))  \sum_{j=1}^\infty e^{-\frac{\delta K_\rho}{Ch}(4j+1)} .
\end{align*}
Next, we note that 
\begin{align*}
 &\frac{\delta}{C}\exp\Big( \frac{\theta_0\delta}{Ch}\Big) (T_\rho(1-\eta')- T_\rho(\eta'))  \sum_{j=1}^\infty e^{-\frac{\delta K_\rho}{Ch}(4j+1)}\\
 \qquad &= \frac{\frac{\delta}{C}}{1- e^{-4\delta K_\rho/Ch}}\exp \Big(\frac{\theta_0\delta}{Ch} - \frac{5\delta K_\rho}{Ch}\Big)[T_\rho(1-\eta') - T_\rho(\eta')].
 \end{align*}
 Using Lemma~\ref{lem:asyp} and the definition of $T_\rho(s)$ we see that there exists $c_{1}=c_1(h)>0$ independent of $\delta \in (0, \delta_*)$ for which 
 \begin{align}
 \label{eqn:tausel1}
 \exp \Big(\frac{\theta_0\delta}{Ch} - \frac{t\delta K_\rho}{Ch}\Big)[T_\rho(1-\eta') - T_\rho(\eta')]\geq c_{1}.  
 \end{align}
 Also, using Lemma~\ref{lem:asyp}, we find that as $\delta \rightarrow 0$, 
 \begin{align*}
 1- e^{-4\delta K_\rho/Ch} = \frac{4\delta K_\rho}{Ch} + O(\delta^2 K_\rho^2).  
 \end{align*}
 Hence, there exists $\delta_* >0$ small enough depending only on $|\bj|, |\bk|, |\bl|, \zeta, \xi, h$ and a constant $c_{2}=c_2(h)>0$ independent of $\delta \in (0, \delta_*)$ for which 
 \begin{align}
 \label{eqn:tausel2}
 \frac{\frac{\delta}{C}}{1- e^{-4\delta K_\rho/Ch}} \geq \frac{c_{2}}{|\log \delta|}  \,\,\, \text{ for all } \delta \in (0, \delta_*).  
  \end{align}
  Setting $c = c_{1} c_{2}$ and combining~\eqref{eqn:tausel1} with~\eqref{eqn:tausel2} completes the proof.  
\end{proof}

\subsection{Proof of Theorem~\ref{thm:NSElb}}
\label{sec:proofpropNSE}
Given the analysis of the transfer of energy in the triples~\eqref{eqn:splittingE} of the previous section, we are now ready to prove the main result of this section, Theorem~\ref{thm:NSElb}.  Due to its length, the proof will be split into two parts depending on what assumption on the forced and damped modes, either (DF1) or (DF2), is satisfied.     

\begin{proof}[Proof of Theorem~\ref{thm:NSElb} under Assumption~\ref{assump:SNS} (DF1)]
Suppose that Assumption~\ref{assump:SNS} (DF1) is satisfied and let $R>3$ and $q=(a,b)\in H_{>R+1}$ so that $|q| >R$.  Then there exists $\bk \in \ZZ^2$ such that either  
\begin{align*}
|a_\bk| \geq \frac{|q|}{\sqrt{d}} \,\, \,\,\text{ or } \,\,\,\, |b_\bk| \geq \frac{|q|}{\sqrt{d}}.  
\end{align*}  
Without loss of generality, we suppose that $|a_\bk| \geq |q|/\sqrt{d}$ and let $S=S_{|\mathscr{S}|}$ for simplicity.

If $\bk \in \mathscr{D}$, then $q\in D_{\frac{1}{d}}$ and hence
\begin{align*}
\PP_q(A_0^{\text{damp}}(D_{1/d})) \geq \PP( u_{10}= \text{damp}) = \frac{1}{|S|}>0. 
\end{align*}  
On the other hand, if $\bk \notin \mathscr{D}$ then $\bk \notin \{ (1,0), (0,1), (N,N) \}$ and hence $|\bk |> 1$.  In this case, either $(1,0)+\bk \in \ZZ^2$ or $(0,1)+ \bk \in \ZZ^2$.  Let $\bj\in \{ (1,0), (0,1) \}$ be such that $\bl=\bj + \bk \in \ZZ^2$.  

For notational simplicity in what follows, let $H_0=H(q)$, $\delta_0 = 1/H_0$, $x_0=a_\bj/H_0$, $y_0=a_\bk/H_0$ and $z_0= a_\bl/H_0$.  Observe that 
\begin{align}
\label{eqn:initialenergyb}
\xi_0:= \frac{1}{2d}\leq \mathcal{E}(x_0, y_0,z_0) \leq 1.
\end{align}
Furthermore, 
\begin{align}
\label{eqn:propSNS1}
E_{\bj\bk\bl}(x_0, y_0, z_0)&\geq \frac{\mathcal{E}(x_0, y_0, z_0)}{|\bl|^2}+ \bigg(\frac{1}{|\bk|^2}- \frac{1}{|\bl|^2}\bigg) y^2_0 \\
\nonumber &\geq \frac{\mathcal{E}(x_0, y_0, z_0)}{|\bl|^2}+  \bigg(\frac{1}{|\bk|^2}- \frac{1}{|\bl|^2}\bigg) \xi_0  
\end{align}
and 
\begin{align}
\label{eqn:propSNS2}
E_{\bj\bk\bl}(x_0, y_0, z_0)&\leq \frac{\mathcal{E}(x_0, y_0, z_0)}{|\bj|^2} - \bigg(\frac{1}{|\bj|^2}- \frac{1}{|\bk|^2} \bigg)y^2_0\\
\nonumber & \leq \frac{\mathcal{E}(x_0, y_0, z_0)}{|\bj|^2} - \bigg(\frac{1}{|\bj|^2}- \frac{1}{|\bk|^2} \bigg)\xi_0.  
\end{align}
In particular, if $\zeta_0=\zeta_0(\bj, \bk)$ as in~\eqref{eqn:zetadef} we have that 
\begin{align}
\label{eqn:F_0int}
\frac{\mathcal{E}(x_0, y_0, z_0)}{|\bl|^2}+ \zeta_0 \leq E_{\bj\bk\bl}(x_0, y_0, z_0) \leq \frac{\mathcal{E}(x_0, y_0, z_0)}{|\bj|^2} - \zeta_0.
\end{align}

\emph{Case 1}: The quantity $E_{\bj\bk\bl}(x_0, y_0, z_0)$ does \emph{not} belong to the interval 
\begin{align}
\bigg[\frac{\mathcal{E}(x_0, y_0, z_0)}{|\bk|^2}- \zeta_0 \delta_0, \frac{\mathcal{E}(x_0, y_0, z_0)}{|\bk|^2} + \zeta_0 \delta_0\bigg]. 
\end{align}
In particular, in this case Assumption~\ref{assump:TNS} is satisfied, where in the statements $(x_0, y_0, z_0)$ replaces the triple $(x,y,z)$ and $\delta_0$ replaces $\delta$.  Applying Theorem~\ref{thm:therm} and using the fact that $\bj\in \mathscr{D}$, we find that there exists constants $\eta_{1}=\eta_{1}(|\bj|, |\bk|, |\bl|, d, h)>0$, $c_{1}= c_1(|\bj|, |\bk|, |\bl|, d, h)>0$, $\delta_{1}(|\bj|, |\bk|, |\bl|,d, h)\in (0,1/4)$ such that $R \geq 1/\delta_{1}$ implies   
\begin{align}
\label{eqn:diss1}
\PP( \varphi_{\tau_{10}}^{a_\bj a_\bk a_\bl}(q) \in D_{\eta_{1}}) \geq  \frac{c_{1}}{|\log H(q)|}.
\end{align}
In particular, using independence we obtain 
\begin{align*}
\PP_q (A_0^{\text{damp}}(D_{\eta_1}))&\geq \PP\Big( \varphi_{\tau_{10}}^{a_\bj a_\bk a_\bl}(q) \in D_{\eta_{1}}, u_{10}= a_\bj a_\bk a_\bl , u_{20}=\text{damp} \Big) \\ 
& \geq \frac{c_1}{\log |H(q)|}\frac{1}{|S| (|S|-1)}. \end{align*}

\emph{Case 2}: The quantity $E_{\bj\bk\bl}(x_0, y_0, z_0)$ belongs to the interval 
\begin{align}
\bigg[\frac{\mathcal{E}(x_0, y_0, z_0)}{|\bk|^2}- \zeta_0 \delta_0, \frac{\mathcal{E}(x_0, y_0, z_0)}{|\bk|^2} + \zeta_0 \delta_0\bigg]. 
\end{align}
This part of the argument is more involved because we need to first force the random dynamics out of this interval. Since $\bj\in \mathscr{F}$, we may suppose without loss of generality that $\beta^1$ has $\Delta_{\bj \bk\bl}^1(\beta^1)>0$ where $\Delta_{\bj\bk\bl}^1$ was defined in~\eqref{def:NRC}.  The argument when $ \Delta_{\bj \bk\bl}^1(\beta^1)<0$ is similar.   
Let 
\begin{align*}
q'=(a',b'):= \varphi_{\tau_{10}}^{1}(q),\end{align*}
and notice that by definition of $V_{1}$
\begin{align*}
q'= q + \beta^1 \tau_{10}= (a_\bn + \beta_{a_\bn}^1\tau_{10}, b_\bn + \beta_{b_\bn}^1 \tau_{10})_{\bn\in \ZZ^2}.  
\end{align*}
In this proof, to connect with the notation used in the analysis of~\eqref{eqn:trips3} above, we offer a convenient abuse of notation and set $H:=H(q')$, $\delta =1/H$, $x= a_\bj'/H$, $y= a_\bk'/H$, $z=a_\bl'/H$.  

First, suppose furthermore that for some $C>0$ 
\begin{align}
\label{eqn:hyp1}
a_\bj \beta_{a_\bj}^1\bigg( \frac{1}{|\bj|^2}- \frac{1}{|\bk|^2}\bigg) - a_\bl \beta_{a_\bl}^1 \bigg( \frac{1}{|\bk|^2}- \frac{1}{|\bl|^2}\bigg) \geq - C.   
\end{align}
Then we have \begin{align}
  E_{\bj\bk\bl}(x,y,z) &=\frac{H_0^2}{H^2}\bigg\{E_{\bj\bk\bl}(x_0, y_0, z_0) +  2  \delta_0^2 \tau_{10}\bigg( \frac{a_\bj \beta_{a_\bj}^1}{|\bj|^2}+ \frac{a_\bk \beta_{a_\bk}^1}{|\bk|^2} + \frac{a_\bl \beta_{a_\bl}^1}{|\bl|^2}\bigg)  \\
 \nonumber &\qquad \qquad \qquad  + \delta_0^2 \tau_{10}^2 E_{\bj\bk\bl}(\beta_{a_\bj}^1, \beta_{a_\bk}^1, \beta_{a_\bl}^1) \bigg\} \\
\nonumber & \geq \frac{H_0^2}{H^2}\bigg\{\frac{\mathcal{E}(x_0, y_0, z_0)}{|\bk|^2} - \zeta_0 \delta_0^2 +  2  \delta_0^2 \tau_{10}\bigg( \frac{a_\bj \beta_{a_\bj}^1}{|\bj|^2}+ \frac{a_\bk \beta_{a_\bk}^1}{|\bk|^2} + \frac{a_\bl \beta_{a_\bl}^1}{|\bl|^2}\bigg)\\
\nonumber &  \qquad \qquad+ \delta_0^2 \tau_{10}^2 E_{\bj\bk\bl}(\beta_{a_\bj}^1, \beta_{a_\bk}^1, \beta_{a_\bl}^1)\bigg\}.
\end{align}
Hence  on the event 
\begin{align*}
B_{\bj\bk\bl}^1:=\left\{ 1+\frac{2C+2 \zeta_0}{ \Delta_{\bj\bk\bl}^1(\beta^1)} \leq \tau_{10} \leq 2 +\frac{2C+2 \zeta_0}{ \Delta_{\bj\bk\bl}^1(\beta^1)}  \right\},
\end{align*}
we have
\begin{align}
  \label{prop:case21}E_{\bj\bk\bl}(x,y,z) &\geq \frac{\mathcal{E}(x,y,z)}{|\bk|^2}- \zeta_0 \delta^2 + 2\delta^2 \tau_{10} \bigg\{a_\bj \beta_{a_\bj}^1\bigg( \frac{1}{|\bj|^2}- \frac{1}{|\bk|^2}\bigg) - a_\bl \beta_{a_\bl}^1 \bigg( \frac{1}{|\bk|^2}- \frac{1}{|\bl|^2}\bigg) \bigg\}  \\
&\qquad \nonumber + \delta^2 \tau^2_{10}   \Delta_{\bj\bk\bl}^1(\beta^1)\\
\nonumber & \geq \frac{\mathcal{E}(x,y,z)}{|\bk|^2} - \zeta_0 \delta^2 -2C \delta^2 \tau_{10} + \delta^2 \tau^2_{10}  \Delta_{\bj\bk\bl}^1(\beta^1)  \geq \frac{\mathcal{E}(x,y,z)}{|\bk|^2} + \zeta_0 \delta^2.
 \end{align}
 Also, provided we choose $R>0$ large enough so that 
\begin{align}
\label{eqn:r_*choice}
|q|\geq R \geq \frac{4|\beta^1|}{\xi_0} \bigg(2+ \frac{2C+2 \zeta_0}{ \Delta_{\bj\bk\bl}^1(\beta^1)} \bigg) +1 ,
\end{align}  
  we have on $B^1_{\bj \bk \bl}$
\begin{align}
\label{eqn:HH_0control}
\frac{H_0}{H} = \frac{|q|+1}{|q+\beta^1 \tau_{10}| +1} \geq \frac{|q|+1}{|q|+ |\beta^1| \tau_{10}+1}= \frac{1}{1+ \frac{|\beta^1| \tau_{10}}{|q|+1}} \geq \frac{1}{2}.
\end{align}
Furthermore, on $B_{\bj\bk\bl}^1$ we also have by~\eqref{eqn:initialenergyb} and  Cauchy-Schwarz  
\begin{align}
\label{eqn:Eineq1} 1&\geq \mathcal{E}(x,y,z) \\
&= \frac{H_0^2}{H^2} \bigg\{ \mathcal{E}(x_0, y_0, z_0)+ 2 \tau_{10} \delta_0^2 (a_\bj \beta_{a_\bj}^1 + a_\bk \beta_{a_\bk}^1 + a_\bl  \beta_{a_\bl}^1)  + \tau_{10}^2 \delta_0^2 \mathcal{E}(\beta_{a_\bj}^1,\beta_{a_\bk}^1, \beta_{a_\bl}^1) \bigg\}\\
\label{eqn:Eineq2}& \geq \frac{H_0^2}{H^2}\bigg\{ \xi_0 - |q|^{-1} 2\Big( 2 + \tfrac{2C+ 2\zeta_0}{\Delta_{\bj\bk\bl}^1(\beta^1)}\Big)|\beta^1|\bigg\} \geq \frac{\xi_0}{16} = :\xi_1, \,\,\,\text{ and}\\
\label{eqn:Eineq3} |y|&=\frac{H_0}{H} |y_0 + \tau_{10} \beta_{a_\bk}^1/H_0| \geq \tfrac{1}{2} \sqrt{\xi_0} - \frac{\tau_{10} |\beta^1|}{|q|+1} \geq \frac{\sqrt{\xi_0}}{4}
\end{align}
where we used the choice~\eqref{eqn:r_*choice} and~\eqref{eqn:HH_0control}.  Using the same argument as in~\eqref{eqn:propSNS2}, we also have on $B_{\bj\bk\bl}^1$   
\begin{align}
\label{eqn:calEineq1}
E_{\bj\bk\bl}(x,y,z) \leq \frac{\mathcal{E}(x,y,z)}{|\bj|^2} - \bigg(\frac{1}{|\bj|^2}- \frac{1}{|\bk|^2} \bigg) y^2\leq \frac{\mathcal{E}(x,y,z)}{|\bj|^2} - \bigg(\frac{1}{|\bj|^2}- \frac{1}{|\bk|^2} \bigg) \xi_1.
\end{align}
Let 
\begin{align*}
\zeta_1:= \min\bigg\{ \xi_0,  \bigg(\frac{1}{|\bj|^2}- \frac{1}{|\bk|^2} \bigg) \xi_1, \bigg(\frac{1}{|\bk|^2}- \frac{1}{|\bl|^2} \bigg) \xi_1 \bigg\}. 
\end{align*}
Then, under the hypothesis~\eqref{eqn:hyp1}, we have on the event $B_{\bj\bk\bl}^1$
\begin{align*}
\xi_1\leq \mathcal{E}(x,y,z) \leq 1 \qquad \text{ and } \qquad  \frac{\mathcal{E}(x,y,z)}{|\bk|^2} + \zeta_1 \delta^2 \leq E_{\bj\bk\bl}(x,y,z) \leq \frac{\mathcal{E}(x,y,z)}{|\bj|^2}- \zeta_1. 
\end{align*}
In particular, Assumption~\ref{assump:TNS} (A1) is satisfied for $q' = \varphi_{\tau_{10}}^{1}(q)$ on $B_{\bj\bk\bl}^1$.  Applying Theorem~\ref{thm:therm},      
there exist constants $\eta_{2} \in (0,1/4), c_{2}>0, \delta_{2}\in (0,1/4)$ depending only on $|\bj|,|\bk|, |\bl|, d, h$ such that $R \geq 1/\delta_{2}$ 
\begin{align*}
\PP\big(\varphi^{a_\bj a_\bk a_\bl}_{\tau_{20}} \circ \varphi_{\tau_{10}}^{1}(q)\in D_{\eta_{2}}\big)  \geq \frac{c_{2}}{|\log H(q)|}.
\end{align*}
Hence using independence
\begin{align*}
\PP_q( A_0^{\text{damp}}(D_{\eta_2})) &\geq \PP\big(\varphi^{a_\bj a_\bk a_\bl}_{\tau_{20}} \circ \varphi_{\tau_{10}}^{1}(q)\in D_{\eta_{2}}, u_{10}=1, u_{20}=a_\bj a_\bk a_\bl, u_{30}= \text{damp} \big)\\
&\geq \frac{c_2}{|\log H(q)|} \frac{1}{|S| (|S|-1)(|S|-2)}. 
\end{align*}

On the other hand, suppose now that 
\begin{align}
\label{eqn:hyp2}
&\frac{\mathcal{E}(x_0, y_0, z_0)}{|\bk|^2}- \zeta_0 \delta_0^2\leq E_{\bj\bk\bl}(x_0, y_0, z_0) \leq \frac{\mathcal{E}(x_0, y_0, z_0)}{|\bk|^2}+ \zeta_0 \delta_0^2 \qquad \text{ and }\\
\nonumber  &\qquad \qquad \qquad a_\bj \beta_{a_\bj}^1\bigg( \frac{1}{|\bj|^2}- \frac{1}{|\bk|^2}\bigg) - a_\bl \beta_{a_\bl}^1 \bigg( \frac{1}{|\bk|^2}- \frac{1}{|\bl|^2}\bigg)  \leq -C.  
\end{align}
Picking $C=\zeta_0 + \Delta_{\bj\bk\bl}^1(\beta^1)$, we consider the event
\begin{align}
B_{\bj\bk\bl}^2:=\bigg\{ 1 \leq \tau_{10} \leq  2 \bigg\}.
\end{align}
Observe that
\begin{align}
\label{prop:case22} E_{\bj\bk\bl}(x,y,z) &=   \frac{H_0^2}{H^2}\bigg\{E_{\bj\bk\bl}(x_0, y_0, z_0) +  2  \delta_0^2 \tau_{10}\bigg( \frac{a_\bj \beta_{a_\bj}^1}{|\bj|^2}+ \frac{a_\bk \beta_{a_\bk}^1}{|\bk|^2} + \frac{a_\bl \beta_{a_\bl}^1}{|\bl|^2}\bigg)  \\
\nonumber &\qquad \qquad + \delta_0^2 \tau_{10}^2 E_{\bj\bk\bl}(\beta_{a_\bj}^1, \beta_{a_\bk}^1, \beta_{a_\bl}^1) \bigg\} \\
\nonumber & \leq \frac{H_0^2}{H^2}\bigg\{\frac{\mathcal{E}(x_0, y_0, z_0)}{|\bk|^2} + \zeta_0 \delta_0^2 +  2  \delta_0^2 \tau_{10}\bigg( \frac{a_\bj \beta_{a_\bj}^1}{|\bj|^2}+ \frac{a_\bk \beta_{a_\bk}^1}{|\bk|^2} + \frac{a_\bl \beta_{a_\bl}^1}{|\bl|^2}\bigg)\\\nonumber &\qquad \qquad+ \delta_0^2 \tau_{10}^2 E_{\bj\bk\bl}(\beta_{a_\bj}^1, \beta_{a_\bk}^1,\beta_{a_\bl}^1)\bigg\}.
\end{align}
Thus on $B_{\bj\bk\bl}^2$ we have 
\begin{align}E_{\bj\bk\bl}(x,y,z) &\leq \frac{\mathcal{E}(x,y,z)}{|\bk|^2}+ \zeta_0 \delta^2 + 2\delta^2 \tau_{10} \bigg\{a_\bj \beta_{a_\bj}^1\bigg( \frac{1}{|\bj|^2}- \frac{1}{|\bk|^2}\bigg) - a_\bl \beta_{a_\bl}^1 \bigg( \frac{1}{|\bk|^2}- \frac{1}{|\bl|^2}\bigg) \bigg\}  \\
\nonumber &\qquad\qquad+ \delta^2 \tau^2_{10}   \Delta_{\bj\bk\bl}^1(\beta^1)\\
\nonumber &\leq  \frac{\mathcal{E}(x,y,z)}{|\bk|^2}+ \zeta_0 \delta^2 - 2C\delta^2 \tau_{10}  + \delta^2 \tau^2_{10}   \Delta_{\bj\bk\bl}^1(\beta^1) \leq \frac{\mathcal{E}(x,y,z)}{|\bk|^2} - \zeta_0 \delta^2 \end{align}
and 
\begin{align}
\label{eqn:HH_0control1}
\frac{H_0}{H} = \frac{|q|+1}{|q+\beta^1 \tau_{10}|+1} \geq \frac{|q|+1}{|q|+ |\beta^1| \tau_{10}+1}= \frac{1}{1+ \frac{|\beta^1| \tau_{10}}{|q|+1}}  \geq \frac{1}{2},
\end{align}
where we used the choice of $R$ in~\eqref{eqn:r_*choice}.  Moreover, on $B_{\bj\bk \bl}^2$ we also have~\eqref{eqn:Eineq1}-\eqref{eqn:Eineq2} by the same arguments used above.  Lastly, we note that on $B_{\bj\bk\bl}^2$ we have by the same arguments in~\eqref{eqn:Eineq3} and~\eqref{eqn:calEineq1}
\begin{align*}
E_{\bj\bk\bl}(x, y, z)\geq \frac{\mathcal{E}(x, y, z)}{|\bl|^2}+ \bigg(\frac{1}{|\bk|^2}- \frac{1}{|\bl|^2}\bigg) y^2 \geq \frac{\mathcal{E}(x_0, y_0, z_0)}{|\bl|^2}+  \bigg(\frac{1}{|\bk|^2}- \frac{1}{|\bl|^2}\bigg) \xi_1.  
\end{align*}
In particular, we have shown that on the event $B^2_{\bj\bk\bl}$ 
\begin{align*}
\xi_1\leq \mathcal{E}(x,y,z) \leq 1 \qquad \text{ and } \qquad  \frac{\mathcal{E}(x,y,z)}{|\bk|^2} + \zeta_1  \leq \mathcal{E}_{\bj\bk\bl}(x,y,z) \leq \frac{\mathcal{E}(x,y,z)}{|\bj|^2}- \zeta_1 \delta^2. 
\end{align*}
In particular, Assumption~\ref{assump:TNS} (A2) is satisfied for $q' = \varphi_{\tau_{10}}^{1}(q)$ on $B_{\bj\bk\bl}^2$.  Applying Theorem~\ref{thm:therm} and adjusting the constants $\eta_{2}, c_{2}, \delta_{2}$ accordingly, it follows that if $R\geq 1/\delta_{2}$ we have    
\begin{align*}
\PP(\varphi^{a_\bj a_\bk a_\bl}_{\tau_{20}} \circ \varphi_{\tau_{10}}^{1}(q)\in D_{\eta_{2}}) \geq  \frac{c_{2}}{|\log H(q)|}.
\end{align*}
Hence using independence
\begin{align*}
\PP_q(A_0^{\text{damp}} (D_{\eta_2})) &\geq \PP\Big(\varphi^{a_\bj a_\bk a_\bl}_{\tau_{20}} \circ \varphi_{\tau_{10}}^{1}(q)\in D_{\eta_{2}}, u_{10}=1, u_{20}=a_\bj a_\bk a_\bl, u_{30}=\text{damp}\Big) \\
& \geq  \frac{c_{2}}{|\log H(q)|}\frac{1}{|S| (|S|-1)(|S|-2)}.
\end{align*}
This finishes the proof. 

\end{proof}

\begin{proof}[Proof of Theorem~\ref{thm:NSElb} under Assumption~\ref{assump:SNS} (DF2)]
Without loss of generality, throughout this argument we suppose that $\{ (1,0), (N,N)\}\subset \mathscr{D}$ as the other case is similar.  Suppose $q=(a,b)\in H_{>R+1}$ so that $|q|\geq R$ for some $R>3$ sufficiently large to be determined below and suppose without loss of generality that $a_\bk$, $\bk \in \ZZ^2$, has
\begin{align*}
|a_\bk| \geq  \frac{|q|}{\sqrt{d}}. 
\end{align*} 
The case when $|b_\bk| \geq |q|/\sqrt{d}$ proceeds in a nearly identical fashion.  Below, we only provide the details in cases of initial data whose arguments are different than those considered under Assumption~\ref{assump:SNS} (DF1).    

\emph{Case 1}.  Suppose that $\bj=(0,1)$, $\bl=\bj+\bk\in \ZZ^2$ and 
\begin{align*}
q=(a,b) \notin \bigg\{ \frac{\mathcal{E}(a_\bj, a_\bk, a_\bl)}{|\bk|^2}+ \zeta_0(\bj,\bk)  \leq E_{\bj\bk\bl}(a_\bj, a_\bk, a_\bl) \leq \frac{\mathcal{E}(a_\bj, a_\bk, a_\bl)}{|\bk|^2} - \zeta_0 (\bj,\bk)\bigg\}.
\end{align*}
Using the arguments in the proof of Theorem~\ref{thm:NSElb}~(DF1) \emph{Case 1}, it follows that there exists constants $\eta_{1}=\eta_{1}(|\bj|, |\bk|, |\bl|,d, h)>0$, $c_{1}= c_1(|\bj|, |\bk|, |\bl|, d, h)>0$, $\delta_{1}(|\bj|, |\bk|, |\bl|,d , h)>0$ such that $R \geq 1/\delta_{1}$ implies   
\begin{align}
\label{eqn:diss1}
\PP( |\varphi_{\tau_{10}}^{a_\bj a_\bk a_\bl}(q) \cdot e_{a_\bj}| \geq \eta_1 |H(q)|) \geq \frac{c_{1}}{|\log H(q)|}.
\end{align}
Note that the difference here under (DF2) versus (DF1) is that $\bj$ is not necessarily damped, so $\varphi_{\tau_{10}}^{a_\bj a_\bk a_\bl}(q)$ may not belong to $D_{\eta_1}$.  However, letting $\bl':=(1,1)$ and noting that $\bl' \in \mathscr{F}$, we have that $\beta_{a_{\bl'}}^\ell\neq 0$ for some $\ell=1,2,\ldots, \mathfrak{m}$.  Suppose without loss of generality that $\beta_{a_{\bl'}}^1>0 $ and define 
\begin{align*}
q'=(a',b')=  \varphi_{\tau_{10}}^{a_\bj a_\bk a_\bl}(q)\quad \text{and} \quad 
q''=(a'',b'')=\varphi_{\tau_{20}}^{1}\circ q'.
\end{align*} 
If $a_{\bl'}\geq -1$, then on the event 
\begin{align*}
\bigg\{  \frac{2}{\beta_{a_{\bl'}}^1} \leq \tau_{20} \leq \frac{4}{\beta_{a_{\bl'}}^1}\bigg\}\cap \bigg\{  |\varphi_{\tau_{10}}^{a_\bj a_\bk a_\bl}(q) \cdot e_{a_\bj}| \geq \eta_1 |H(q)|\bigg\}
\end{align*}
and ensuring $R>0$ is large enough so that 
\begin{align*}
R \geq \frac{8}{\eta_1}\frac{|\beta^1|}{\beta^1_{a_{\bl'}}}  +1
\end{align*}
we have 
\begin{align*}
&a_{\bl'}''= a_{\bl'}+ \beta^1_{a_{\bl'}} \tau_{20} \geq - 1 + 2 =1, \quad \frac{H(q'')}{H(q')}\geq \frac{|q'|+1- |\beta^1| \tau_{20}}{|q'| +1} \geq 1- \frac{|\beta^1| \tau_{20}}{|q|} \geq 1/2,\\
&\frac{H(q')}{H(q'')} \geq \frac{1}{1+ |\beta^1| \tau_{20}/H(q)} \geq \frac{1}{2}, \quad |a_{\bj}''| = |a_{\bj}' + \beta^1_{a_\bj} \tau_{20}| \geq \eta_1 H(q) - |\beta^1| \tau_{20} \geq \frac{\eta_1}{4} H(q'').  
\end{align*}
Thus letting $\bk'=(1,0)$ and applying Lemma~\ref{lem:spinners}, we find that there exists $\eta_3, c_3>0$ small enough such that for all $R>0$ large enough \begin{align}
\PP\Big(\varphi_{\tau_{30}}^{a_{\bj} a_{\bk'} a_{\bl'}} \circ \varphi_{\tau_{20}}^{1}\circ \varphi_{\tau_{10}}^{a_\bj a_\bk a_\bl}(q)\in D_{\eta_3}\Big) \geq  \frac{c_3}{|\log(H(q))|}.
\end{align}
In particular,
 \begin{align*}
 &\PP_q\Big( A_0^\text{damp}(D_{\eta_{3}})\Big)\\
 &\geq \PP\Big(\varphi_{\tau_{30}}^{a_{\bj} a_{\bk'} a_{\bl'}} \circ \varphi_{\tau_{20}}^{1}\circ \varphi_{\tau_{10}}^{a_\bj a_\bk a_\bl}(q)\in D_{\eta_3}, u_{10}=a_\bj a_\bk a_\bl, u_{20}=1, u_{30}= a_{\bj} a_{\bk'} a_{\bl'}\Big)\\
 & \geq \frac{c_3}{|\log H(q)|} \frac{1}{|S| (|S|-1)(|S|-2)}. \end{align*}
One can apply a similar argument in the case when $a_{\bl'}\leq -1$ to obtain the same conclusion, provided we adjust $\eta_3, c_3$ and $R$ accordingly.  

\emph{Case 2}: Suppose that $\bj=(0,1)$, $\bl=\bj+\bk\in \ZZ^2$ and 
\begin{align*}
q=(a,b) \in \bigg\{ \frac{\mathcal{E}(a_\bj, a_\bk, a_\bl)}{|\bk|^2}+ \zeta_0(\bj,\bk) \leq E_{\bj\bk\bl}(a_\bj, a_\bk, a_\bl) \leq \frac{\mathcal{E}(a_\bj, a_\bk, a_\bl)}{|\bk|^2} - \zeta_0 (\bj,\bk)\bigg\}.
\end{align*}
 In this case, we slightly modify the arguments used in the proof of Theorem~\ref{thm:NSElb} (DF1) \emph{Case 2}.  Since $\bj=(0,1)\in \mathscr{F}$, we may again suppose without loss of generality that $\beta^1$ has $\Delta_{\bj \bk\bl}^1(\beta^1)>0$.  Let $\bm=(1,1)$ and recall that $\bm\in \mathscr{F}$.  We suppose without loss of generality that $\beta^1_{a_\bm}\neq 0$.

Let 
\begin{align*}
q'=(a',b'):= \varphi_{\tau_{10}}^{1}(q) \quad \text{ and } \quad q''=(a'',b'') =\varphi^{a_\bj a_\bk a_\bl}_{\tau_{20}} \circ \varphi_{\tau_{10}}^{1}(q), \end{align*}
In this proof, to connect with the notation used in the analysis of~\eqref{eqn:trips3} above, we again offer a convenient abuse of notation and set $H:=H(q')$, $\delta =1/H$, $x= a_\bj'/H$, $y= a_\bk'/H$, $z=a_\bl'/H$.  

First, suppose furthermore that 
\begin{align}
\label{eqn:hyp3}
a_\bj \beta_{a_\bj}^1\bigg( \frac{1}{|\bj|^2}- \frac{1}{|\bk|^2}\bigg) - a_\bl \beta_{a_\bl}^1 \bigg( \frac{1}{|\bk|^2}- \frac{1}{|\bl|^2}\bigg) \geq -  (\zeta_0 + \Delta_{\bj\bk\bl}^1(\beta^1)).   
\end{align}
Observe that on the event 
\begin{align*}
B_{\bj\bk\bl}^1:=\left\{ 1+\frac{4 \zeta_0+ 2\Delta_{\bj\bk\bl}^1(\beta^1)}{ \Delta_{\bj\bk\bl}^1(\beta^1)} \leq \tau_{10} \leq 2 +\frac{4 \zeta_0 +2\Delta_{\bj\bk\bl}^1(\beta^1) }{ \Delta_{\bj\bk\bl}^1(\beta^1)}  \right\},
\end{align*}
for $R>0$ large enough following the calculations starting at~\eqref{prop:case21}, we have 
\begin{align*}
\xi_1\leq \mathcal{E}(x,y,z) \leq 1 \qquad \text{ and } \qquad  \frac{\mathcal{E}(x,y,z)}{|\bk|^2} + \zeta_1 \delta^2 \leq E_{\bj\bk\bl}(x,y,z) \leq \frac{\mathcal{E}(x,y,z)}{|\bj|^2}- \zeta_1. 
\end{align*}
On the other hand
\begin{align*}
|a_{\bm}'| = |a_{\bm}+ \beta^1_{a_{\bm}} \tau_{10}  | \leq \frac{|\beta^1_{a_{\bm}}|}{4} \quad \text{ if and only if } \quad  -\frac{1}{4} - \frac{a_\bm}{\beta^1_{a_\bm}} \leq \tau_{10} \leq \frac{1}{4}- \frac{a_{\bm}}{\beta^1_{a_{\bm}}}.
\end{align*}
In particular, because $|(-1/4+ \alpha, \alpha+1/4)|=1/2$ for any $\alpha \in \R$ where $| \cdot | $ denotes Lebesgue measure, the event
\begin{align*}
\tilde{B}_{\bj\bk\bl}^1:= B_{\bj\bk\bl}^1 \cap \Big\{  \tau_{10} \notin [ -\tfrac{1}{4} - \tfrac{a_\bm}{\beta^1_{a_\bm}},\tfrac{1}{4} - \tfrac{a_\bm}{\beta^1_{a_\bm}} ]\Big\}\end{align*}
satisfies 
\begin{align*}
\PP(\tilde{B}_{\bj\bk\bl}^1)\geq \PP\bigg\{ \frac{3}{2}+\frac{4 \zeta_0+ 2\Delta_{\bj\bk\bl}^1(\beta^1)}{ \Delta_{\bj\bk\bl}^1(\beta^1)} \leq \tau_{10} \leq 2 +\frac{4 \zeta_0 +2\Delta_{\bj\bk\bl}^1(\beta^1) }{ \Delta_{\bj\bk\bl}^1(\beta^1)} \bigg\}>0.
\end{align*}
Hence Assumption~\ref{assump:TNS} (A1) is satisfied for $q' = \varphi_{\tau_{10}}^{1}(q)$ on $\tilde{B}_{\bj\bk\bl}^1$ provided~\eqref{eqn:hyp3} is also satisfied.  Applying Theorem~\ref{thm:therm},      
there exist constants $\eta_{4} \in (0,1/4), c_{4}>0, \delta_{4}\in (0,1)$ depending only on $|\bj|,|\bk|, |\bl|, d, h$ such that $R\geq 1/\delta_{4}$ and $q$ satisfying the above hypotheses implies
\begin{align*}
\PP\Big(|\varphi^{a_\bj a_\bk a_\bl}_{\tau_{20}} \circ \varphi_{\tau_{10}}^{1}(q)\cdot a_\bj|\geq \eta_4 H(q)\Big) \geq   \frac{c_{4}}{|\log H(q)|}.
\end{align*}
Additionally, we have that $|a_\bm'| \geq |\beta^1_{a_\bm}|/4>0$ on $\tilde{B}_{\bj\bk\bl}^1$.  Letting $\bn=(1,0)$ and applying Lemma~\ref{lem:spinners}, there exists $\eta_5, \delta_5, c_5>0$ small enough so that $R\geq 1/\delta_5$ and $q$ satisfying the hypotheses above has
\begin{align*}
&\PP_q(A_0^\text{damp}(D_{\eta_5})) \\
&\geq \PP\Big( \varphi_{\tau_{30}}^{a_\bn a_\bj a_\bm}\circ\varphi^{a_\bj a_\bk a_\bl}_{\tau_{20}} \circ \varphi_{\tau_{10}}^{1}(q) \in D_{\eta_5}, u_{10}=1, u_{20}=a_\bj a_\bk a_\bl, u_{30}=a_\bn a_\bj a_\bm\Big)\\
&\geq \frac{c_{5}}{|\log H(q)|}\frac{1}{|S| (|S|-1)(|S|-2)}.
\end{align*}

On the other hand, suppose that 
\begin{align}
\label{eqn:hyp4}
&\frac{\mathcal{E}(a_\bj, a_\bk, a_\bl)}{|\bk|^2}- \zeta_0 \leq E_{\bj\bk\bl}(a_\bj, a_\bk, a_\bl) \leq \frac{\mathcal{E}(a_\bj, a_\bk, a_\bl)}{|\bk|^2}+ \zeta_0  \qquad \text{ and }\\
\nonumber  &\qquad \qquad \qquad a_\bj \beta_{a_\bj}^1\bigg( \frac{1}{|\bj|^2}- \frac{1}{|\bk|^2}\bigg) - a_\bl \beta_{a_\bl}^1 \bigg( \frac{1}{|\bk|^2}- \frac{1}{|\bl|^2}\bigg)  \leq -(\zeta_0 + \Delta_{\bj\bk\bl}^1(\beta^1)).  
\end{align}
Recalling the event
\begin{align}
B_{\bj\bk\bl}^2=\bigg\{ 1 \leq \tau_{10} \leq 2 \bigg\},
\end{align}
observe that on $B_{\bj\bk\bl}^2$, following the calculations starting at~\eqref{prop:case22} we have 
\begin{align*}
\xi_1\leq \mathcal{E}(x,y,z) \leq 1 \qquad \text{ and } \qquad  \frac{\mathcal{E}(x,y,z)}{|\bk|^2} + \zeta_1  \leq E_{\bj\bk\bl}(x,y,z) \leq \frac{\mathcal{E}(x,y,z)}{|\bj|^2}- \zeta_1 \delta^2. 
\end{align*}
On the other hand
\begin{align*}
|a_{\bm}'| = |a_{\bm} + \beta^1_{a_{\bm}} \tau_{10}  | \leq \frac{|\beta^1_{a_{\bm}}|}{4} \quad \text{ if and only if } \quad  -\frac{1}{4} - \frac{a_\bm}{\beta^1_{a_\bm}} \leq \tau_{10} \leq \frac{1}{4}- \frac{a_\bm}{\beta^1_{a_{\bm}}}.  
\end{align*}
In particular, because $|(-1/4+ \alpha, \alpha+1/4)|=1/2$ for any $\alpha \in \R$, the event
\begin{align*}
\tilde{B}_{\bj\bk\bl}^2:= B_{\bj\bk\bl}^2 \cap \Big\{  \tau_{10} \notin [ -\tfrac{1}{4} - \tfrac{a_\bm}{\beta^1_{a_\bm}},\tfrac{1}{4} - \tfrac{a_\bm}{\beta^1_{a_\bm}} ]\Big\}\end{align*}
satisfies 
\begin{align*}
 \PP(\tilde{B}_{\bj\bk\bl}^2)\geq \PP\bigg\{ \frac{3}{2} \leq \tau_{10} \leq 2  \bigg\}>0.
\end{align*}
By the above, we note that Assumption~\ref{assump:TNS} (A2) is satisfied for $q' = \varphi_{\tau_{10}}^{1}(q)$ on $\tilde{B}_{\bj\bk\bl}^2$ provided~\eqref{eqn:hyp4} is also satisfied.  Applying Theorem~\ref{thm:therm},      
there exist constants $\eta_{6} \in (0,1/4), c_{6}>0, \delta_{6}\in (0,1)$ depending only on $|\bj|,|\bk|, |\bl|,d, h$ such that $R \geq 1/\delta_{6}$ implies
\begin{align*}
\PP\Big( |\varphi^{a_\bj a_\bk a_\bl}_{\tau_{20}} \circ \varphi_{\tau_{10}}^{1}(q)\cdot a_\bj|\geq \eta_6 H(q)\Big) \geq  \frac{c_{6}}{|\log H(q)|}.
\end{align*}
Additionally, we have that $|a_\bm'| \geq |\beta^1_{a_\bm}|/4>0$ on $\tilde{B}_{\bj\bk\bl}^2$.  Letting $\bn=(1,0)$ and applying Lemma~\ref{lem:spinners}, there exists $\eta_7, \delta_7, c_7>0$ small enough so that $R\geq 1/\delta_7$ and $q$ satisfying the above has
\begin{align*}
\PP\Big( \varphi_{\tau_{30}}^{a_\bn a_\bj a_\bm}\circ\varphi^{a_\bj a_\bk a_\bl}_{\tau_{20}} \circ \varphi_{\tau_{10}}^{1}(q) \in D_{\eta_7}\Big)  \geq  \frac{c_{7}}{|\log H(q)|}
\end{align*}
so that 
\begin{align*}
&\PP_q(A_0^\text{damp}(D_{\eta_7})) \\
&\geq \PP\Big( \varphi_{\tau_{30}}^{a_\bn a_\bj a_\bm}\circ\varphi^{a_\bj a_\bk a_\bl}_{\tau_{20}} \circ \varphi_{\tau_{10}}^{1}(q) \in D_{\eta_7}, u_{10}=1, u_{20}= a_\bj a_\bk a_\bl, u_{30}=a_\bn a_\bj a_\bm\Big) \\
& \geq \frac{c_{7}}{|\log H(q)|}\frac{1}{|S| (|S|-1)(|S|-2)}. 
\end{align*}

Putting the above arguments together, we have shown that we can choose $R>3$ large enough, $c_*, \eta_*>0$ small enough so that for $q\in H_{R+1}$ we have
\begin{align*}
\PP_q(A_0^\text{damp}(D_{\eta_*}))\geq \frac{c_*}{\log(H(q))}.
\end{align*}
This concludes the proof. 

\end{proof}

\section*{Appendix: Derivation of the Galerkin splittings from the Euler equation}  

In this section, we derive the splitting studied in Section~\ref{sec:NSE} starting from the two-dimensional forced and damped Euler equation on the torus $\mathbf{T}^2=[0,2\pi]^2$.  
 Most of the derivation is done in~\cite{AMM_23}, but we provide the details for completeness.

Recall that in vorticity formulation, the damped and forced two-dimensional Euler equation on $\mathbf{T}^2$ reads 
\begin{align}
\label{eqn:vort}
\begin{cases}
\partial_t q + (\mathcal{K} q\cdot \nabla) q +\Lambda q = \beta,\\
\text{div}(q) = 0
\end{cases}
\end{align}
where $q$ is a scalar quantity called the \emph{vorticity}, $\mathcal{K} = \nabla^\perp (-\Delta)^{-1}$ is the \emph{Biot-Savart} operator, and $\nabla^\perp:=(\partial_2, -\partial_1)$.  Although we will describe $\Lambda, \beta$ more precisely below, the linear operator $\Lambda$ should be thought of as a partial Laplacian while $\beta$ is a background forcing term acting on select frequencies.   

Assuming there is no mean flow; that is,
\begin{align}
\int_{\mathbf{T}^2} q(t, x) \, dx =0 \,\,\, \text{ for all } t\geq 0, 
\end{align}
we express the solution $q$ of~\eqref{eqn:vort} in Fourier space as 
\begin{align}
q(t,x) = \sum_{\bj \in \mathbf{Z}^2_{\neq 0}} q_\bj(t) e_\bj(x)
\end{align}
where $\{ e_\bj \}_{\bj\in \mathbf{Z}_{\neq 0}^2}$ is the orthonormal family on $L^2(\mathbf{T}^2; \R)$ given by $$e_\bj(x) := \frac{1}{2\pi} \exp( i x \cdot \bj).$$  Let $\mathscr{D}\subset \Z_{\neq 0}^2$. We suppose that 
\begin{align}
\Lambda q := \sum_{\bj \in \mathscr{D}} \lambda_\bj q_\bj e_\bj  \quad \text{ and }\quad  \beta= \sum_{\bj \in \Z_{\neq 0}^2} (\beta_{a_\bj} + i \beta_{b_\bj} )e_\bj 
\end{align}  
where $\lambda_\bj$, $\bj\in \mathbf{Z}_{\neq 0}^2$ and $\beta_{a_\bj}, \beta_{b_\bj} \in \R$.

Plugging this information into~\eqref{eqn:vort} we arrive at the following equation 
\begin{align}
\label{eqn:eulerF}
\dot{q} = - \sum_{\bj+\bk+\bl =0}  \theta_{\bk\bl} \bar{q}_\bk \bar{q}_\bl e_\bj - \sum_{\bj\in \mathscr{D}} \lambda_\bj q_\bj e_\bj  + \sum_{\bj \in \Z_{\neq 0}^2} (\beta_{a_\bj}+ i \beta_{b_\bj})e_\bj
\end{align}
where, in the sum above, $\bj,\bk, \bl \in \mathbf{Z}_{\neq 0}^2$ and the coefficients $\theta_{\bk \bl}$ in~\eqref{eqn:eulerF} are as in~\eqref{eqn:coef}.   
%
% To the equation~\eqref{eqn:eulerF}, we add damping and forcing in the following way
%\begin{align}
%\label{eqn:eulerfd}
%\dot{q}= - \sum_{j+k+\bl =0}  C_{k\bl} \bar{q}_k \bar{q}_\bl e_j - \sum_{i\in \mathcal{D}} \lambda_i q_i e_i  + \sum_{i\in \mathcal{F}} \sigma_i e_i
%\end{align}
%where $\mathcal{D}$ and $\mathcal{F}$ are subsets of $\mathbf{Z}_{\neq 0}^2$, the $\lambda_i $, $i\in \mathbf{Z}_{\neq 0}^2$, are positive constants, and $\sigma_i \in \mathbf{C}$.  We will assume that the forced constants $\sigma_i= \alpha_i + i \gamma_i$ are such that $\alpha_i, \gamma_i \in \R_{\neq 0}$ for all $i\in \mathcal{F}$.  Further assumptions on the sets $\mathcal{D}$ and $\mathcal{F}$ will be made momentarily.  

%In what follows, we consider Galerkin projections of relation~\eqref{eqn:eulerfd} to the finite lattice
%\begin{align}
%\Pi_N\mathbf{Z}_{\neq 0}^2:= \{ j=(j_1, j_2) \in \mathbf{Z}_{\neq 0}^2 \, : \, |j_1| \leq N, \, |j_2| \leq N\}.  
%\end{align}       
%In particular, the (projected) equation of interest reads 
%\begin{align}
%\label{eqn:eulerfd}
%\dot{q}= - \sum_{\substack{j+k+\bl =0\\j,k, \bl \in \Pi_N \mathbf{Z}_{\neq 0}^2}}  C_{k\bl} \bar{q}_k \bar{q}_\bl e_j - \sum_{i\in \Pi_N\mathcal{D}} \lambda_i q_i e_i  + \sum_{i\in \Pi_N \mathcal{F}} \sigma_i e_i
%\end{align}
%where in~\eqref{eqn:eulerfd} 
%\begin{align}
%q=\sum_{j \in\Pi_N \mathbf{Z}_{\neq 0}^2} q_j e_j,  \qquad  \Pi_N\mathcal{D}:= \mathcal{D} \cap \Pi_N\mathbf{Z}_{ \neq 0}^2 , \qquad \Pi_N \mathcal{F}:= \mathcal{F} \cap \Pi_N\mathbf{Z}_{\neq 0}^2.  
%\end{align}

Letting $q_\bj=a_\bj + i b_\bj$, $\bj\in  \mathbf{Z}^2_{\neq 0}$, and using redundancies created by the reality condition $\overline{q}= q$, we can express~\eqref{eqn:eulerF} as the family of equations
\begin{align}
\label{eqn:rimproj}
\begin{cases}
\dot{a}_\bj =\displaystyle{ \sum_{\bj+\bk-\bl =0} \theta_{\bk\bl}( a_\bk a_\bl + b_\bk b_\bl) + \sum_{\bj-\bk-\bl=0} \theta_{\bk\bl} (b_{\bk} b_\bl - a_\bk a_\bl)- \mathbf{1}_{\mathscr{D}}(\bj) \lambda_\bj a_\bj + \beta_{a_\bj} } \\
\dot{b}_\bj = \displaystyle{\sum_{\bj+\bk-\bl=0} \theta_{\bk\bl}(a_\bk b_\bl -b_\bk a_\bl) - \sum_{\bj-\bk-\bl=0} \theta_{\bk\bl}(a_\bk b_\bl + b_{\bk} a_\bl) - \mathbf{1}_{\mathscr{D}}(j) \lambda_\bj b_\bj + \beta_{b_\bj}}
\end{cases}
\end{align}
where all indices $\bj$, $\bk$, $\bl$ belong to the  first quadrant 
\begin{align}
\mathbf{Z}_+^2 = \{ \bj=(j_1, j_2) \in \mathbf{Z}^2_{\neq 0} \, : \, j_2 >0  \} \cup \{ \bj=(j_1, j_2)\in  \mathbf{Z}_{\neq 0}^2 \, : \, j_2 =0 \text{ and } j_1 >0  \}.
\end{align}

Ignoring the dissipative and forced terms in equation~\eqref{eqn:rimproj}, we note that for any triple $\bj,\bk, \bl \in \mathbf{Z}^2_{+}$ with $\bj+\bk- \bl =0$ (which is the same as $\bl -\bj-\bk=0$), we can break apart the nonlinear term in~\eqref{eqn:rimproj} that has only the these indices into four groups of three equations given by: 
\begin{align}
\label{eqn:splittingE}
\begin{cases}
\dot{a}_\bj = \theta_{\bk\bl} a_\bk a_\bl\\
\dot{a}_\bk = \theta_{\bj\bl} a_\bj a_\bl \\
\dot{a}_\bl = - \theta_{\bj\bk} a_\bj a_\bk 
\end{cases}, \,\,\, \begin{cases}
\dot{a}_\bj = \theta_{\bk\bl} b_\bk b_\bl \\
\dot{b}_\bk = \theta_{\bj\bl} a_\bj b_\bl \\
\dot{b}_\bl = -\theta_{\bj\bk} a_\bj b_\bk 
\end{cases}, \,\,\, \begin{cases}
\dot{b}_\bj = \theta_{\bk\bl} a_\bk b_\bl \\
\dot{a}_\bk = \theta_{\bj\bl} b_\bj b_\bl \\
\dot{b}_\bl= -\theta_{\bj\bk} b_\bj a_\bk 
\end{cases},\,\,\,
\begin{cases}
\dot{b}_\bj = -\theta_{\bk\bl} b_\bk a_\bl \\
\dot{b}_\bk =- \theta_{\bj\bl} b_\bj a_\bl \\
\dot{a}_\bl= \theta_{\bj\bk} b_\bj b_\bk.   
\end{cases}
\end{align} 
Importantly, each of the systems above conserves the \emph{relative enstrophy} and \emph{relative energy}; that is, for example,  for the first equation in~\eqref{eqn:splittingE}, the \emph{relative enstrophy} is 
\begin{align}
\mathcal{E}(a_\bj, a_\bk, a_\bl):=a_\bj^2+ a_\bk^2+a_\bl^2, 
\end{align}
which is conserved, and the \emph{relative energy} is 
\begin{align}
E_{\bj\bk\bl}(a_\bj, a_\bk, a_\bl):=\frac{a_\bj^2}{|\bj|^2}+ \frac{a_\bk^2}{|\bk|^2}+ \frac{a_\bl^2}{|\bl|^2}, 
\end{align}
which is also conserved.  We let $V_{a_\bj a_\bk a_\bl}$, $V_{a_\bj b_\bk b_\bl}$,  $V_{b_\bj a_\bk b_\bl} $ and $V_{b_\bj b_\bk a_\bl}$ denote the vector fields corresponding to the systems in~\eqref{eqn:splittingE} from left to right, respectively.  That is, for example, $V_{a_\bj b_\bk b_\bl}$ has $a_\bj$th coordinate given by $\theta_{\bk\bl}  b_\bk b_\bl$, $b_\bk$th coordinate given by $\theta_{\bj \bl} a_\bj b_\bl$, $b_\bl$th coordinate given by $-\theta_{\bj\bk} a_\bj b_\bk$ with all other coordinates (in $\R^\infty$) equal to $0$.

Let $V_\text{damp}$ denote the vector field (on $\R^\infty$) such that for every $\bj\in  \Z_{+}^2$, the $a_\bj$th entry is given by $-\textbf{1}_{\mathscr{D}}(\bj)\lambda_\bj a_\bj $ and $b_\bj$th entry is given by $-\mathbf{1}_{\mathscr{D}}(\bj) \lambda_\bj b_\bj$.  Fixing $\mathfrak{m} \in \N$, for any $\bj\in \Z_{+}^2$ we let 
\begin{align*}
\beta_{a_\bj}&= \sum_{\ell=1}^\mathfrak{m} \beta_{a_\bj}^\ell \quad \text{ and } \quad \beta_{b_\bj}= \sum_{\ell=1}^\mathfrak{m} \beta_{b_\bj}^\ell
\end{align*} 
where $\beta_{a_\bj}^\ell, \beta_{b_\bj}^\ell \in \R$.  For any $\ell=1,2,\ldots, \mathfrak{m}$, we let $V_{\ell} $ denote the vector field (on $\R^\infty$) such that for every $\bj \in \Z_{+}^2$, the $a_\bj$the entry is $\beta_{a_\bj}^\ell$ and the $b_\bj$th entry is $\beta_{b_\bj}^\ell$.

Consider now the collection of vector fields 
\begin{align}
\label{eqn:splitE}
\mathscr{S}:= \{V_{\text{damp}}, V_{1}, \ldots, V_{\mathfrak{m}}\} \cup \bigcup_{\substack{\bj+\bk-\bl=0\\\bj,\bk, \bl \in \Z_+^2\\\langle \bj, \bk^\perp\rangle\neq 0}} \{ V_{a_\bj a_\bk a_\bl}, V_{a_\bj b_\bk b_\bl}, V_{b_\bj a_\bk b_\bl}, V_{b_\bj b_\bk a_\bl} \} 
\end{align}
Note that, if $\langle \bj,\bk^\perp\rangle=0$, then the vector fields $V_{a_\bj a_\bk a_\bl}, V_{a_\bj, b_\bk b_\bl}, V_{b_\bj, a_\bk b_\bl}, V_{b_\bj b_\bk a_\bl}$ with $\bj+\bk=\bl$ are identically zero, hence the restriction above in~\eqref{eqn:splitE}.  

The focus in Section~\ref{sec:NSE} is on a finite-dimensional version of the above set $\mathscr{S}$ where we project the above vector fields onto indices belonging to the finite lattice
\begin{align}
\ZZ^2 := \{ j=(j_1, j_2) \in \mathbf{Z}_{+}^2 \, : \, |j_1| \leq N, \, |j_2 | \leq N\}
\end{align}
as well as project the vector fields onto the relevant finite-dimensional Euclidean space. 
That is, we still consider the splitting in~\eqref{eqn:splitE}, but the vector fields are now vector fields on $\R^{2 | \ZZ^2|}=\R^{2N(N+2)}=:\R^{d}$ rather than $\R^\infty$ and all indices $\bj,\bk, \bl $ belong to $\ZZ^2$.  In Section~\ref{sec:NSE}, we make the slight abuse of notation and give these vector fields the same names as above.  That is, our sought after splitting is described by the collection of vector fields on $\R^{d}$ as in~\eqref{eqn:NSEsplit}.

\bibliographystyle{plain}
\bibliography{randoms}

\end{document}